\documentclass[12pt,reqno]{amsart}
\usepackage{}
\usepackage{mathrsfs}
\usepackage{amssymb}
\usepackage{amsthm}
\usepackage{mathrsfs}
\usepackage[centertags]{amsmath}
\usepackage{amsfonts}

\usepackage{color}
\usepackage{extarrows}
\usepackage{fullpage}

\usepackage{CJK}
\usepackage[colorlinks=true, pdfstartview=Fitv,linkcolor=blue, citecolor=blue]{hyperref}
\input amssym.def
\input amssym.tex

\date{}

\newtheorem{Theorem}{Theorem}[section]

\newtheorem{Lemma}{Lemma}[section]

\newcommand\R{\mbox{\bf R}}

\newcommand\Z{\mbox{\bf Z}}
\newcommand\z{\mbox{\bf z}}
\newcommand\SR{\mbox{\scriptsize\bf R}}

\newcommand{\definition}{{\lower .5ex
  \hbox{$\>\>\stackrel{\triangle}{=}\>\>$} }}
\newcommand\supp{\mathop{\rm supp}}

\begin{document}

\baselineskip=22pt
\thispagestyle{empty}

\begin{center}
{\Large \bf Strichartz estimates for orthonormal functions and convergence problem
of density functions of Boussinesq operator on manifolds }\\[1ex]

{ Xiangqian Yan\footnote{Email:yanxiangqian213@126.com}$^{a}$,\,
Yongsheng Li\footnote{Email:  yshli@scut.edu.cn}$^{a}$,\,Wei Yan$^{b*}$\footnote{Email:  011133@htu.edu.cn}, \,Xin Liu$^{b*}$}\\[1ex]

{$^a$School of Mathematics,
 South China University of Technology,
 Guangzhou, Guangdong 510640, P. R. China}\\[1ex]

{$^b$School of Mathematics and Information Science, Henan
Normal University,}\\
{Xinxiang, Henan 453007,   China}\\[1ex]

\end{center}
\noindent{\bf Abstract.}
This  paper  is  devoted  to studying the maximal-in-time estimates
 and  Strichartz estimates for orthonormal functions and convergence problem
  of density functions related to Boussinesq operator on manifolds. Firstly,
   we present the pointwise convergence of density function related to Boussinesq
   operator with $\gamma_{0}\in\mathfrak{S}^{\beta}(\dot{H}^{\frac{1}{4}}(\mathbf{R}))(\beta<2)$
    with the aid of the maximal-in-time estimates
  related to Boussinesq operator with orthonormal function on $\R$. Secondly,
  we present the pointwise convergence of density function related to Boussinesq operator with $\gamma_{0}\in\mathfrak{S}^{\beta}(\dot{H}^{s})(\frac{d}{4}\leq s<\frac{d}{2},\,
  0<\alpha\leq d, 1\leq\beta<\frac{\alpha}{d-2s})$ with the aid of the maximal-in-time estimates
  related to Boussinesq operator with orthonormal function on the unit ball
  $\mathbf{B}^{d}(d\geq1)$ established in this paper; we also present the Hausdorff
  dimension of the divergence set of density function related to Boussinesq
  operator $dim_{H}D(\gamma_{0})\leq (d-2s)\beta$.  Thirdly,  we show the Strichartz
  estimates for orthonormal functions and Schatten bound with space-time norms related
   to Boussinesq operator on $\mathbf{T}$ with the aid of   the noncommutative-commutative
    interpolation theorems established in this paper, which are just Lemmas 3.1-3.4 in this
     paper; we also prove that Theorems 1.5, 1.6 are optimal. Finally, by using full randomization,
      we present the probabilistic convergence of density function related to Boussinesq operator
       on $\R$, $\mathbf{T}$ and $\Theta=\{x\in\R^{3}:|x|<1\}$ with $\gamma_{0}\in\mathfrak{S}^{2}$.

 \bigskip

\noindent {\bf Keywords}: Boussinesq operator;  Maximal estimates for orthonormal
 functions; Strichartz estimates for orthonormal functions; Noncommutative-commutative
  interpolation theorems; Convergence problem of density functions
\medskip

\medskip
\noindent {\bf Corresponding Author:} Wei Yan

\medskip
\noindent {\bf Email Address:}011133@htu.edu.cn

\medskip
\noindent {\bf MSC2020-Mathematics Subject Classification}: Primary-35Q41, 35B45, 42B20; Secondary-35B65, 42B37

\leftskip 0 true cm \rightskip 0 true cm

\newpage

\baselineskip=20pt

\bigskip
\bigskip
\tableofcontents

\section{Introduction and the main results }
\bigskip

\setcounter{Theorem}{0} \setcounter{Lemma}{0}\setcounter{Definition}{0}\setcounter{Proposition}{1}

\setcounter{section}{1}
\subsection {Model and synopsis of result}
In this paper,  we  investigate  the infinite system of Boussinesq type equation  in the following
\begin{eqnarray}
&&i\frac{du_{j}}{dt}=-\sqrt{\partial_{x}^{4}-\partial_{x}^{2}}u_{j},
\,j\in \mathbb{N}^{+}, x\in\R,\,t\in I\subset\R,\label{1.01}\\
&&u_{j}(0,x)=f_{j}(x),\label{1.02}
\end{eqnarray}
where $(f_j)_j$ is  an orthonormal system in $L^{2}(\mathbf{R})$. By using elementary
 calculation, we get the operator valued form of \eqref{1.01} as follows
\begin{eqnarray}
&&i\frac{d\gamma(t)}{dt}=\left[-\sqrt{\partial_{x}^{4}-\partial_{x}^{2}}, \gamma\right], \label{1.03}\\
&&\gamma(0)=\gamma_{0},\label{1.04}
\end{eqnarray}
When  $\gamma_{0}=\sum\limits_{j=1}^{\infty}\lambda_{j}|f_{j}\rangle\langle f_{j}|$,
obviously,
\begin{eqnarray}
&\gamma(t):=e^{it\sqrt{\partial_{x}^{4}-\partial_{x}^{2}}}\gamma_{0}
e^{-it\sqrt{\partial_{x}^{4}-\partial_{x}^{2}}}=\sum\limits_{j=1}^{\infty}\lambda_{j}
\left|e^{it\sqrt{\partial_{x}^{4}-\partial_{x}^{2}}}f_{j}\right\rangle\left\langle e^{it\sqrt{\partial_{x}^{4}-\partial_{x}^{2}}}f_{j}\right|,\label{1.05}
\end{eqnarray}
is the solution to (\ref{1.03})-(\ref{1.04}).
Here $|u\rangle\langle v| f\rightarrow\langle v,f\rangle u=u\int_{I}\overline{v}f dx$,
 $I\subset\R$, $\lambda_{j}$ is the eigenvalue of the operator $\gamma_{0}$. By using a
  elementary calculation, we derive the kernel of $\gamma(t)$ as follows:
\begin{eqnarray}
&K(x,y,t)=\sum\limits_{j=1}^{\infty}\lambda_{j}
\left(e^{it\sqrt{\partial_{x}^{4}-\partial_{x}^{2}}}f_{j}\right)(t,x)\overline
{\left(e^{it\sqrt{\partial_{x}^{4}-\partial_{x}^{2}}}f_{j}\right)}(t,y).\label{1.06}
\end{eqnarray}
It is follows from \eqref{1.06} that the density function $\rho_{\gamma(t)}
=K(x,x,t)=\sum\limits_{j=1}^{\infty}\lambda_{j}
\left|e^{it\sqrt{\partial_{x}^{4}-\partial_{x}^{2}}}f_{j}\right|^{2}$, thus, we have
\begin{align}
Trace(\gamma(t))=\int_{\SR}K(x,x,t)dx&=
\sum\limits_{j=1}^{\infty}\lambda_{j}\int_{\SR}
\left|e^{it\sqrt{\partial_{x}^{4}-\partial_{x}^{2}}}f_{j}\right|^{2}dx\nonumber\\
&=\sum\limits_{j=1}^{\infty}\lambda_{j}\|f_{j}\|_{L^{2}}^{2}
=\sum\limits_{j=1}^{\infty}\lambda_{j}.\label{1.07}
\end{align}

The study of generalizing some classical inequalities of single functions to
systems of orthogonal functions goes back to Lieb-Thirring \cite{LT1975}.
Frank et al. \cite{FLLS2014} established the following Strichartz estimates for
 orthonormal functions
\begin{eqnarray}
&\left\|\sum\limits_{j=1}^{+\infty}\lambda_{j}
\left|e^{it\Delta}f_{j}\right|^{2}\right\|_{L_{t}^{p}L_{x}^{q}(\SR\times\SR^{d})}\leq C\left(\sum\limits_{j=1}^{\infty}|\lambda_{j}|^{\beta}\right)^{\frac{1}{\beta}},\,
1\leq \beta\leq \frac{2q}{q+1},\,q\leq1+\frac{2}{d},d\geq1,\label{1.08}
\end{eqnarray}
which generalized Strichartz estimates of \cite{C2003,KPV1991}
\begin{eqnarray}
&\left\|e^{it\Delta}f\right\|_{L_{t}^{p}L_{x}^{q}(\SR\times\SR^{d})}\leq C\|f\|_{L^{2}}, p,q\geq2,\, \frac{2}{p}+\frac{d}{q}=\frac{d}{2},\,(d,p,q)\neq(2,\infty,2).\label{1.09}
\end{eqnarray}
\eqref{1.08} is used to study the nonlinear evolution and scattering of quantum systems
 with an infinite number of particles\cite{CHP2017,CHP2018,LS2014,LS2015}.
 Frank and Sabin \cite{FS2017} proved that \eqref{1.08} holds except
  the endpoints. Moreover, by  using  the duality principle related to Schatten norms and
    the properties of Dirac function,  they
  presented the Schatten bound with space-time norms.
   For more research results about \eqref{1.08} on $\R\times\R^{d}$, we can refer the readers
    to \cite{BHLNS2019,BLN2020,BLS2021,FS2017}. For $(t, x)\in\mathbf{T}\times\mathbf{T}^{d}$ and
    $d\geq1$, by using the duality principle related to Schatten norms and the idea
     of Vega \cite{V1990},    Nakamura \cite{N2020} proved
\begin{eqnarray}
&\left\|\sum\limits_{j=1}^{+\infty}\lambda_{j}\left|e^{it\Delta}P_{\leq N}f_{j}\right|^{2}\right\|_{L_{t}^{p}L_{x}^{q}(\mathbf{T}^{d}\times\mathbf{T})}\leq CN^{\frac{1}{p}}\left(\sum\limits_{j=1}^{+\infty}|\lambda_{j}|^{\beta}\right)^{\frac{1}{\beta}},
\,(\frac{1}{q},\frac{1}{p})\in (A,B]\label{1.010}
\end{eqnarray}
 where $\beta\leq\frac{2q}{q+1}$, $A=\left(\frac{d-1}{d+1},\frac{d}{d+1}\right)$,
$B=(1,0)$ and $(A,B]$ represents the open line segment combining $A, B$. Moreover,
for $(t, x)\in\mathbf{T}\times\mathbf{T}$ and  $\beta\leq 2$, Nakamura \cite{N2020} established the
 following estimates,
\begin{eqnarray}
&\left\|\sum\limits_{j=1}^{+\infty}\lambda_{j}\left|e^{it\Delta}P_{\leq N}f_{j}\right|^{2}\right\|_{L_{t}^{2}L_{x}^{\infty}(\mathbf{T}\times\mathbf{T})}\leq CN^{\frac{1}{2}}\left(\sum\limits_{j=1}^{+\infty}|\lambda_{j}|^{\beta}\right)^{\frac{1}{\beta}}\label{1.011}.
\end{eqnarray}
Nakamura \cite{N2020} also constructed the counterexample to show that
   $\beta\leq\frac{2q}{q+1}$ is necessary condition for  \eqref{1.010} and $\beta\leq2$
    is necessary condition for  \eqref{1.011}.

The pointwise convergence problem was initiated by Carleson \cite{C1979}
 and was further investigated by some authors in higher dimensions
 \cite{B1995,B2016,CLV2012,C1982,DK1982,DG,D2017,DGL2017,DZ2019,DGLZ2018,L2006,
 LR,LR2017,MV2008,S,V1988,Z2014}. Recently, by using randomized initial data method introduced
  by Lebowitz-Rose-Speer  \cite{LRS1988} in statistical physics
  and Bourgain  \cite{B1994,B1996} in mathematics and developed by Burq-Tzvetkov
  \cite{BT2007,BT2008I,BT2008II} with the aid of Khinchine type inequalities,
  Compaan et al. \cite{CLS2021} investigated the  pointwise convergence
   of the Schr\"{o}dinger flow with random data. Moreover, by using random data
    technique, Yan et al. \cite{YZY2022} and Wang et al. \cite{WYY2021}  investigated the pointwise
 convergence problem of Schr\"odinger equation on $\mathbf{R}^{d}$, $\mathbf{T}^{d}$
  and $\Theta=\{x:|x|<1\}$, respectively.
By using  Frostman's  lemma which can be found in Theorem 8.8, page 112 of
\cite{M1995}),    Barcel$\acute{o}$ et al.\cite{BBCR} gave the Hausdorff dimension
 of divergence sets of dispersive
equations. Recently,  Li and Li \cite{LL2023} studied the pointwise convergence
 and the Hausdorff dimension of the divergence set
   of the Boussinesq operator on $\R$.

For the orthogonal system $(f_j)_j$, Bez et al. \cite{BLN2020} studied the pointwise
 convergence  problem of the density function of the following operator equation
\begin{eqnarray*}
&&i\frac{d\gamma(t)}{dt}=\left[-\partial_{x}^{2}, \gamma\right], \\
&&\gamma(0)=\gamma_{0}=\sum\limits_{j=1}^{+\infty}\lambda_{j}|f_{j}\rangle\langle f_{j}|
\end{eqnarray*}
Firstly, by using the duality principle related to Schatten norms and interpolation
 techniques,
which are Lemmas 2.15-2.17 in this paper, Bez et al. \cite{BLN2020} established
the following maximal-in-time estimates
related to Schr\"odinger operator with orthonormal function on $\R$
\begin{eqnarray}
&\left\|\sum\limits_{j}\lambda_{j}\left|e^{it\Delta}f_{j}\right|^{2}
\right\|_{L_{x}^{2,\infty}L_{t}^{\infty}(\mathbf{R}\times I)}\leq C
\|\lambda\|_{\ell^{\beta}},\label{1.012}
\end{eqnarray}
where $\beta<2$, $\lambda=(\lambda_{j})_{j}\in \ell^{\beta}$
and $(f_j)_j$ is  an orthonormal
 system in $\dot{H}^{\frac{1}{4}}(\mathbf{R})$.
By using \eqref{1.012}, Bez et al. \cite{BLN2020} proved that
\begin{eqnarray}
&\lim\limits_{t\longrightarrow 0}\rho_{\gamma(t)}(x)=
\rho_{\gamma_{0}}(x),\,\,a.e.x\in\mathbf{R},\label{1.013}
\end{eqnarray}
where $\gamma(t)=\sum\limits_{j=1}^{+\infty}\lambda_{j}|
e^{it\partial_{x}^{2}}f_{j}\rangle\langle e^{it\partial_{x}^{2}}f_{j}|$ ,
$\rho_{\gamma(t)},\, \rho_{\gamma_{0}}$ denote the density function of
  $\gamma(t), \,\gamma_{0}$, respectively. Recently, by using the maximal
function estimates for orthonormal functions related to $\alpha$ dimension
 Borel measure supported on unit ball,
Bez et al. \cite{BKS2023} studied the
    pointwise convergence of density function $\sum\limits_{j=1}^{\infty}\lambda_{j}
    |e^{it(-\Delta)^{m/2}}f_{j}|^{2}$ related to
   operator $\sum\limits_{j=1}^{\infty}\lambda_{j}|e^{it(-\Delta)^{m/2}}
   f_{j}\rangle \langle e^{it(-\Delta)^{m/2}}f_{j}|$ with the aid of $\alpha$
   dimension Borel measure
    and presented the Hausdorff dimension of the divergence set
related to density function $\sum\limits_{j=1}^{\infty}\lambda_{j}|e^{it(-\Delta)^{m/2}}f_{j}|^{2}$  of
   operator $\sum\limits_{j=1}^{\infty}\lambda_{j}|e^{it(-\Delta)^{m/2}}f_{j}\rangle
   \langle e^{it(-\Delta)^{m/2}}f_{j}| $ with the aid of Frostman's  Lemma.

\subsection{Motivation}The motivation of this paper is
fourfold.
Firstly, Bez et al. \cite{BLN2020}    employed  an idea of \cite{KPV1991} that has the
 effect of switching the roles of space and time at the cost of replacing the classical
 Schr\"odinger propagator with the fractional Schr\"odinger propagator of order $1/2$
to establish the maximal-in-time estimates;  by using the maximal-in-time estimates,
Bez et al. \cite{BLN2020} proved that
 \begin{eqnarray*}
 \lim\limits_{t\rightarrow0}\sum\limits_{j=1}^{+\infty}\lambda_{j}|e^{it\partial_{x}^{2}}f_{j}|^{2}
 =\sum\limits_{j=1}^{+\infty}\lambda_{j}|f_{j}|^{2}, a.e. x\in \R.
 \end{eqnarray*}
 Here  $\sum\limits_{j=1}^{+\infty}\lambda_{j}|e^{it\partial_{x}^{2}}f_{j}|^{2}$
 and $\sum\limits_{j=1}^{+\infty}\lambda_{j}|f_{j}|^{2}$
 are density functions of
 $\sum\limits_{j=1}^{+\infty}\lambda_{j}|e^{it\partial_{x}^{2}}f_{j}\rangle
  \langle e^{it\partial_{x}^{2}}f_{j}|$
and $\sum\limits_{j=1}^{+\infty}\lambda_{j}|f_{j}\rangle \langle f_{j}|$, respectively.
Thus, a natural question appears:  is it possible to find
a slightly different method that is used to establish
maximal-in-time estimates related to $e^{it\sqrt{\partial_{x}^{4}-\partial_{x}^{2}}}$? Moreover,
if the maximal-in-time estimates are established, is
 \begin{eqnarray*}
 \lim\limits_{t\rightarrow0}\sum\limits_{j=1}^{+\infty}\lambda_{j}|e^{it\sqrt{\partial_{x}^{4}-\partial_{x}^{2}}}f_{j}|^{2}
 =\sum\limits_{j=1}^{+\infty}\lambda_{j}|f_{j}|^{2},a.e. x\in \R
 \end{eqnarray*}
 valid?
 Here  $\sum\limits_{j=1}^{+\infty}\lambda_{j}|e^{it\sqrt{\partial_{x}^{4}-\partial_{x}^{2}}}f_{j}|^{2}$
 and $\sum\limits_{j=1}^{+\infty}\lambda_{j}|f_{j}|^{2}$
 denote the density functions of
 $$\sum\limits_{j=1}^{+\infty}\lambda_{j}|e^{it\sqrt{\partial_{x}^{4}-\partial_{x}^{2}}}f_{j}\rangle \langle e^{it\sqrt{\partial_{x}^{4}-\partial_{x}^{2}}}f_{j}|,\sum\limits_{j=1}^{+\infty}
 \lambda_{j}|f_{j}\rangle \langle f_{j}|,$$ respectively.
Secondly,  Bez et al. \cite{BKS2023} established the Strichartz estimates
for orthonormal functions and pointwise convergence of density function
$\sum\limits_{j=1}^{+\infty}\lambda_{j}|e^{-it(-\Delta)^{m/2}}f_{j}|^{2}$ of
operator $\sum\limits_{j=1}^{+\infty}\lambda_{j}|e^{-it(-\Delta)^{m/2}}f_{j}\rangle
\langle e^{-it(-\Delta)^{m/2}}f_{j}|(\mu-a.e. x\in\mathbf{B^{d}})$, where $\mu$
 is $\alpha$ dimension Borel measure; moreover,
Bez et al. \cite{BKS2023} presented the upper bound of Hausdorff dimension of
$D(\sum\limits_{j=1}^{+\infty}\lambda_{j}|f_{j}|^{2})=\{x\in \mathbf{B}^{d}:\lim\limits_{t\longrightarrow 0}
\sum\limits_{j=1}^{+\infty}\lambda_{j}|e^{-it(-\Delta)^{m/2}}f_{j}|^{2}
\neq\sum\limits_{j=1}^{+\infty}\lambda_{j}|f_{j}|^{2}\}$.
Thus, a natural question appears:  is it possible to  establish
maximal-in-time estimates for  orthonormal functions  related to $e^{it\sqrt{\partial_{x}^{4}-\partial_{x}^{2}}}$? Moreover,
if the maximal-in-time estimates are established, is
 \begin{eqnarray*}
 \lim\limits_{t\rightarrow0}\sum\limits_{j=1}^{+\infty}\lambda_{j}|e^{it\sqrt{\partial_{x}^{4}-\partial_{x}^{2}}}f_{j}|^{2}
 =\sum\limits_{j=1}^{+\infty}\lambda_{j}|f_{j}|^{2},\mu-a.e. x\in \R
 \end{eqnarray*}
 valid? Thirdly, Nakamura \cite{N2020} established the Strichartz estimates for orthonormal functions
  on $\mathbf{T^{d}}$ related to $e^{it \Delta}$ and presented
 the optimality of the validity of the Strichartz estimates for orthonormal functions on $\mathbf{T^{d}}$.
Thus, a natural question appears:  is it possible to  establish
Strichartz estimates  for  orthonormal functions  related to $e^{it\sqrt{\partial_{x}^{4}-\partial_{x}^{2}}}$
 on $\mathbf{T}$ and prove the optimality?
Finally,  by using full randomization of compact operator,
      Yan et al.\cite{YDHXY2024} presented the probabilistic convergence of density function related to $e^{it\Delta}$
       on $\R$, $\mathbf{T}$ and $\Theta=\{x\in\R^{3}:|x|<1\}$ with $\sum\limits_{j=1}^{+\infty}\lambda_{j}|f_{j}\rangle \langle f_{j}|\in\mathfrak{S}^{2}$.
Thus, a natural question appears:  is it possible to  establish
 the probabilistic convergence of density function
 of $ e^{it\sqrt{\partial_{x}^{4}-\partial_{x}^{2}}}$ on $\R$, $\mathbf{T}$ and $\Theta=\{x\in\R^{3}:|x|<1\}$ with $\sum\limits_{j=1}^{+\infty}\lambda_{j}|f_{j}\rangle \langle f_{j}|\in\mathfrak{S}^{2}$.

\subsection{Main contribution}
In this paper we investigate the maximal-in-time estimates and  Strichartz
 estimates for orthonormal functions and convergence problem of density functions
  related to $e^{it\sqrt{\partial_{x}^{4}-\partial_{x}^{2}}}$ on manifolds.

Firstly, we use a approach slightly different  from \cite{BLN2020}
to establish the maximal-in-time estimates
for orthonormal functions  related to $e^{it\sqrt{\partial_{x}^{4}-\partial_{x}^{2}}}$
 on $\R$. More precisely,   for $\beta<2$,
$\lambda=(\lambda_{j})_{j}\in \ell^{\beta}$,  by using homogeneous dyadic
decomposition instead of inhomogeneous dyadic decomposition and a direct proof instead of
switching the roles of space and time, we have
\begin{eqnarray}
&&\Big\|\sum_{j}\lambda_{j}|e^{it\sqrt{-\partial_{x}^{2}+\partial_{x}^{4}}}f_{j}|^{2}
\Big\|_{L_{x}^{2,\infty}L_{t}^{\infty}(\mathbf{R}\times I)}\leq C\|\lambda\|_{\ell^{\beta}}.\label{1.014}
\end{eqnarray}
Here $(f_j)_j$ is
an orthonormal system in $\dot{H}^{\frac{1}{4}}(\mathbf{R})$.
Then, by using the maximal-in-time estimates \eqref{1.014}, we present the pointwise
convergence of density function related to $e^{it\sqrt{\partial_{x}^{4}-\partial_{x}^{2}}}$.  For $\beta<2$ and $\gamma_{0}\in \mathfrak{S}^{\beta}(\dot{H}^{\frac{1}{4}}(\mathbf{R}))$ which is self-adjoint,  we have
\begin{eqnarray}
&\lim\limits_{t\longrightarrow 0}\rho_{\gamma(t)}(x)=\rho_{\gamma_{0}}(x), \,\,a.e.x\in\mathbf{R},\label{1.015}
\end{eqnarray}
where $\rho_{\gamma(t)},\, \rho_{\gamma_{0}}$ denote the density function of
$\gamma(t)=\sum\limits_{j=1}^{\infty}\lambda_{j}
\left|e^{it\sqrt{\partial_{x}^{4}-\partial_{x}^{2}}}f_{j}\right\rangle\left\langle
e^{it\sqrt{\partial_{x}^{4}-\partial_{x}^{2}}}f_{j}\right|$, $\gamma_{0}=
\sum\limits_{j=1}^{\infty}\lambda_{j}|f_{j}\rangle\langle f_{j}|$, respectively.
 \eqref{1.015} generalizes the results of \cite{LL2023}, since if we choose
  $\lambda_{1}=1$ and $\lambda_{j}=0(j\neq 1)$, by using \eqref{1.015}, we have
\begin{eqnarray*}
&\lim\limits_{t\longrightarrow 0}e^{i\sqrt{-\partial_{x}^{2}+\partial_{x}^{4}}}f_{1}=f_{1},\,\,a.e.x\in\mathbf{R}.
\end{eqnarray*}

Secondly, inspired by \cite{BKS2023}, we establish the maximal-in-time estimates for orthonormal functions
  related to $e^{it\sqrt{\partial_{x}^{4}-\partial_{x}^{2}}}$. More precisely,  for $d\geq1$, $\frac{d}{4}\leq s<\frac{d}{2}$,
   $0<\alpha\leq d$, $1\leq\beta<\frac{\alpha}{d-2s}$ and $\mu\in\mathcal{M}^{\alpha}
   (\mathbf{B}^{d})$,  we have
\begin{eqnarray}
&&\left\|\sum\limits_{j}\lambda_{j}|e^{it\sqrt{-\partial_{x}^{2}+
\partial_{x}^{4}}}f_{j}|^{2}\right\|_{L_{x}^{1}(\mathbf{B}^{d},d\mu)L_{t}^{\infty}(0,1)}\leq C\|\lambda\|_{\ell^{\beta}}.\label{1.016}
\end{eqnarray}
Here $\mathcal{M}^{\alpha}$ denotes the collection of all $\alpha$-dimensional
 probability measures supported on the unit ball $\mathbf{B^{d}}$.
Then, by using \eqref{1.016}, we present the pointwise convergence of density function
 related to $e^{it\sqrt{-\partial_{x}^{2}+
\partial_{x}^{4}}}f_{j}$. More precisely,  for $d\geq1$, $\frac{d}{4}\leq s<\frac{d}{2}$,
 $0<\alpha\leq d$ and $\gamma_{0}\in \mathfrak{S}^{\beta}(\dot{H}^{s})$ which is self-adjoint
  with $1\leq\beta<\frac{\alpha}{d-2s}$,  we have
\begin{eqnarray}
&\lim\limits_{t\longrightarrow 0}\rho_{\gamma(t)}(x)=\rho_{\gamma_{0}}(x), \,\,\mu-a.e.x\in \mathbf{B}^{d},\label{1.017}
\end{eqnarray}
 where $\rho_{\gamma(t)}, \rho_{\gamma_{0}}$ denote the density function of
 $\gamma(t)=\sum\limits_{j=1}^{\infty}\lambda_{j}
\left|e^{it\sqrt{\partial_{x}^{4}-\partial_{x}^{2}}}f_{j}\right\rangle\left\langle
e^{it\sqrt{\partial_{x}^{4}-\partial_{x}^{2}}}f_{j}\right|$, $\gamma_{0}=
\sum\limits_{j=1}^{\infty}\lambda_{j}|f_{j}\rangle\langle f_{j}|$,
respectively.
Moreover, we present the Hausdorff dimension of the divergence set of density
function related to $e^{it\sqrt{\partial_{x}^{4}-\partial_{x}^{2}}}f_{j}$. More precisely, we prove
\begin{eqnarray}
&&dim_{H}D(\gamma_{0})\leq (d-2s)\beta,\label{1.018}
\end{eqnarray}
where $D(\gamma_{0})=\{x\in \mathbf{B^{d}}:\lim\limits_{t\longrightarrow 0}\rho_{\gamma(t)}(x)\neq\rho_{\gamma_{0}}(x)\}$, $dim_{H}D(\gamma_{0})$
is the Hausdorff dimension of $D(\gamma_{0})$.

Thirdly, by using techniques of \cite{V1990} and \cite{N2020} and
 noncommutative commutative Stein interpolation theorem related to the mixed Lebesgue spaces, which
  is just Lemma 3.4 in this paper, we show the Schatten bound with space-time norms
   related to $e^{it\sqrt{\partial_{x}^{4}-\partial_{x}^{2}}}$ on $\mathbf{T}$. More precisely, for all $W_{1},W_{2}\in L_{t}^{2p^{\prime}}L_{x}^{2q^{\prime}}(I_{N}\times\mathbf{T})$,
we prove
\begin{eqnarray}
&\|W_{1}\mathscr{D}\mathscr{D}^{\ast}W_{2}\|_{\mathfrak{S}^{2q^{\prime}}(L^{2}(I_{N}\times\mathbf{T}))}\leq C\|W_{1}\|_{L_{t}^{2p^{\prime}}L_{x}^{2q^{\prime}}(I_{N}\times\mathbf{T})}
\|W_{2}\|_{L_{t}^{2p^{\prime}}L_{x}^{2q^{\prime}}(I_{N}\times\mathbf{T})}.\label{1.019}
\end{eqnarray}
Here $I_{N}=[-\frac{1}{2N},\frac{1}{2N}],\,S_{N}=\Z\cap[-N,N]$, $\frac{1}{p^{\prime}}
+\frac{1}{2q^{\prime}}=1$, $\frac{2}{3}<\frac{1}{q^{\prime}}<1$.
Moreover,  by using the duality principle related to Schatten norms and techniques of
\cite{V1990} and \cite{N2020}, we also prove Strichartz estimates for orthonormal functions
 related to $e^{it\sqrt{\partial_{x}^{4}-\partial_{x}^{2}}}$ on $\mathbf{T}$. More precisely,    we prove
\begin{eqnarray}
&\left\|\sum\limits_{j=1}^{\infty}\lambda_{j}\left|e^{it\sqrt{\partial_{x}^{4}-\partial_{x}^{2}}}P_{\leq N}f_{j}\right|^{2}\right\|_{L_{t}^{p}L_{x}^{q}(\mathbf{T}\times\mathbf{T})}\leq CN^{\frac{1}{p}}\|\lambda\|_{\ell^{\beta}},\beta\leq2,\label{1.020}
\end{eqnarray}
and
\begin{eqnarray}
&\left\|\sum\limits_{j=1}^{\infty}\lambda_{j}\left|e^{it\sqrt{\partial_{x}^{4}-\partial_{x}^{2}}}P_{\leq N}f_{j}\right|^{2}\right\|_{L_{t}^{2}L_{x}^{\infty}(\mathbf{T}\times\mathbf{T})}\leq CN^{\frac{1}{2}}\|\lambda\|_{\ell^{\beta}}.\label{1.021}
\end{eqnarray}
Here $N\geq1$, $(\frac{1}{q},\frac{1}{p})\in (A,B],\,A=(0,\frac{1}{2}),\,B=(1,0)$,
 $\lambda=(\lambda_{j})_{j=1}^{+\infty}\in\ell^{\beta}
 (\beta\leq\frac{2q}{q+1})$.
We can construct the counterexample to show that  $\beta\leq\frac{2q}{q+1}$ is necessary
condition for \eqref{1.020} and $\beta\leq2$ is necessary condition for \eqref{1.021}.

Fourthly, by using full randomization of compact operator, we present the probabilistic convergence of density
 function related to $e^{it\sqrt{\partial_{x}^{4}-\partial_{x}^{2}}}$ on $\R$
  with $\gamma_{0}\in\mathfrak{S}^{2}$. More precisely, for any  $\epsilon>0$,
  $\exists0<\delta<\frac{\epsilon^{2}}{M_{1}\sqrt{1+M_{1}^{2}}}$, when $|t|<\delta$, we prove
\begin{eqnarray*}
&&\left\|\sum_{j=1}^{\infty}\lambda_{j}g_{j}^{(2)}(\widetilde{\omega})A_{1}\right\|_{L_{\omega,\widetilde{\omega}}^{r}
(\Omega\times\widetilde{\Omega})}
\leq Cr^{\frac{3}{2}}\epsilon^{2}(\|\gamma_{0}\|_{\mathfrak{S}^{2}}+1)
\end{eqnarray*}
and
\begin{eqnarray*}
&&\lim\limits_{t \rightarrow 0}(\mathbb{P}\times \widetilde{\mathbb{P}})\left(\left\{(\omega \times \widetilde{\omega}) \in(\Omega \times \widetilde{\Omega}) || F(t, \omega, \widetilde{\omega})|>C\alpha_{1}\right\}\right)=0.
\end{eqnarray*}
Here $r \in[2, \infty)$, $A_{1}=\left(|f_{j}^{\omega}|^{2}-|e^{it\sqrt{\partial_{x}^{4}-\partial_{x}^{2}}}f_{j}^{\omega}|^{2}\right)$,
$(f_{j})_{j=1}^{\infty}$ is an orthonormal system in
 $L^{2}\left(\R\right)$ and  $f_{j}^{\omega}$ is the randomization of $f_{j}$  defined as in
 \eqref{1.026}.
\begin{eqnarray*}
&&F(t, \omega, \widetilde{\omega})=\sum\limits_{j=1}^{\infty} \lambda_{j} g_{j}^{(2)}(\widetilde{\omega})A_{1},\alpha_{1}=\left(\left\|\gamma_0\right\|_{\mathfrak{S}^2}+1\right)
e \epsilon^{\frac{1}{2}}\left(\epsilon \ln \frac{1}{\epsilon}\right)^{\frac{3}{2}},\\
&&
\rho_{\gamma_{0}^{\omega, \tilde{\omega}}}=\sum_{j=1}^{\infty} \lambda_j g_{j}^{(2)}(\widetilde{\omega})
\left|f_{j}^{\omega}\right|^2, \\ &&\rho_{e^{it\sqrt{\partial_{x}^{4}-\partial_{x}^{2}}} \gamma_{0}^{\omega, \widetilde{\omega}}e^{-it\sqrt{\partial_{x}^{4}-\partial_{x}^{2}}}}=\sum_{j=1}^{\infty} \lambda_{j} g_{j}^{(2)}(\widetilde{\omega})\left|e^{it\sqrt{\partial_{x}^{4}-\partial_{x}^{2}}} f_{j}^{\omega}\right|^{2},
\end{eqnarray*}
$\rho_{\gamma_{0}^{\omega, \tilde{\omega}}},\rho_{e^{it\sqrt{\partial_{x}^{4}-\partial_{x}^{2}}} \gamma_{0}^{\omega, \widetilde{\omega}}e^{-it\sqrt{\partial_{x}^{4}-\partial_{x}^{2}}}} $ denote the density function
 of $\gamma_{0}^{\omega, \tilde{\omega}},e^{it\sqrt{\partial_{x}^{4}-\partial_{x}^{2}}} \gamma_{0}^{\omega, \widetilde{\omega}}e^{-it\sqrt{\partial_{x}^{4}-\partial_{x}^{2}}},$ respectively.

Fifthly, by using full randomization of compact operator, we present the probabilistic convergence of density
 function related to $e^{it\sqrt{\partial_{x}^{4}-\partial_{x}^{2}}}$ on $\mathbf{T}$. More precisely,  for any  $\epsilon>0$, $\exists0<\delta<\frac{\epsilon^{2}}{M_{1}\sqrt{1+M_{1}^{2}}}$, when $|t|<\delta$,
   we prove
\begin{eqnarray*}
&&\left\|\sum_{j=1}^{\infty}\lambda_{j}g_{j}^{(2)}(\widetilde{\omega})A_{2}\right\|_{L_{\omega,\widetilde{\omega}}^{r}
(\Omega\times\widetilde{\Omega})}
\leq Cr^{\frac{3}{2}}\epsilon^{2}(\|\gamma_{0}\|_{\mathfrak{S}^{2}}+1)
\end{eqnarray*}
and
\begin{eqnarray*}
&&\lim\limits_{t \rightarrow 0}(\mathbb{P}\times \widetilde{\mathbb{P}})
\left(\left\{(\omega \times \widetilde{\omega}) \in(\Omega \times \widetilde{\Omega})|| F(t, \omega, \widetilde{\omega})|>C\alpha_{2}\right\}\right)=0.
\end{eqnarray*}
Here $r \in[2, \infty)$, $A_{2}=\left(|f_{j}^{\omega}|^{2}-|e^{it\sqrt{\partial_{x}^{4}-\partial_{x}^{2}}}f_{j}^{\omega}|^{2}\right)$, $(f_{j})_{j=1}^{\infty}$
 is an orthonormal system in $L^{2}(\mathbf{T})$ and $f_{j}^{\omega}$ is the randomization of
  $f_{j}$  defined as in \eqref{1.027}.
\begin{eqnarray*}
&&F(t, \omega, \widetilde{\omega})=\sum\limits_{j=1}^{\infty} \lambda_{j} g_{j}^{(2)}(\widetilde{\omega})A_{2},\alpha_{2}=\left(\left\|\gamma_0\right\|_{\mathfrak{S}^2}+1\right)
 e \epsilon^{\frac{1}{2}}\left(\epsilon \ln \frac{1}{\epsilon}\right)^{\frac{3}{2}},\\
&&
\rho_{\gamma_{0}^{\omega, \tilde{\omega}}}=\sum_{j=1}^{\infty} \lambda_j g_{j}^{(2)}
(\widetilde{\omega})\left|f_{j}^{\omega}\right|^2,\\ &&\rho_{e^{it\sqrt{\partial_{x}^{4}-\partial_{x}^{2}}} \gamma_{0}^{\omega, \widetilde{\omega}}e^{-it\sqrt{\partial_{x}^{4}-\partial_{x}^{2}}}}=\sum_{j=1}^{\infty} \lambda_{j} g_{j}^{(2)}(\widetilde{\omega})\left|e^{it\sqrt{\partial_{x}^{4}-\partial_{x}^{2}}} f_{j}^{\omega}\right|^{2},
\end{eqnarray*}
$\rho_{\gamma_{0}^{\omega, \tilde{\omega}}},\rho_{e^{it\sqrt{\partial_{x}^{4}-\partial_{x}^{2}}} \gamma_{0}^{\omega, \widetilde{\omega}}e^{-it\sqrt{\partial_{x}^{4}-\partial_{x}^{2}}}} $ denote
 the density function of $\gamma_{0}^{\omega, \tilde{\omega}},e^{it\sqrt{\partial_{x}^{4}-\partial_{x}^{2}}} \gamma_{0}^{\omega, \widetilde{\omega}}e^{-it\sqrt{\partial_{x}^{4}-\partial_{x}^{2}}},$  respectively.

Finally, by using full randomization of compact operator, we present the probabilistic convergence of density
 function related to $e^{it\sqrt{\partial_{x}^{4}-\partial_{x}^{2}}}$ on $\Theta=\{x\in\R^{3}:|x|<1\}$
  with $\gamma_{0}\in\mathfrak{S}^{2}$. More precisely,  for any  $\epsilon>0$,
   $\exists0<\delta<\frac{\epsilon^{2}}{M_{1}\sqrt{1+M_{1}^{2}}}$, when $|t|<\delta$, we prove
\begin{eqnarray*}
&&\left\|\sum_{j=1}^{\infty}\lambda_{j}g_{j}^{(2)}(\widetilde{\omega})A_{3}\right\|_{L_{\omega_{1},\widetilde{\omega}}^{r}
(\Omega_{1}\times\widetilde{\Omega})}
\leq Cr^{\frac{3}{2}}\epsilon^{2}(\|\gamma_{0}\|_{\mathfrak{S}^{2}}+1)
\end{eqnarray*}
and
\begin{eqnarray*}
&&\lim\limits_{t \rightarrow 0}(\mathbb{P}\times \widetilde{\mathbb{P}})
\left(\left\{(\omega_{1} \times \widetilde{\omega}) \in(\Omega_{1}\times
\widetilde{\Omega}) || F(t, \omega_{1}, \widetilde{\omega})|>C\alpha_{3}\right\}\right)=0.
\end{eqnarray*}
Here $r \in[2, \infty)$, $A_{3}=\left(|f_{j}^{\omega_{1}}|^{2}-|e^{it\sqrt{\partial_{x}^{4}-
\partial_{x}^{2}}}f_{j}^{\omega_{1}}|^{2}\right)$, $f_{j}^{\omega_{1}}$ is the
 randomization of $f_{j}$  defined
  as in \eqref{1.028}
\begin{eqnarray*}
&&F(t, \omega_{1}, \widetilde{\omega})=\sum\limits_{j=1}^{\infty}
\lambda_{n} g_{j}^{(2)}
(\widetilde{\omega})A_{3},\alpha_{3}=
\left(\left\|\gamma_0\right\|_{\mathfrak{S}^2}+1\right) e \epsilon^{\frac{1}{2}}
\left(\epsilon \ln \frac{1}{\epsilon}\right)^{\frac{3}{2}},\\
&&
\rho_{\gamma_{0}^{\omega_{1}, \tilde{\omega}}}=\sum_{j=1}^{\infty} \lambda_{j}g_{j}^{(2)}
(\widetilde{\omega})\left|f_{j}^{\omega_{1}}\right|^2,\\ &&
\rho_{e^{it\sqrt{\partial_{x}^{4}-\partial_{x}^{2}}} \gamma_{0}^{\omega_{1},
\widetilde{\omega}}e^{-it\sqrt{\partial_{x}^{4}-\partial_{x}^{2}}}}=
\sum_{j=1}^{\infty} \lambda_{j} g_{j}^{(2)}(\widetilde{\omega})
\left|e^{it\sqrt{\partial_{x}^{4}-\partial_{x}^{2}}} f_{j}^{\omega_{1}}\right|^{2},
\end{eqnarray*}
$\rho_{\gamma_{0}^{\omega_{1},\tilde{\omega}}},\rho_{e^{it\sqrt{\partial_{x}^{4}
-\partial_{x}^{2}}} \gamma_{0}^{\omega_{1}, \widetilde{\omega}}
e^{-it\sqrt{\partial_{x}^{4}-\partial_{x}^{2}}}}$
denote the density function of $\gamma_{0}^{\omega_{1}, \tilde{\omega}},
e^{it\sqrt{\partial_{x}^{4}-\partial_{x}^{2}}} \gamma_{0}^{\omega_{1},\widetilde{\omega}}
e^{-it\sqrt{\partial_{x}^{4}-\partial_{x}^{2}}},$ respectively.

\subsection{Introduction to notation and Spaces}
We present some notation before stating the main results.
$\frac{1}{p}+\frac{1}{p^{\prime}}=1$. When $x\in\R$, we have
\begin{eqnarray*}
&&\mathscr{F}_{x}f(\xi)=\frac{1}{2\pi}\int_{\SR}e^{-ix\xi}f(x)dx,\,\, \mathscr{F}_{x}^{-1}f(x)=\frac{1}{2\pi}\int_{\SR}e^{ix\xi}\mathscr{F}_{x}f(\xi)d\xi.
\end{eqnarray*}
When $x\in\mathbf{T}=[0,2\pi)$, we have
\begin{eqnarray*}
&&\mathscr{F}_{x}f(k)=\frac{1}{2\pi}\int_{\mathbf{T}}e^{-ixk}f(x)dx,\,\, \mathscr{F}_{x}^{-1}f(x)=\frac{1}{2\pi}\sum\limits_{k=-\infty}^{+\infty}e^{ixk}\mathscr{F}_{x}f(k).
\end{eqnarray*}
For given $\alpha\in(0,d]$, the Borel measure $\mu$ on $\R^{d}$ is said to be $\alpha$-dimensional if
\begin{eqnarray*}
&&\sup\limits_{x\in\SR^{d},r\geq0}\frac{\mu(B(x,r))}{r^{\alpha}}<\infty,
\end{eqnarray*}
and we shall use $\mathcal{M}^{\alpha}(\mathbf{B^{d}})$ to denote the collection of all $\alpha$-dimensional
 probability measures supported on the unit ball $\mathbf{B^{d}}$.

The Schatten space $\mathfrak{S}^{\alpha}=\mathfrak{S}^{\alpha}(L^{2})$, $1\leq\alpha<\infty$,
 be defined as the set of all compact operators $\gamma$ on $L^{2}$ such that the sequence of
  eigenvalues $(\lambda_{j}^{2})_{j}$ of $\gamma^{\ast}\gamma$ belongs to $\ell^{\alpha}(\mathbb{C})$,
  in this case we define
\begin{eqnarray*}
&\|\gamma\|_{\mathfrak{S}^{\alpha}}=\|\lambda\|_{\ell^{\alpha}}=\left(\sum\limits_{j}
|\lambda_{j}|^{\alpha}\right)^{\frac{1}{\alpha}}.
\end{eqnarray*}
When $\alpha=2$, this coincides with the Hilbert-Schmidt norm and, when the operator is given
by an integral kernel, this coincides with the $L^{2}$ norm of the kernel. Moreover, when
$\alpha=\infty$, we define $\|\gamma\|_{\mathfrak{S}^{\infty}}$ to be the operator norm
of $\gamma$ on $L^{2}$.
We denote $L^{p,r}=L^{p,r}(\R)$ for the Lorentz space of measurable functions $f$ on $\R$
with $\|f\|_{L^{p,r}}<\infty$, where
\begin{equation}
\|f\|_{L^{p,r}}=\left\{
\begin{array}{lr}
\left(\int_{0}^{\infty}(t^{\frac{1}{p}}f^{\ast}(t))^{r}\frac{dt}{t}\right)^{\frac{1}{r}},
 &  p, r\in [1,\infty)\nonumber\\
\sup\limits_{t>0}t^{\frac{1}{p}}f^{\ast}(t),  & p\in[1,\infty],\, r=\infty.
\end{array}
\right.
\end{equation}
Here $f^{\ast}$ is the decreasing rearrangement of $f$ be defined as
\begin{eqnarray*}
&f^{\ast}(t)=\inf\{\mu\geq0:\alpha_{f}(s)\leq t\},
\end{eqnarray*}
where $\alpha_{f}$ is the distribution function of $f$ given by
\begin{eqnarray*}
&\alpha_{f}(s)=|\{x\in\R:|f(x)|>s\}|,
\end{eqnarray*}
from which, we have
\begin{equation}
\|f\|_{L^{p,r}}=\left\{
\begin{array}{lr}
p^{\frac{1}{r}}\left(\int_{0}^{\infty}(s\alpha_{f}(s)^{\frac{1}{p}})^{r}\frac{ds}{s}\right)^{\frac{1}{r}},
 &  p, r\in [1,\infty)\nonumber\\
\sup\limits_{s>0}s\alpha_{f}(s)^{\frac{1}{p}},  & p\in[1,\infty],\, r=\infty.
\end{array}
\right.
\end{equation}
In particularly, when $p=r$, we have $L^{p,p}=L^{p}$.
For a general Banach space $X$ and sequence $(g_{k})_{k}\subset X$, the norm is given by
\begin{eqnarray*}
&&\|(g_{k})_{k}\|_{\ell_{\mu}^{p}(X)}=\left(\sum\limits_{k}2^{p\mu k}\|g_{k}\|_{X}^{p}\right)^{\frac{1}{p}}
\end{eqnarray*}
for $p<\infty$, and $\|(g_{k})_{k}\|_{\ell_{\mu}^{\infty}(X)}=\sup\limits_{k}2^{\mu k}\|g_{k}\|_{X}$.

\subsection{Randomization of the data on $\R$, $\mathbf{T}$ and $\Theta$}
By using the method in \cite{BOP-2015, BOP2015, ZF2011, ZF2012}, we consider the following
randomization process. We assume that $\psi\in C_{c}^{\infty}(\R)$  is a real-valued,
even, non-negative bump function with $\supp\psi\subset [-1,1]$, such that
\begin{eqnarray}
&&\sum\limits_{k\in \z}\psi(\xi-k)=1,\label{1.022}
\end{eqnarray}
which is known as  Wiener decomposition of  the frequency space. For $\forall k\in \Z$,
 we define the function $\psi(D-k)f:\R\rightarrow\mathbb{C}$ by
\begin{eqnarray}
&&(\psi(D-k)f)(x)=\mathscr{F}^{-1}\big(\psi(\xi-k)\mathscr{F}f\big)(x),x\in \R\label{1.023}.
\end{eqnarray}
Suppose that $\{g_{k}^{(1)}(\omega)\}_{k\in \z}$ and
 $\{g_{j}^{(2)}(\widetilde{\omega})\}_{j\in \mathbb{N}^{+}}$
 are independent sequences of zero-mean real-valued random
  variables with probability
  distributions $\mu_{k}^{1}(k\in\Z)$ and $\mu_{j}^{2}(j\in\mathbb{N}^{+})$ on the
   probability Spaces $(\Omega,\mathcal{A}, \mathbb{P})$
    and $(\widetilde{\Omega},\widetilde{\mathcal{A}},
     \widetilde{\mathbb{P}})$, respectively.
   Assume that $\mu_{k}^{1}$ and $\mu_{j}^{2}$  satisfy
    the following property:
   for $\forall \gamma\in \R,\, \forall k\in\Z,\, \forall j\in\mathbb{N}^{+}$, $\exists c>0$
\begin{eqnarray}
&&\Big|\int_{-\infty}^{+\infty}e^{\gamma_{k} x}d\mu_{k}^{1}(x)\Big|\leq e^{c\gamma_{k}^2},\label{1.024}\\
&&\Big|\int_{-\infty}^{+\infty}e^{\gamma_{j} x}d\mu_{j}^{2}(x)\Big|\leq e^{c\gamma_{j}^2}.\label{1.025}
\end{eqnarray}
We define
\begin{eqnarray}
&&f_{j}^{\omega}=\sum_{k\in\z}g_{k}^{(1)}(\omega)\psi(D-k)f_{j}\label{1.026}
\end{eqnarray}
as the randomization of $f_{j}\in L^2(\R)(j\in\mathbb{N}^{+})$.

\noindent For $f_{j}\in L^{2}(\mathbf{T})(j\in\mathbb{N}^{+})$, we define
 its randomization as
\begin{eqnarray}
&&f_{j}^{\omega}=\sum_{k\in\z}g_{k}^{(1)}(\omega)e^{ixk}\mathscr{F}_{x}f_{j}(k).\label{1.027}
\end{eqnarray}
Let $f_{j}\in L^{2}(\Theta)(j\in \mathbb{N}^{+})$ and $\Theta=\{x\in\R^{3}:|x|<1\}$ be an
orthonormal system and be radial, then, we have $f_{n}=\sum\limits_{m=1}^{\infty}c_{n,m}e_{m}$.
Inspired by \cite{BT2008II, YDHXY2024}, for any $\omega_{1}\in \Omega_{1}$, we define its randomization by
\begin{eqnarray}
&&f_{j}^{\omega_{1}}=\sum_{m=1}^{\infty}g_{m}^{(1)}(\omega_{1})\frac{c_{j,m}}{m\pi}e_{m},\label{1.028}
\end{eqnarray}
where $e_{m}=\frac{\sin m\pi|x|}{(2\pi)^{\frac{1}{2}}|x|}$, $c_{j,m}=\int_{\Theta}f_{j}e_{m}dx$.
Moreover, we define
\begin{eqnarray}
\|f\|_{L_{\omega,\widetilde{\omega}}^{p}(\Omega\times\widetilde{\Omega})}=
\left(\int_{\Omega\times\widetilde{\Omega}}|f(\omega,\widetilde{\omega})|^{p}
d(\mathbb{P}\times \widetilde{\mathbb{P}})\right)^{\frac{1}{p}}=\left(\int_{\Omega}\int_{\widetilde{\Omega}}
|f(\omega,\widetilde{\omega})|^{p}d\mathbb{P}(\omega)
d\widetilde{\mathbb{P}}(\widetilde{\omega})
\right)^{\frac{1}{p}}.\label{1.029}
\end{eqnarray}
\subsection{Full randomization of compact operator on $\R$, $\mathbf{T}$ and $\Theta$}
Inspired by \cite{HY,YDHXY2024}, now we introduce complete randomization of compact operator.
 Suppose that $\gamma_{0}$ is a compact operator. Obviously, we have the following singular value decomposition
\begin{eqnarray}
&&\gamma_{0}=\sum_{j=1}^{\infty}\lambda_{j}|f_{j}\rangle\langle f_{j}|,\label{1.030}
\end{eqnarray}
where $\lambda_{j}\in\mathbb{C}, (f_j)_{j=1}^{\infty} $ is orthonormal system in  $L^{2}(\R)$,
$L^{2}(\mathbf{T})$ and $L^{2}(\Theta)$.

\noindent For any $(\omega, \widetilde{\omega})\in \Omega\times \widetilde{\Omega}$, we define the full randomization
 of compact operator $\gamma_{0}=\sum\limits_{j=1}^{\infty}\lambda_{j}|f_{j}\rangle\langle f_{j}|$ by
\begin{eqnarray}
&&\gamma_{0}^{\omega,\widetilde{\omega}}=\sum\limits_{j=1}^{\infty}\lambda_{j}g_{j}^{(2)}
(\widetilde{\omega})|f_{j}^{\omega}\rangle\langle f_{j}^{\omega}|,\label{1.031}
\end{eqnarray}
where $(f_{j})_{j=1}^{\infty}$
 is an orthonormal system in $L^{2}(\mathbf{R})$ and $f_{j}^{\omega}$ is the randomization of
  $f_{j}$  defined as in \eqref{1.026}.

\noindent For any $(\omega, \widetilde{\omega})\in \Omega\times \widetilde{\Omega}$, we define the full randomization
 of compact operator $\gamma_{0}=\sum\limits_{j=1}^{\infty}\lambda_{j}|f_{j}\rangle\langle f_{j}|$ by
\begin{eqnarray}
&&\gamma_{0}^{\omega,\widetilde{\omega}}=\sum\limits_{j=1}^{\infty}\lambda_{j}g_{j}^{(2)}
(\widetilde{\omega})|f_{j}^{\omega}\rangle\langle f_{j}^{\omega}|,\label{1.032}
\end{eqnarray}
where $(f_{j})_{j=1}^{\infty}$
 is an orthonormal system in $L^{2}(\mathbf{T})$ and $f_{j}^{\omega}$ is the randomization of
  $f_{j}$  defined as in \eqref{1.027}.

\noindent For any $(\omega_{1}, \widetilde{\omega})\in \Omega_{1}\times \widetilde{\Omega}$, we define the full randomization
 of compact operator $\gamma_{0}=\sum\limits_{j=1}^{\infty}\lambda_{j}|f_{j}\rangle\langle f_{j}|$ by
\begin{eqnarray}
&&\gamma_{0}^{\omega_{1},\widetilde{\omega}}=\sum\limits_{j=1}^{\infty}\lambda_{j}g_{j}^{(2)}
(\widetilde{\omega})|f_{j}^{\omega_{1}}\rangle\langle f_{j}^{\omega_{1}}|,\label{1.033}
\end{eqnarray}
where $(f_{j})_{j=1}^{\infty}$
 is an orthonormal system in $L^{2}(\mathbf{\Theta})$ and $f_{j}^{\omega}$ is the randomization of
  $f_{j}$  defined as in \eqref{1.028}.

\noindent From \eqref{1.031}, by using the simple calculation, we have
\begin{eqnarray}
&&\rho_{\gamma_{0}^{\omega,\widetilde{\omega}}}
=\sum_{j=1}^{\infty}\lambda_{j}g_{j}^{(2)}(\widetilde{\omega})\left|e^{it\sqrt{\partial_{x}^{4}
-\partial_{x}^{2}}}f_{j}^{\omega}\right|^{2}.\label{1.034}
\end{eqnarray}

\noindent From \eqref{1.032}, by using the simple calculation, we have
\begin{eqnarray}
&&\rho_{\gamma_{0}^{\omega,\widetilde{\omega}}}
=\sum_{j=1}^{\infty}\lambda_{j}g_{j}^{(2)}(\widetilde{\omega})\left|e^{it\sqrt{\partial_{x}^{4}
-\partial_{x}^{2}}}f_{j}^{\omega}\right|^{2}.\label{1.035}
\end{eqnarray}

\noindent From \eqref{1.033}, by using the simple calculation, we have
\begin{eqnarray}
&&\rho_{\gamma_{0}^{\omega_{1},\widetilde{\omega}}}=\sum_{j=1}^{\infty}\lambda_{j}g_{j}^{(2)}(\widetilde{\omega})\left|e^{it\sqrt{\partial_{x}^{4}
-\partial_{x}^{2}}}f_{j}^{\omega_{1}}\right|^{2}.\label{1.036}
\end{eqnarray}

\subsection{The main results }

\begin{Theorem}(Maximal-in-time estimates for orthonormal functions related to $e^{it\sqrt{\partial_{x}^{4}
-\partial_{x}^{2}}}$  on $\mathbf{R}$)\label{Theorem1}
Let $\beta<2$, $\lambda=(\lambda_{j})_{j}\in \ell^{\beta}$  and $(f_j)_j$ be
 an orthonormal system in $\dot{H}^{\frac{1}{4}}(\mathbf{R})$.  Then,  the following inequality
\begin{eqnarray}
&&\Big\|\sum_{j}\lambda_{j}|e^{it\sqrt{-\partial_{x}^{2}+\partial_{x}^{4}}}f_{j}|^{2}
\Big\|_{L_{x}^{2,\infty}L_{t}^{\infty}(\mathbf{R}\times I)}\leq C\|\lambda\|_{\ell^{\beta}}\label{1.037}
\end{eqnarray}
is valid.
\end{Theorem}

\noindent{\bf Remark 1:} Theorem 1.1 is used to prove Theorem 1.2.

\begin{Theorem}(Pointwise convergence of density function related to $e^{it\sqrt{\partial_{x}^{4}
-\partial_{x}^{2}}}$)\label{Theorem2}
Let $\beta<2$ and $\gamma_{0}\in \mathfrak{S}^{\beta}(\dot{H}^{\frac{1}{4}}(\mathbf{R}))$ be self-adjoint. Then, we have
\begin{eqnarray}
&\lim\limits_{t\longrightarrow 0}\rho_{\gamma(t)}(x)=\rho_{\gamma_{0}}(x), \,\,a.e.x\in\mathbf{R},\label{1.038}
\end{eqnarray}
where $\rho_{\gamma(t)}, \rho_{\gamma_{0}}$ denote the density function of
$\gamma(t)=\sum\limits_{j=1}^{\infty}\lambda_{j}
\left|e^{it\sqrt{\partial_{x}^{4}-\partial_{x}^{2}}}f_{j}\right\rangle\left\langle
e^{it\sqrt{\partial_{x}^{4}-\partial_{x}^{2}}}f_{j}\right|$, $\gamma_{0}=
\sum\limits_{j=1}^{\infty}\lambda_{j}|f_{j}\rangle\langle f_{j}|$, respectively.
\end{Theorem}

\begin{Theorem}(Maximal-in-time estimates for orthonormal functions related to
 $e^{it\sqrt{\partial_{x}^{4}
-\partial_{x}^{2}}}$ on $\mathbf{B}^{d}$)\label{Theorem3}
Let $d\geq1$, $\frac{d}{4}\leq s<\frac{d}{2}$, $0<\alpha\leq d$, $1\leq\beta<\frac{\alpha}{d-2s}$
 and $\mu\in\mathcal{M}^{\alpha}(\mathbf{B}^{d})$. Then, we have
\begin{eqnarray}
&&\left\|\sum\limits_{j}\lambda_{j}|e^{it\sqrt{-\partial_{x}^{2}+\partial_{x}^{4}}}f_{j}|^{2}
\right\|_{L_{x}^{1}(\mathbf{B}^{d},d\mu)L_{t}^{\infty}(0,1)}\leq C\|\lambda\|_{\ell^{\beta}},\label{1.039}
\end{eqnarray}
where $\lambda=(\lambda_{j})_{j}\in\ell^{\beta}$ and    $(f_{j})_{j}$ are orthonormal systems in
  $\dot{H}^{s}(\R^{d})$.
\end{Theorem}

\noindent{\bf Remark 2:} Theorem 1.3 is used to prove Theorem 1.4.

\begin{Theorem}(Pointwise convergence with respect to $d \mu$ and the Hausdorff dimension of
 the divergence set)\label{Theorem4}
Let $d\geq1$, $\frac{d}{4}\leq s<\frac{d}{2}$, $0<\alpha\leq d$ and $\gamma_{0}\in
\mathfrak{S}^{\beta}(\dot{H}^{s})$ be self-adjoint with
$1\leq\beta<\frac{\alpha}{d-2s}$. Then, the following equality
\begin{eqnarray}
&\lim\limits_{t\longrightarrow 0}\rho_{\gamma(t)}(x)=\rho_{\gamma_{0}}(x), \,\,\mu-a.e.x\in
 \mathbf{B}^{d}\label{1.040}
\end{eqnarray}
is valid, where $\rho_{\gamma(t)}, \rho_{\gamma_{0}}$ denote the density function of
\,$\gamma(t)=\sum\limits_{j=1}^{\infty}\lambda_{j}
\left|e^{it\sqrt{\partial_{x}^{4}-\partial_{x}^{2}}}f_{j}\right\rangle\left\langle
e^{it\sqrt{\partial_{x}^{4}-\partial_{x}^{2}}}f_{j}\right|$, $\gamma_{0}=
\sum\limits_{j=1}^{\infty}\lambda_{j}|f_{j}\rangle\langle f_{j}|$, respectively.
 Moreover, we have
\begin{eqnarray}
&&dim_{H}D(\gamma_{0})\leq (d-2s)\beta,\label{1.041}
\end{eqnarray}
where  $D(\gamma_{0})=
\{x\in \mathbf{B}^{d}:\lim\limits_{t\longrightarrow 0}\rho_{\gamma(t)}(x)
\neq\rho_{\gamma_{0}}(x)\}$, $dim_{H}D(\gamma_{0})$ is the Hausdorff
dimension of $D(\gamma_{0})$.
\end{Theorem}
\begin{Theorem}\label{Theorem5} (Strichartz estimates for orthonormal functions related to
 $e^{it\sqrt{\partial_{x}^{4}
-\partial_{x}^{2}}}$ on $\mathbf{T}$)Let $\beta\leq\frac{2q}{q+1}$,
  $N\geq1,\,S_{N}=\Z\cap[-N,N]$,
 $(\frac{1}{q},\frac{1}{p})\in (A,B],\,A=(0,\frac{1}{2}),\,B=(1,0)$,
 $\lambda=(\lambda_{j})_{j=1}^{+\infty}\in\ell^{\beta}$ and orthonormal
 system $(f_{j})_{j=1}^{+\infty}$ in $L^{2}(\mathbf{T})$. Then, we have
\begin{eqnarray}
&\left\|\sum\limits_{j=1}^{\infty}\lambda_{j}\left|\mathscr{D}_{N}f_{j}\right|^{2}
\right\|_{L_{t}^{p}L_{x}^{q}(\mathbf{T}\times\mathbf{T})}\leq CN^{\frac{1}{p}}
\|\lambda\|_{\ell^{\beta}},\label{1.042}
\end{eqnarray}
where $\mathbf{T}=[0,2\pi)$ and
\begin{eqnarray*}
&&\mathscr{D}_{N}f_{j}=e^{it\sqrt{\partial_{x}^{4}-\partial_{x}^{2}}}P_{\leq N}f_{j}
=\frac{1}{2\pi}\sum\limits_{n\in S_{N}}\mathscr{F}_{x}f(n)e^{i(xn+t\sqrt{n^{2}+n^{4}})}.
\end{eqnarray*}
\end{Theorem}
\begin{Theorem}\label{Theorem6} (Maximal-in-space estimates for orthonormal functions
related to  $e^{it\sqrt{\partial_{x}^{4}
-\partial_{x}^{2}}}$ on $\mathbf{T}$)Let $\beta\leq2$,
$N\geq1,\, \mathbf{T}=[0,2\pi)$,
\,$S_{N}=\Z\cap[-N,N]$,$(\frac{1}{q},\frac{1}{p})=A=(0,\frac{1}{2})$,
$\lambda=(\lambda_{j})_{j=1}^{+\infty}\in\ell^{\beta}$ and orthonormal system
$(f_{j})_{j=1}^{+\infty}$ in $L^{2}(\mathbf{T})$. Then, we have
\begin{eqnarray}
&\left\|\sum\limits_{j=1}^{\infty}\lambda_{j}\left|\mathscr{D}_{N}f_{j}\right|^{2}
\right\|_{L_{t}^{2}L_{x}^{\infty}(\mathbf{T}\times\mathbf{T})}\leq CN^{\frac{1}{2}}
\|\lambda\|_{\ell^{\beta}},\label{1.043}
\end{eqnarray}
where
\begin{eqnarray*}
&&\mathscr{D}_{N}f_{j}=e^{it\sqrt{\partial_{x}^{4}-\partial_{x}^{2}}}P_{\leq N}f_{j}
=\frac{1}{2\pi}\sum\limits_{n\in S_{N}}\mathscr{F}_{x}f(n)e^{i(xn+t\sqrt{n^{2}+n^{4}})}.
\end{eqnarray*}
\end{Theorem}
\begin{Theorem}\label{Theorem7} (Optimality of Theorems 1.5, 1.6)If $\beta>\frac{2q}{q+1}$,
 then we have that \eqref{1.042} is not valid. If $\beta>2$, then we have that
  \eqref{1.043} is not valid.
\end{Theorem}
\begin{Theorem}\label{Theorem8} (Stochastic continuity at zero related to Schatten norm on $\R$)
Suppose that $r \in[2, \infty)$, $(f_{j})_{j=1}^{\infty}$ is an orthonormal system in
 $L^{2}\left(\R\right)$ and $f_{j}^{\omega}$  the randomization of $f_{j}$  defined as in
 \eqref{1.026}. Then,  $\forall \epsilon>0$,
  $\exists0<\delta<\frac{\epsilon^{2}}{M_{1}\sqrt{1+M_{1}^{2}}}$, when $|t|<\delta$, we derive
\begin{eqnarray}
&&\left\|\sum_{j=1}^{\infty}\lambda_{j}g_{j}^{(2)}
(\widetilde{\omega})A_{1}\right\|_{L_{\omega,\widetilde{\omega}}^{r}
(\Omega\times\widetilde{\Omega})}
\leq Cr^{\frac{3}{2}}\epsilon^{2}(\|\gamma_{0}\|_{\mathfrak{S}^{2}}+1),\label{1.044}
\end{eqnarray}
where $A_{1}=\left(|f_{j}^{\omega}|^{2}-|e^{it\sqrt{\partial_{x}^{4}
-\partial_{x}^{2}}}f_{j}^{\omega}|^{2}\right)$ and  $\gamma_{0} \in \mathfrak{S}^{2}$ and.
Moreover, we have
\begin{eqnarray}
&&\lim\limits_{t \rightarrow 0}(\mathbb{P}\times \widetilde{\mathbb{P}})
\left(\left\{(\omega \times \widetilde{\omega}) \in(\Omega \times \widetilde{\Omega}) || F(t, \omega, \widetilde{\omega})|>C\alpha_{1}\right\}\right)=0.\label{1.045}
\end{eqnarray}
Here
\begin{eqnarray*}
&&F(t, \omega, \widetilde{\omega})=\sum\limits_{j=1}^{\infty} \lambda_{j} g_{j}^{(2)}(\widetilde{\omega})A_{1},\alpha_{1}=\left(\left\|\gamma_0\right\|_{\mathfrak{S}^2}+1\right)
e \epsilon^{\frac{1}{2}}\left(\epsilon \ln \frac{1}{\epsilon}\right)^{\frac{3}{2}},\\
&&
\rho_{\gamma_{0}^{\omega, \tilde{\omega}}}=\sum_{j=1}^{\infty} \lambda_j g_{j}^{(2)}(\widetilde{\omega})
\left|f_{j}^{\omega}\right|^2, \\ &&\rho_{e^{it\sqrt{\partial_{x}^{4}-\partial_{x}^{2}}} \gamma_{0}^{\omega, \widetilde{\omega}}e^{-it\sqrt{\partial_{x}^{4}-\partial_{x}^{2}}}}=\sum_{j=1}^{\infty} \lambda_{j} g_{j}^{(2)}(\widetilde{\omega})\left|e^{it\sqrt{\partial_{x}^{4}-\partial_{x}^{2}}} f_{j}^{\omega}\right|^{2},
\end{eqnarray*}
$\rho_{\gamma_{0}^{\omega, \tilde{\omega}}},\rho_{e^{it\sqrt{\partial_{x}^{4}-\partial_{x}^{2}}} \gamma_{0}^{\omega, \widetilde{\omega}}e^{-it\sqrt{\partial_{x}^{4}-\partial_{x}^{2}}}} $ denote the density function
 of $\gamma_{0}^{\omega, \tilde{\omega}},e^{it\sqrt{\partial_{x}^{4}-\partial_{x}^{2}}} \gamma_{0}^{\omega, \widetilde{\omega}}e^{-it\sqrt{\partial_{x}^{4}-\partial_{x}^{2}}},$ respectively.

\end{Theorem}
\begin{Theorem}\label{Theorem9} (Stochastic continuity at zero related
to Schatten norm on $\mathbf{T}$) Suppose that $r \in[2, \infty)$, $(f_{j})_{j=1}^{\infty}$
 is an orthonormal system in $L^{2}(\mathbf{T})$ and $f_{j}^{\omega}$  is the randomization of
  $f_{j}$  defined as in \eqref{1.027}. Then, $\forall \epsilon>0$,
  $\exists0<\delta<\frac{\epsilon^{2}}{M_{1}\sqrt{1+M_{1}^{2}}}$, when $|t|<\delta$,
   we have
\begin{eqnarray}
&&\left\|\sum_{j=1}^{\infty}\lambda_{j}g_{j}^{(2)}(\widetilde{\omega})A_{2}\right\|_{L_{\omega,\widetilde{\omega}}^{r}
(\Omega\times\widetilde{\Omega})}
\leq Cr^{\frac{3}{2}}\epsilon^{2}(\|\gamma_{0}\|_{\mathfrak{S}^{2}}+1),\label{1.046}
\end{eqnarray}
where $A_{2}=\left(|f_{j}^{\omega}|^{2}-|e^{it\sqrt{\partial_{x}^{4}-\partial_{x}^{2}}}f_{j}^{\omega}|^{2}\right)$ and $\gamma_{0} \in \mathfrak{S}^{2}$.
Moreover, we have
\begin{eqnarray}
&&\lim\limits_{t \rightarrow 0}(\mathbb{P}\times \widetilde{\mathbb{P}})
\left(\left\{(\omega \times \widetilde{\omega}) \in(\Omega \times \widetilde{\Omega})|| F(t, \omega, \widetilde{\omega})|>C\alpha_{2}\right\}\right)=0.\label{1.047}
\end{eqnarray}
Here
\begin{eqnarray*}
&&F(t, \omega, \widetilde{\omega})=\sum\limits_{j=1}^{\infty} \lambda_{j} g_{j}^{(2)}(\widetilde{\omega})A_{2},\alpha_{2}=\left(\left\|\gamma_0\right\|_{\mathfrak{S}^2}+1\right)
 e \epsilon^{\frac{1}{2}}\left(\epsilon \ln \frac{1}{\epsilon}\right)^{\frac{3}{2}},\\
&&
\rho_{\gamma_{0}^{\omega, \tilde{\omega}}}=\sum_{j=1}^{\infty} \lambda_j g_{j}^{(2)}
(\widetilde{\omega})\left|f_{j}^{\omega}\right|^2,\\ &&\rho_{e^{it\sqrt{\partial_{x}^{4}-\partial_{x}^{2}}} \gamma_{0}^{\omega, \widetilde{\omega}}e^{-it\sqrt{\partial_{x}^{4}-\partial_{x}^{2}}}}=\sum_{j=1}^{\infty} \lambda_{j} g_{j}^{(2)}(\widetilde{\omega})\left|e^{it\sqrt{\partial_{x}^{4}-\partial_{x}^{2}}} f_{j}^{\omega}\right|^{2},
\end{eqnarray*}
$\rho_{\gamma_{0}^{\omega, \tilde{\omega}}},\rho_{e^{it\sqrt{\partial_{x}^{4}-\partial_{x}^{2}}} \gamma_{0}^{\omega, \widetilde{\omega}}e^{-it\sqrt{\partial_{x}^{4}-\partial_{x}^{2}}}} $ denote
 the density function of $\gamma_{0}^{\omega, \tilde{\omega}},e^{it\sqrt{\partial_{x}^{4}-\partial_{x}^{2}}} \gamma_{0}^{\omega, \widetilde{\omega}}e^{-it\sqrt{\partial_{x}^{4}-\partial_{x}^{2}}},$  respectively.
\end{Theorem}
\begin{Theorem}\label{Theorem10} (Stochastic continuity at zero related to
 Schatten norm
 on $\Theta=\{x\in\R^{3}:|x|<1\}$) Let $r \in[2, \infty)$,  $f_{j}^{\omega_{1}}$ be the randomization of $f_{j}$  defined
  as in \eqref{1.028}. For  $\forall \epsilon>0$, $\exists0<\delta<\frac{\epsilon^{2}}{M_{1}\sqrt{1+M_{1}^{2}}}$, when $|t|<\delta$,  the following inequality is valid:
\begin{eqnarray}
&&\left\|\sum_{j=1}^{\infty}\lambda_{j}g_{j}^{(2)}(\widetilde{\omega})A_{3}\right\|_{L_{\omega_{1},\widetilde{\omega}}^{r}
(\Omega_{1}\times\widetilde{\Omega})}
\leq Cr^{\frac{3}{2}}\epsilon^{2}(\|\gamma_{0}\|_{\mathfrak{S}^{2}}+1),\label{1.048}
\end{eqnarray}
where $A_{3}=\left(|f_{j}^{\omega_{1}}|^{2}-|e^{it\sqrt{\partial_{x}^{4}-
\partial_{x}^{2}}}f_{j}^{\omega_{1}}|^{2}\right)$ and  and $\gamma_{0} \in \mathfrak{S}^{2}$.
Moreover, we have
\begin{eqnarray}
&&\lim\limits_{t \rightarrow 0}(\mathbb{P}\times \widetilde{\mathbb{P}})\left(\left\{(\omega_{1} \times \widetilde{\omega}) \in(\Omega_{1}\times \widetilde{\Omega}) || F(t, \omega_{1}, \widetilde{\omega})|>C\alpha_{3}\right\}\right)=0.\label{1.049}
\end{eqnarray}
Here
\begin{eqnarray*}
&&F(t, \omega_{1}, \widetilde{\omega})=\sum\limits_{j=1}^{\infty} \lambda_{n} g_{j}^{(2)}(\widetilde{\omega})A_{3},\alpha_{3}=\left(\left\|\gamma_0\right\|_{\mathfrak{S}^2}+1\right) e \epsilon^{\frac{1}{2}}\left(\epsilon \ln \frac{1}{\epsilon}\right)^{\frac{3}{2}},\\
&&
\rho_{\gamma_{0}^{\omega_{1}, \tilde{\omega}}}=\sum_{j=1}^{\infty} \lambda_{j}g_{j}^{(2)}(\widetilde{\omega})\left|f_{j}^{\omega_{1}}\right|^2,\\ &&\rho_{e^{it\sqrt{\partial_{x}^{4}-\partial_{x}^{2}}} \gamma_{0}^{\omega_{1}, \widetilde{\omega}}e^{-it\sqrt{\partial_{x}^{4}-\partial_{x}^{2}}}}=\sum_{j=1}^{\infty} \lambda_{j} g_{j}^{(2)}(\widetilde{\omega})\left|e^{it\sqrt{\partial_{x}^{4}-\partial_{x}^{2}}} f_{j}^{\omega_{1}}\right|^{2},
\end{eqnarray*}
$\rho_{\gamma_{0}^{\omega_{1},\tilde{\omega}}},\rho_{e^{it\sqrt{\partial_{x}^{4}-\partial_{x}^{2}}} \gamma_{0}^{\omega_{1}, \widetilde{\omega}}e^{-it\sqrt{\partial_{x}^{4}-\partial_{x}^{2}}}}$ denote the density function of $\gamma_{0}^{\omega_{1}, \tilde{\omega}},e^{it\sqrt{\partial_{x}^{4}-\partial_{x}^{2}}} \gamma_{0}^{\omega_{1},\widetilde{\omega}}e^{-it\sqrt{\partial_{x}^{4}-\partial_{x}^{2}}},$ respectively.
\end{Theorem}

The rest of the paper is arranged as follows.

In Section 2,  firstly, we present the pointwise convergence result related to Boussinesq equation;
secondly, by using the result of Zhang \cite{Z2014} and
polar  coordinate transformation and asymptotic expansion of Fourier transform on the spherical surface as well as Lemma 3.20 of \cite{YDHXY2024}, we present
some oscillatory integral estimates;  by using Van der Corput method in analytic number theory, we prove that
 \begin{eqnarray*}
 \left|\sum\limits_{|k|\leq N}e^{itk\sqrt{k^{2}+1}+ikx}\right|\leq C|t|^{-1/2};
 \end{eqnarray*}
thirdly, we present some Strichartz estimates for orthonormal functions and Schatten bound  with space-time norms;
fourthly, we present some interpolation theorems. Finally, we present the three line theorem.

In Section 3,  we present noncommutative-commutative Riesz-Throin interpolation theorem   related to  Lebesgue spaces, noncommutative-commutative Riesz-Throin interpolation theorem related to mixed Lebesgue spaces, noncommutative-commutative Stein interpolation theorem, noncommutative-commutative Stein interpolation theorem related to  mixed Lebesgue spaces, which are Lemma 3.1, Lemma 3.2, Lemma 3.3, Lemma 3.4, respectively.

In Section 4, we present  probabilistic estimates of some random series and some estimates related to $p$-th moment of random series.

 In Section 5,
 we prove
 Theorem 1.1. More precisely, we present the maximal function estimates for orthonormal functions on $\R.$

 In Section 6, we prove
 Theorems 1.2. More precisely, we present the pointwise convergence of density function of $\sum\limits_{j=1}^{\infty}\lambda_{j}|e^{it(-\Delta)^{m/2}}f_{j}\rangle
   \langle e^{it(-\Delta)^{m/2}}f_{j}| $.

 In Section 7, we prove
 Theorem 1.3. More precisely, we present the maximal function estimates for orthonormal functions on $\mathbf{B}^{d}(d\geq1)$.

 In Section 8, we prove
 Theorem 1.4. More precisely, we present  the Hausdorff dimension of the divergence set
related to density function $\sum\limits_{j=1}^{\infty}\lambda_{j}|e^{it(-\Delta)^{m/2}}f_{j}|^{2}$  of
   operator $\sum\limits_{j=1}^{\infty}\lambda_{j}|e^{it(-\Delta)^{m/2}}f_{j}\rangle
   \langle e^{it(-\Delta)^{m/2}}f_{j}| $ with the aid of Frostman's  Lemma.

 In Section 9,
 we prove
 Theorem 1.5. More precisely, we present  the Schatten bound with space-time norms on $\mathbf{T}$ and Strichartz estimates for orthonormal functions.

 In Section 10, we prove
 Theorems 1.6. More precisely, we present  the Schatten bound with space-time norms on $\mathbf{T}$ and maximal function  estimates for orthonormal functions with the space vatiables.

 In Section 11, we prove
 Theorem 1.7. More precisely, we present the optimality of Theorems 1.5, 1.6.

 In Section 12, we prove
 Theorem 1.8. More precisely, we present the stochastic continuity at zero related to Schatten norm on $\mathbf{R}$.

 In Section 13, we prove
 Theorem 1.9. More precisely, we present the stochastic continuity at zero related to Schatten norm on $\mathbf{T}$.

 In Section 14, we prove
 Theorem 1.10. More precisely, we present the  stochastic continuity at zero related to Schatten norm on $\Theta=\{x\in\R^{3}:|x|<1\}$.

\section{ Preliminaries }

\setcounter{equation}{0}

\setcounter{Theorem}{0}

\setcounter{Lemma}{0}

\setcounter{section}{2}

\begin{Lemma}\label{lem2.1}
For $f\in \dot{H}^{\frac{1}{4}}(\R)$, we have
\begin{eqnarray}
&\lim\limits_{t\longrightarrow 0}U(t)f(x)=f(x),\quad almost \quad  everywhere \,\,x\in\mathbf{R}, \label{2.01}
\end{eqnarray}
where $U(t)f=\int_{\SR}e^{ix\xi}e^{it\sqrt{\xi^{4}+\xi^{2}}}\mathscr{F}_{x}f(\xi)d\xi.$
\end{Lemma}

For the proof of  Lemma 2.1, we refer  the readers  to  \cite{LL2023}.

\begin{Lemma}\label{lem2.2}
Let $d\geq1$. Then,   we have
\begin{eqnarray*}
\Big|\int_{\SR^{d}}e^{ix \xi+it|\xi|^{2}}\frac{d\xi}{|\xi|^{\frac{d}{2}}}\Big|\leq C|x|^{-\frac{d}{2}}.
\end{eqnarray*}
\end{Lemma}

For the proof of  Lemma 2.2, we refer  the readers  to Lemmas 2.1,  2.3 of \cite{Z2014}.

\begin{Lemma}\label{lem2.3}
Let $|t|\leq 1$, $\phi(\xi)=\sqrt{\xi^{2}+\xi^{4}}$. Then,  we have
\begin{eqnarray}
\Big|\int_{\SR}e^{ix \xi+it\phi(\xi)}\frac{d\xi}{|\xi|^{\frac{1}{2}}}\Big|\leq C|x|^{-\frac{1}{2}}\label{2.02}.
\end{eqnarray}
\end{Lemma}

By using a proof similar to Lemma 2.2 of \cite{LL2023}, we have that Lemma 2.3 is valid.

\begin{Lemma}\label{lem2.4}
Let $d\in\mathbb{N}^{+}$ and $\psi\in C_{0}^{\infty}$ be supported on $\{x\in\R:\frac{1}{2}<x<2\}$. For $m\in(0,\infty)\setminus\{1\}$, $k\in\Z$,  we have
\begin{eqnarray}
&&\sup\limits_{t\in\SR}\left|\int_{\SR^{d}}e^{i(x\xi+t|\xi|^{m})}\phi(2^{-k}|\xi|)d\xi\right|\leq C\frac{2^{dk}}{(1+2^{k}|x|)^{\frac{d}{2}}}.\label{2.03}
\end{eqnarray}
\end{Lemma}

For the proof of  Lemma 2.4, we refer  the readers  to \cite{BKS2023}.

\begin{Lemma}\label{lem2.5}
Let $d\geq1$, $|t|\leq 1$, $|x|\leq 1$, $\frac{d}{4}<s<\frac{d}{2}$, $\psi\in C_{0}^{\infty}$ be supported on $\{x\in\R:\frac{1}{2}<x<2\}$. For $k\in\Z$,  we have
\begin{eqnarray}
&&\sup\limits_{|t|\leq 1}\left|\int_{\SR^{d}}e^{i(x\cdot\xi+t\sqrt{|\xi|^{4}+|\xi|^{2}})}\phi(2^{-k}|\xi|)d\xi\right|\leq C\frac{2^{dk}}{(1+2^{k}|x|)^{\frac{d}{2}}},\label{2.04}\\
&&\sup\limits_{|t|\leq 1}\left|\int_{\SR^{d}}e^{i(x\cdot\xi+t\sqrt{|\xi|^{4}+|\xi|^{2}})}\frac{\phi(2^{-k}|\xi|)^{2}}{|\xi|^{2s}}d\xi\right|\leq C\frac{2^{(d-2s)k}}{(1+2^{k}|x|)^{\frac{d}{2}}},\label{2.05}\\
&&\sup\limits_{|t|\leq 1}\left|\int_{\SR^{d}}e^{i(x\cdot\xi+t\sqrt{|\xi|^{4}+|\xi|^{2}})}\frac{1}{|\xi|^{\frac{d}{2}}}d\xi\right|\leq C|x|^{-\frac{d}{2}}.\label{2.06}
\end{eqnarray}
\end{Lemma}
\begin{proof}
For $\eqref{2.04}$, we consider the following case $2^{k}|x|\leq 1$ and $2^{k}|x|\geq1$, respectively.
When $2^{k}|x|\leq 1$, we have $\frac{2^{dk}}{(1+2^{dk}|x|)^{\frac{1}{2}}}\sim 2^{dk}$, then \eqref{2.04} is valid.

\noindent When $2^{k}|x|\geq 1$, we have
\begin{align}
\left|\int_{\SR^{d}}e^{i(x\cdot\xi+t\sqrt{|\xi|^{4}+|\xi|^{2}})}\phi(2^{-k}|\xi|)d\xi\right|&\leq \left|\int_{\SR^{d}}e^{ix\cdot\xi}\left(e^{it\sqrt{|\xi|^{4}+|\xi|^{2}}}-e^{it(|\xi|^{2}+\frac{1}{2})}\right)\phi(2^{-k}|\xi|)d\xi\right|\nonumber\\
&\qquad+
\left|\int_{\SR^{d}}e^{ix\cdot\xi}e^{it(|\xi|^{2}+\frac{1}{2})}\phi(2^{-k}|\xi|)d\xi\right|\nonumber\\
&=I_{1}+I_{2}.\label{2.07}
\end{align}
For $I_{1}$, by using Taylor expansion, since $|t|\leq 1$, $|x|\leq 1$, we have
\begin{align}
I_{1}&=\left|\sum\limits_{m=1}^{\infty}\int_{\SR^{d}}e^{ix\cdot\xi}e^{it|\xi|^{2}}
\frac{(it(\sqrt{|\xi|^{4}+|\xi|^{2}}-(|\xi|^{2}+\frac{1}{2})))^{m}}{m!}\phi(2^{-k}|\xi|)d\xi\right|
\nonumber\\
&\leq\sum\limits_{m=1}^{\infty}\frac{|t|^{m}}{4^{m}m!}\left|\int_{\SR^{d}}e^{ix\cdot\xi}e^{it|\xi|^{2}}\left(\frac{1}{\sqrt{|\xi|^{4}+|\xi|^{2}}+|\xi|^{2}
+\frac{1}{2}}\right)^{m}\phi(2^{-k}|\xi|)d\xi\right|\nonumber\\
&\leq  \sum\limits_{m=1}^{\infty}\frac{1}{4^{m}m!}\left|\int_{\SR^{d}}e^{ix\cdot\xi}e^{it|\xi|^{2}}
\left(\frac{1}{\sqrt{|\xi|^{4}+|\xi|^{2}}+|\xi|^{2}+\frac{1}{2}}\right)^{m}\phi(2^{-k}|\xi|)
d\xi\right|\nonumber\\
&= \sum\limits_{m=1}^{\infty}\frac{1}{4^{m}m!} J_{m}.\label{2.08}
\end{align}
By using polar coordinate transformation and the asymptotic expansion of Fourier transform on a sphere, which can be found in  Page 338 and Page 347 of \cite{S1993}, we have
\begin{eqnarray}
&&J_{m}=\left|\int_{0}^{\infty}\widehat{d\sigma}(\rho x)e^{it\rho^{2}}\left(\frac{1}{\sqrt{\rho^{4}+\rho^{2}}+\rho^{2}+\frac{1}{2}}\right)^{m}\rho^{d-1}\phi(2^{-k}\rho)d\rho
\right|,\label{2.09}
\end{eqnarray}
where $\widehat{d\sigma}(\rho x)\sim e^{\pm i\rho|x|}\left(\sum\limits_{l=0}^{N-1}(\rho|x|)^{\frac{1-d}{2}-l}+O\left((\rho|x|)^{\frac{1-d}{2}-N}\right)\right)$.

\noindent By using \eqref{2.09}, we have
\begin{eqnarray}
&&J_{m}\leq \sum_{l=0}^{N-1}\left|\int_{0}^{\infty}e^{\pm i\rho|x|}e^{it\rho^{2}}(\rho|x|)^{\frac{1-d}{2}-l}
\left(\frac{1}{\sqrt{\rho^{4}+\rho^{2}}+\rho^{2}+\frac{1}{2}}\right)^{m}\rho^{d-1}\phi(2^{-k}\rho)d\rho
\right|\nonumber\\
&&+\left|\int_{0}^{\infty}e^{\pm i\rho|x|}e^{it\rho^{2}}O\left((\rho|x|)^{\frac{1-d}{2}-N}\right)
\left(\frac{1}{\sqrt{\rho^{4}+\rho^{2}}+\rho^{2}+\frac{1}{2}}\right)^{m}\rho^{d-1}\phi(2^{-k}\rho)d\rho
\right|\nonumber\\
&&=J_{m1}+J_{m2}.\label{2.010}
\end{eqnarray}
For $J_{m1}$, we have
\begin{align}
J_{m1}&\leq \sum_{l=0}^{N-1}|x|^{\frac{1-d}{2}-l}\int_{0}^{\infty}\rho^{\frac{d-1}{2}-l-2m}
|\phi(2^{-k}\rho)|d\rho\nonumber\\
&\leq C\sum_{l=0}^{N-1}4^{m}|x|^{\frac{1-d}{2}-l} 2^{(\frac{d+1}{2}-2m-l)k}.\label{2.011}
\end{align}
For $J_{m2}$, we have
\begin{align}
J_{m2}&\leq |x|^{\frac{1-d}{2}-N}\int_{0}^{\infty}\rho^{\frac{d-1}{2}-2m-N}
|\phi(2^{-k}\rho)|d\rho\nonumber\\
&\leq C4^{m}|x|^{\frac{1-d}{2}-N} 2^{(\frac{d+1}{2}-2m-N)k}.\label{2.012}
\end{align}
From \eqref{2.08}-\eqref{2.012}, since $2^{k}|x|\geq 1$, $|x|\leq 1$, we have
\begin{align}
I_{1}&\leq C\sum_{l=0}^{N-1}\sum_{m=1}^{\infty}\frac{4^{m}}{4^{m}m!}|x|^{\frac{1-d}{2}-l} 2^{(\frac{d+1}{2}-2m-l)k}+\sum_{m=1}^{\infty}\frac{4^{m}}{4^{m}m!}|x|^{\frac{1-d}{2}-N} 2^{(\frac{d+1}{2}-2m-N)k}\nonumber\\
&\leq C\sum_{l=0}^{N-1}\sum_{m=1}^{\infty}|x|^{\frac{1-d}{2}-l} 2^{(\frac{d+1}{2}-2m-l)k}+\sum_{m=1}^{\infty}|x|^{\frac{1-d}{2}-N} 2^{(\frac{d+1}{2}-2m-N)k}\nonumber\\
&\leq C\sum_{l=0}^{N-1}|x|^{\frac{1-d}{2}-l}2^{(\frac{d}{2}-l)k}+ |x|^{\frac{1-d}{2}-N} 2^{(\frac{d}{2}-N)k}\nonumber\\
&\leq C\sum_{l=0}^{N-1}|x|^{-\frac{d}{2}}2^{\frac{d}{2}k}2^{lk}2^{-lk}+|x|^{-\frac{d}{2}}2^{\frac{d}{2}k}2^{Nk}2^{-Nk}\nonumber\\
&\leq C|x|^{-\frac{d}{2}}2^{\frac{d}{2}k}.\label{2.013}
\end{align}
For $I_{2}$, by using Lemma 2.4, we have
\begin{eqnarray}
&&I_{2}\leq C\frac{2^{dk}}{(1+2^{k}|x|)^{\frac{d}{2}}}.\label{2.014}
\end{eqnarray}
From \eqref{2.013}-\eqref{2.014}, we have that \eqref{2.04} is valid.

\noindent For \eqref{2.05}, by using Young inequality and \eqref{2.04}, we have
\begin{eqnarray}
&&\left|\int_{\SR^{d}}e^{i(x\cdot\xi+t\sqrt{|\xi|^{4}+|\xi|^{2}})}\frac{\phi(2^{-k}|\xi|)^{2}}{|\xi|^{2s}}d\xi\right|\nonumber\\&&
\leq \left\|\mathscr{F}_{x}^{-1}\left(e^{it\sqrt{|\xi|^{4}+|\xi|^{2}}}\phi(2^{-k}|\xi|)\right)\right\|_{L^{\infty}}
\left\|\mathscr{F}_{x}^{-1}\left(\frac{\phi(2^{-k}|\xi|)}{|\xi|^{2s}}\right)\right\|_{L^{1}}\nonumber\\
&&\leq C\frac{2^{dk}}{(1+2^{k}|x|)^{\frac{d}{2}}}\left\|\mathscr{F}_{x}^{-1}\left(\frac{\phi(2^{-k}|\xi|)}{|\xi|^{2s}}
\right)\right\|_{L^{1}}.\label{2.015}
\end{eqnarray}
By using a change of the variable $\eta=2^{-k}\xi$, $y=2^{k}x$, then, we have
\begin{align}
\left\|\mathscr{F}_{x}^{-1}\left(\frac{\phi(2^{-k}|\xi|)}{|\xi|^{2s}}\right)\right\|_{L^{1}}&=\int_{\SR^{d}}
\left|\int_{\SR^{d}}e^{ix\cdot\xi}\frac{\phi(2^{-k}|\xi|)}{|\xi|^{2s}}d\xi\right|dx\nonumber\\
&=2^{-2sk}2^{dk}\int_{\SR^{d}}
\left|\int_{\SR^{d}}e^{i2^{k}x\cdot\eta}\frac{\phi(|\eta|)}{|\eta|^{2s}}d\eta\right|dx\nonumber\\
&=2^{-2sk}2^{dk}2^{-dk}\int_{\SR^{d}}
\left|\int_{\SR^{d}}e^{iy\cdot\eta}\frac{\phi(|\eta|)}{|\eta|^{2s}}d\eta\right|dy\nonumber\\
&\leq C2^{-2sk}.\label{2.016}
\end{align}
From \eqref{2.015}-\eqref{2.016}, we have
\begin{align}
\left|\int_{\SR^{d}}e^{i(x\cdot\xi+t\sqrt{|\xi|^{4}+|\xi|^{2}})}\frac{\phi(2^{-k}|\xi|)^{2}}{|\xi|^{2s}}d\xi\right|\leq C\frac{2^{(d-2s)k}}{(1+2^{k}|x|)^{\frac{d}{2}}}.\label{2.017}
\end{align}
From \eqref{2.017}, we have that \eqref{2.05} is valid.

\noindent For $\eqref{2.06}$, we consider the following case $2^{k}|x|\leq 1$ and $2^{k}|x|\geq1$, respectively.

\noindent Since
\begin{eqnarray}
&&\left|\int_{\SR^{d}}e^{i(x\cdot\xi+t\sqrt{|\xi|^{4}+|\xi|^{2}})}\frac{1}{|\xi|^{\frac{d}{2}}}d\xi\right|\leq \sum_{k\in\z}\left|\int_{\SR^{d}}e^{i(x\cdot\xi+t\sqrt{|\xi|^{4}+|\xi|^{2}})}\frac{\phi(2^{-k}|\xi|)}{|\xi|^{\frac{d}{2}}}d\xi\right|.\label{2.018}
\end{eqnarray}
When $2^{k}|x|\leq 1$, by using polar coordinate transformation and \eqref{2.018}, $|x|\leq 1$, we have
\begin{align}
\sum_{2^{k}\leq|x|^{-1}}\left|\int_{\SR^{d}}e^{i(x\cdot\xi+t\sqrt{|\xi|^{4}+|\xi|^{2}})}\frac{\phi(2^{-k}|\xi|)}{|\xi|^{\frac{d}{2}}}d\xi\right|
&\leq \sum_{[2^{k}\leq|x|^{-1}]+1}\int_{\SR^{d}}\frac{\phi(2^{-k}|\xi|)}{|\xi|^{\frac{d}{2}}}d\xi\nonumber\\
&\leq \sum_{k=-\infty}^{[-\log_{2}|x|]+1}\int_{0}^{+\infty}\rho^{\frac{d}{2}-1}\phi(2^{-k}\rho) d\rho\nonumber\\
&\leq C\sum_{k=-\infty}^{[-\log_{2}|x|]+1}2^{\frac{d}{2}k}\leq C|x|^{-\frac{d}{2}}.\label{2.019}
\end{align}
When $2^{k}|x|\geq 1$, since
\begin{align}
\left|\int_{\SR^{d}}e^{i(x\cdot\xi+t\sqrt{|\xi|^{4}+|\xi|^{2}})}\frac{\phi(2^{-k}|\xi|)}{|\xi|^{\frac{d}{2}}}d\xi\right|&\leq \left|\int_{\SR^{d}}e^{ix\cdot\xi}\left(e^{it\sqrt{|\xi|^{4}+|\xi|^{2}}}-e^{it(|\xi|^{2}+\frac{1}{2})}\right)\frac{\phi(2^{-k}|\xi|)}{|\xi|^{\frac{d}{2}}}
d\xi\right|\nonumber\\
&\qquad+
\left|\int_{\SR^{d}}e^{ix\cdot\xi}e^{it(|\xi|^{2}+\frac{1}{2})}\frac{\phi(2^{-k}|\xi|)}{|\xi|^{\frac{d}{2}}}d\xi\right|\nonumber\\
&=I_{3}+I_{4}.\label{2.020}
\end{align}
For $I_{3}$, by using Taylor expansion, since $|t|\leq 1$, $|x|\leq 1$, we have
\begin{align}
I_{3}&=\left|\sum\limits_{m=1}^{\infty}\int_{\SR^{d}}e^{ix\cdot\xi}e^{it|\xi|^{2}}\frac{(it(\sqrt{|\xi|^{4}+|\xi|^{2}}-(|\xi|^{2}+\frac{1}{2})))^{m}}{m!}
\frac{\phi(2^{-k}|\xi|)}{|\xi|^{\frac{d}{2}}}d\xi\right|
\nonumber\\
&\leq\sum\limits_{m=1}^{\infty}\frac{|t|^{m}}{4^{m}m!}\left|\int_{\SR^{d}}e^{ix\cdot\xi}e^{it|\xi|^{2}}
\left(\frac{1}{\sqrt{|\xi|^{4}+|\xi|^{2}}+|\xi|^{2}
+\frac{1}{2}}\right)^{m}\frac{\phi(2^{-k}|\xi|)}{|\xi|^{\frac{d}{2}}}d\xi\right|\nonumber\\
&\leq  \sum\limits_{m=1}^{\infty}\frac{1}{4^{m}m!}\left|\int_{\SR^{d}}e^{ix\cdot\xi}e^{it|\xi|^{2}}
\left(\frac{1}{\sqrt{|\xi|^{4}+|\xi|^{2}}+|\xi|^{2}+\frac{1}{2}}\right)^{m}
\frac{\phi(2^{-k}|\xi|)}{|\xi|^{\frac{d}{2}}}
d\xi\right|\nonumber\\
&= \sum\limits_{m=1}^{\infty}\frac{1}{4^{m}m!} L_{m}.\label{2.021}
\end{align}
By using polar coordinate transformation and the asymptotic expansion of Fourier transform on a sphere, which can be found in  Page 338 and Page 347 of \cite{S1993}, we have
\begin{eqnarray}
&&L_{m}=\left|\int_{0}^{\infty}\widehat{d\sigma}(\rho x)e^{it\rho^{2}}
\left(\frac{1}{\sqrt{\rho^{4}+\rho^{2}}+\rho^{2}+\frac{1}{2}}\right)^{m}\rho^{d-1}
\frac{\phi(2^{-k}\rho)}{\rho^{\frac{d}{2}}}d\rho
\right|,\label{2.022}
\end{eqnarray}
where $\widehat{d\sigma}(\rho x)\sim e^{\pm i\rho|x|}\left(\sum\limits_{l=0}^{N-1}(\rho|x|)^{\frac{1-d}{2}-l}+O\left((\rho|x|)^{\frac{1-d}{2}-N}\right)\right)$.

\noindent By using \eqref{2.022}, we have
\begin{eqnarray}
&&L_{m}\leq \sum_{l=0}^{N-1}\left|\int_{0}^{\infty}e^{\pm i\rho|x|}e^{it\rho^{2}}(\rho|x|)^{\frac{1-d}{2}-l}
\left(\frac{1}{\sqrt{\rho^{4}+\rho^{2}}+\rho^{2}+\frac{1}{2}}\right)^{m}\rho^{d-1}\frac{\phi(2^{-k}\rho)}{\rho^{\frac{d}{2}}}d\rho
\right|\nonumber\\
&&+\left|\int_{0}^{\infty}e^{\pm i\rho|x|}e^{it\rho^{2}}O\left((\rho|x|)^{\frac{1-d}{2}-N}\right)
\left(\frac{1}{\sqrt{\rho^{4}+\rho^{2}}+\rho^{2}+\frac{1}{2}}\right)^{m}\rho^{d-1}\frac{\phi(2^{-k}\rho)}{\rho^{\frac{d}{2}}}d\rho
\right|\nonumber\\
&&=L_{m1}+L_{m2}.\label{2.023}
\end{eqnarray}
For $L_{m1}$, we have
\begin{align}
L_{m1}&\leq \sum_{l=0}^{N-1}|x|^{\frac{1-d}{2}-l}\int_{0}^{\infty}\rho^{-\frac{1}{2}-l-2m}
|\phi(2^{-k}\rho)|d\rho\nonumber\\
&\leq C\sum_{l=0}^{N-1}4^{m}|x|^{\frac{1-d}{2}-l} 2^{(\frac{1}{2}-2m-l)k}.\label{2.024}
\end{align}
For $L_{m2}$, we have
\begin{align}
L_{m2}&\leq |x|^{\frac{1-d}{2}-N}\int_{0}^{\infty}\rho^{-\frac{1}{2}-2m-N}
|\phi(2^{-k}\rho)|d\rho\nonumber\\
&\leq C4^{m}|x|^{\frac{1-d}{2}-N} 2^{(\frac{1}{2}-2m-N)k}.\label{2.025}
\end{align}
From \eqref{2.021}-\eqref{2.025}, since $2^{k}|x|\geq 1$, $|x|\leq 1$, we have
\begin{align}
\sum\limits_{k=[-\log_{2}|x|]+1}^{\infty}I_{3}&\leq C\sum\limits_{k=[-\log_{2}|x|]+1}^{\infty}\sum_{l=0}^{N-1}\sum_{m=1}^{\infty}\frac{4^{m}}{4^{m}m!}|x|^{\frac{1-d}{2}-l} 2^{(\frac{1}{2}-2m-l)k}\nonumber\\
&\qquad+C\sum\limits_{k=[-\log_{2}|x|]+1}^{\infty}\sum_{m=1}^{\infty}\frac{4^{m}}{4^{m}m!}|x|^{\frac{1-d}{2}-N} 2^{(\frac{1}{2}-2m-N)k}\nonumber\\
&\leq C\sum\limits_{k=[-\log_{2}|x|]+1}^{\infty}\sum_{l=0}^{N-1}\sum_{m=1}^{\infty}|x|^{\frac{1-d}{2}-l} 2^{(\frac{1}{2}-2m-l)k}\nonumber\\
&\qquad+C\sum\limits_{k=[-\log_{2}|x|]+1}^{\infty}\sum_{m=1}^{\infty}|x|^{\frac{1-d}{2}-N} 2^{(\frac{1}{2}-2m-N)k}\nonumber\\
&\leq C\sum_{l=0}^{N-1}\sum\limits_{k=[-\log_{2}|x|]+1}^{\infty}|x|^{\frac{1-d}{2}-l}2^{-lk}+ C\sum\limits_{k=[-\log_{2}|x|]+1}^{\infty}|x|^{\frac{1-d}{2}-N} 2^{-Nk}\nonumber\\
&\leq C|x|^{-\frac{d}{2}-l}|x|^{l}+|x|^{-\frac{d}{2}-N}|x|^{N}=C|x|^{-\frac{d}{2}}.\label{2.026}
\end{align}
For $I_{4}$, by using Lemma 2.2, we have
\begin{eqnarray}
&&\sum\limits_{k\in\z}I_{4}\leq C|x|^{-\frac{d}{2}}.\label{2.027}
\end{eqnarray}
From \eqref{2.026}-\eqref{2.027}, we have that \eqref{2.06} is valid.

This completes the proof of Lemma 2.5.
\end{proof}

\begin{Lemma}\label{lem2.6} Let $b, c, d_{3}\in\R$. Then, we have
\begin{eqnarray*}
\left|b\int_{c}^{d_{3}}e^{-it\rho}\rho^{-1+ib}d\rho\right|\leq C(e^{\pi b}+e^{-\pi b})(1+|b|)^{2}.
\end{eqnarray*}

Here $C$ is independent of $c, b, d_{3}, t$.

\end{Lemma}

For the proof of Lemma 2.6, we refer the readers to Lemma 3.20 of \cite{YDHXY2024}.

\begin{Lemma}\label{lem2.7}(Van der Corput method in analytic number Theory) Let $|\phi^{\prime}|\leq \frac{1}{2}$ and $\phi^{\prime}$ and $g$ be monotone. Then, we have
\begin{eqnarray*}
&&\left|\int_{a}^{b}g(x)e^{2i\pi\phi(x)}dx-\sum_{a\leq n\leq b}g(n)e^{2i\pi\phi(n)}\right|\leq A\max_{a\leq x\leq b}g(x),
\end{eqnarray*}
where $A$ is a universal constant.
\end{Lemma}

For the proof of Lemma 2.7, we refer the readers to Page 226 of \cite{Z1968}.

\begin{Lemma}\label{lem2.8}
Let $0<|t|\leq \delta\leq1$. Then,  we have
\begin{eqnarray}
&&\left|\int_{\SR}e^{ix\xi+ it\xi\sqrt{1+\xi^{2}}}d\xi\right|\leq C|t|^{-\frac{1}{2}}.\label{2.028}
\end{eqnarray}
\end{Lemma}
\begin{proof}
Obviously, we have
\begin{eqnarray}
&&\left|\int_{\SR}e^{ix\xi+ it\xi\sqrt{1+\xi^{2}}}d\xi\right|\leq \left|\int_{\SR}e^{ix\xi}e^{it(\xi^{2}+\frac{1}{2})}d\xi\right|\nonumber\\
&&+\left|\int_{\SR}e^{ix\xi} e^{it\xi^{2}+\frac{1}{2}}\left(e^{it(\sqrt{\xi^{4}+\xi^{2}}-(\xi^{2}+\frac{1}{2}))}-1\right)\right|.\label{2.029}
\end{eqnarray}
For the first term on the right of \eqref{2.029}, we have
\begin{eqnarray}
\left|\int_{\SR}e^{ix\xi}e^{it(\xi^{2}+\frac{1}{2})}d\xi\right|= \left|\frac{1}{(-4\pi it)^{\frac{1}{2}}}e^{\frac{x^{2}}{4it}}\right|\leq C|t|^{-\frac{1}{2}}.\label{2.030}
\end{eqnarray}
For the second term on the right of \eqref{2.029}, since $0<|t|\leq\delta\leq1$, $\left|e^{it(\sqrt{\xi^{4}+\xi^{2}}-(\xi^{2}+\frac{1}{2}))}-1\right|\leq 2$ and $\left|e^{it(\sqrt{\xi^{4}+\xi^{2}}-(\xi^{2}+\frac{1}{2}))}-1\right|\leq C\frac{|t|}{\xi^{2}+\frac{1}{2}+\sqrt{\xi^{4}+\xi^{2}}}$,  we have
\begin{align}
\left|\int_{\SR}e^{ix\xi} e^{it(\xi^{2}+\frac{1}{2})}\left(e^{it(\sqrt{\xi^{4}+\xi^{2}}-(\xi^{2}+\frac{1}{2}))}-1\right)\right|&\leq C\left(\int_{|\xi|\leq 1}d\xi+\int_{|\xi|\geq1}\frac{|t|}{\xi^{2}}d\xi\right)\nonumber\\
&\leq C.\label{2.031}
\end{align}
From \eqref{2.029}-\eqref{2.031}, since $0<|t|\leq\delta\leq1$, we have
\begin{eqnarray}
&&\left|\int_{\SR}e^{ix\xi\pm it\xi\sqrt{1+\xi^{2}}}d\xi\right|\leq C(1+|t|^{-\frac{1}{2}})\leq C|t|^{-\frac{1}{2}}.\label{2.032}
\end{eqnarray}

This completes the proof of Lemma 2.8.
\end{proof}

\begin{Lemma}\label{lem2.9}
Let $N\in \mathbb{N}^{+}$, $0<|t|\leq N^{-1}$, $x\in[-1,1]$ and $\phi(\xi)=\xi\sqrt{\xi^{2}+1}$. Then, we have
\begin{eqnarray}
&&\left|\sum_{1\leq k\leq N}e^{it\phi(k)+kx}\right|\leq C|t|^{-\frac{1}{2}}.\label{2.033}
\end{eqnarray}
\end{Lemma}
\begin{proof}
We present two methods to prove Lemma 2.9. Inspired by Theorem 5.3 of \cite{KPV1991}, we present the first method to
prove Lemma 2.9. For $t>0$, we define
\begin{eqnarray}
&&g(v)=1,\,\Phi(v)=\frac{1}{2\pi}(t\phi(v)+vx),\,\Phi_{p}(v)=e^{2i\pi(\Phi(v)-pv)},\,1\leq v\leq N.\label{2.034}
\end{eqnarray}
Notice that
$$\Phi^{\prime\prime}(v)=\frac{1}{2\pi}t\phi^{\prime\prime}(v),\,\phi^{\prime\prime}(v)
=\frac{2v^{3}+3v}{(1+v^{2})^{\frac{3}{2}}},\,\phi^{\prime\prime\prime}
=\frac{3}{(1+v^{2})^{\frac{5}{2}}}.$$
Since $(\Phi(v)-pv)^{\prime\prime}=\Phi^{\prime\prime}(v)>0(t>0)$, this implies that
$(\Phi(v)-pv)^{\prime}$ is monotone. For $p\in\Z$, we assume that $\alpha_{p}$ satisfy
\begin{eqnarray}
&&\Phi^{\prime}(\alpha_{p})=\frac{1}{2\pi}\left[t\left(2\sqrt{1+\alpha_{p}^{2}}
-\frac{1}{\sqrt{1+\alpha_{p}^{2}}}+x\right)\right]=p-\frac{1}{2}.\label{2.035}
\end{eqnarray}
By using \eqref{2.031}, since $0<t<N^{-1}$, for $1\leq \alpha_{p}\leq N$, we have
\begin{eqnarray*}
&&|p-\frac{1}{2}|=|\Phi^{\prime}(\alpha_{p})|\leq C.
\end{eqnarray*}
Thus, we have
\begin{eqnarray}
&&|p|\leq C.\label{2.036}
\end{eqnarray}
Since $(\Phi(v)-pv)^{\prime}$  and $g(k)$ are monotone and $0<t<N^{-1}$ and
 $|(\Phi(v)-pv)^{\prime}|=|\Phi^{\prime}(v)-p|\leq \frac{1}{2},\,(v\in(\alpha_{p},\alpha_{p+1}))$,
  by using Lemmas 2.7, 2.8, we have
\begin{align}
\left|\sum_{1\leq k\leq N}g(k)e^{i2\pi\Phi(k)}\right|&=\left|\sum_{1\leq k\leq N}g(k)\Phi_{p}(k)\right|
\leq \sum_{|p|\leq C}\left|\sum_{k=\alpha_{p}}^{\alpha_{p+1}}g(k)\Phi_{p}(k)\right|\nonumber\\
&\leq \sum_{|p|\leq C}\left|\sum_{k=\alpha_{p}}^{\alpha_{p+1}}g(k)\Phi_{p}(k)-
\int_{\alpha_{p}}^{\alpha_{p+1}}g(x)\Phi_{p}(x)dx\right|\nonumber\\
&\,\,+\sum_{|p|\leq C}\left|\int_{\alpha_{p}}^{\alpha_{p+1}}g(x)\Phi_{p}(x)dx\right|
\leq Ct^{-\frac{1}{2}}.\label{2.037}
\end{align}
For case $t<0$, by using a proof similar to case $t>0$, we have that \eqref{2.033}
is valid.

Now, we present the second method to prove Lemma 2.9
\begin{align}
\left|\sum_{1\leq k\leq N}e^{i(t\sqrt{k^{4}+k^{2}}+kx)}\right|&=
\left|\sum_{1\leq k\leq N}e^{ikx}\left(e^{it\sqrt{k^{4}+k^{2}}}-
e^{it(k^{2}+\frac{1}{2})}\right)+\sum_{k}e^{ikx}e^{itk^{2}}
\right|\nonumber\\
&\leq \left|\sum_{1\leq k\leq N}e^{ikx}\left(e^{it\sqrt{k^{4}+k^{2}}}-
e^{it(k^{2}+\frac{1}{2})}\right)\right|+\left|\sum_{1\leq k\leq N}
e^{ikx}e^{it(k^{2}+\frac{1}{2})}\right|\nonumber\\
&=J_{1}+J_{2}.\label{2.038}
\end{align}
For $J_{2}$, by using (5.9) page 60 of \cite{KPV1991}, we have
\begin{eqnarray}
&&|J_{2}|\leq C|t|^{-\frac{1}{2}}.\label{2.039}
\end{eqnarray}
For $J_{1}$, since $\left|e^{-it(-\sqrt{k^{4}+k^{2}}+k^{2}+\frac{1}{2})}-1\right|
\leq \frac{1}{4}\frac{|t|}{k^{2}+\frac{1}{2}+\sqrt{k^{4}+k^{2}}} $, $0<t\leq N^{-1}$, we have
\begin{align}
|J_{1}|&= \left|\sum_{1\leq k\leq N}e^{ikx}\left(e^{it\sqrt{k^{4}+k^{2}}}-
e^{it(k^{2}+\frac{1}{2})}\right)\right|\nonumber\\
&=\left|\sum_{1\leq k\leq N}e^{ikx}e^{itk^{2}}
\left(e^{-it(-\sqrt{k^{4}+k^{2}}+k^{2}+\frac{1}{2})}-1\right)
\right|\nonumber\\
&\leq \sum_{1\leq k\leq N}\frac{1}{4}\frac{|t|}{k^{2}+\frac{1}{2}+\sqrt{k^{4}+k^{2}}}\nonumber\\
&\leq C|t|\sum_{1\leq k\leq N}\frac{1}{k^{2}}\leq C|t|^{-\frac{1}{2}}.\label{2.040}
\end{align}

This completes the proof of Lemma 2.9.

\end{proof}

\begin{Lemma}\label{lem2.10}
Let $0<|t|\leq N^{-1}$, $x\in[-1,1]$ and $\phi(\xi)=\xi\sqrt{\xi^{2}+1}$. Then, we have
\begin{eqnarray}
&&\left|\sum_{-N\leq k\leq -1}e^{it\phi(k)+kx}\right|\leq C|t|^{-\frac{1}{2}}.\label{2.041}
\end{eqnarray}
\end{Lemma}
\begin{proof}
By using a proof similar to Lemma 2.9, we obtain that \eqref{2.041} is valid.

This completes the proof of Lemma 2.10.
\end{proof}

\begin{Lemma}\label{lem2.11}
Let $0<|t|\leq N^{-1}$, $x\in[-1,1]$ and $\phi(\xi)=\xi\sqrt{\xi^{2}+1}$. Then, we have
\begin{eqnarray}
&&\left|\sum_{ |k|\leq N,k\neq0}e^{it\phi(k)+kx}\right|\leq C|t|^{-\frac{1}{2}}.\label{2.042}
\end{eqnarray}
\end{Lemma}
\begin{proof}
By using Lemma 2.9 with Lemma 2.10, we have
\begin{align*}
\left|\sum_{|k|\leq N,k\neq0}e^{it\phi(k)+kx}\right|&\leq
\left|\sum_{1\leq k\leq N,k\neq0}e^{it\phi(k)+kx}\right|+
\left|\sum_{-N\leq k\leq -1}e^{it\phi(k)+kx}\right|\nonumber\\
&\leq C|t|^{-\frac{1}{2}}.
\end{align*}

This completes the proof of Lemma 2.11.

\end{proof}

\begin{Lemma}\label{lem2.12}
Let $0<|t|\leq N^{-1}$, $x\in[-1,1]$ and $\phi(\xi)=\xi\sqrt{\xi^{2}+1}$. Then, we have
\begin{eqnarray}
&&\left|\sum_{ |k|\leq N}e^{it\phi(k)+kx}\right|\leq C|t|^{-\frac{1}{2}}.\label{2.043}
\end{eqnarray}
\end{Lemma}
\begin{proof}
By using a proof similar to Lemma 2.11, we obtain that \eqref{2.043} is valid.

This completes the proof of Lemma 2.12.
\end{proof}

\begin{Lemma}\label{lem2.13}
Let $\lambda=(\lambda_{j})_{j}\in \ell^{\beta}$, $(f_j)_j$ be an orthonormal
system in $L^{2}(\R)$ and  $A:L^{2}(\mathbf{R})\rightarrow L_{x}^{q,\infty}L_{t}^{r}(\R^{2})$
 be a bounded linear operator for some $q > 2$ and $r\geq2$ and  $\beta\geq1$. Then,
\begin{eqnarray*}
&&\Big\|\sum_{j}\lambda_{j}|Af_{j}|^{2}
\Big\|_{L_{x}^{\frac{q}{2},\infty}L_{t}^{\frac{r}{2}}(\mathbf{R}^{2})}
\leq C\|\lambda\|_{\ell^{\beta}}
\end{eqnarray*}
hold if and only if  the following inequality
\begin{eqnarray*}
&&\|WAA^{\ast}\overline{W}\|_{\mathfrak{S}^{\beta^{\prime}}}\leq
 C\|W\|_{L_{x}^{q_{0},2}L_{t}^{r_{0}}}^{2}
\end{eqnarray*}
is valid, where $\frac{1}{q}+\frac{1}{q_0}=\frac{1}{2}$ and
 $\frac{1}{r}+\frac{1}{r_0}=\frac{1}{2}$ and  $W\in L_{t}^{q_{0},2}L_{x}^{r_{0}}(\R^{2})$.
\end{Lemma}

For the proof of Lemma 2.13, we refer the readers to Proposition
2.1 of \cite{BHLNS2019} and Lemma 3 of \cite{FS2017}.

\begin{Lemma}\label{lem2.14}
Let $\beta\geq1$, $\lambda=(\lambda_{j})_{j}\in \ell^{\beta}$, $(f_j)_j$
be an orthonormal system in $L^{2}(\mathbf{T})$ and  $A:L^{2}(\mathbf{T})
\rightarrow L_{t}^{p}L_{x}^{q}(\mathbf{T}^{2})$ be a bounded linear operator. Then
\begin{eqnarray*}
&&\Big\|\sum_{j}\lambda_{j}|Af_{j}|^{2}\Big\|_{L_{t}^{p}L_{x}^{q}(\mathbf{T}^{2})}
\leq C\|\lambda\|_{\ell^{\beta}}
\end{eqnarray*}
hold if and only if  the following inequality
\begin{eqnarray*}
&&\|WAA^{\ast}\overline{W}\|_{\mathfrak{S}^{\beta^{\prime}}}\leq
C\|W\|_{L_{t}^{2p^{\prime}}L_{x}^{2q^{\prime}}}^{2}
\end{eqnarray*}
is valid, where $\frac{1}{q}+\frac{1}{q^{\prime}}=\frac{1}{p}+\frac{1}{p^{\prime}}=1$ and  $W\in L_{t}^{2p^{\prime}}L_{x}^{2q^{\prime}}(\mathbf{T}^{2})$.
\end{Lemma}

For the proof of Lemma 2.14, we refer the readers to Lemma 3 of \cite{FS2017}.

\begin{Lemma}\label{lem2.15}
For $i=0,1$, let $A_{i}, B_{i}, C_{i}$ be Banach spaces and  $T$
be a bilinear operator such that
\begin{align*}
&T: \,A_{0}\times B_{0}\rightarrow C_{0},\\
&T:\, A_{0}\times B_{1}\rightarrow C_{1},\\
&T:\,A_{1}\times B_{1}\rightarrow C_{1}.
\end{align*}
Then,  for $\theta=\theta_{0}+\theta_{1}$ and $\frac{1}{p}+\frac{1}{q}\geq 1$, we have
$$T:\,(A_{0},A_{1})_{\theta_{0},pr}\times(B_{0},B_{1})_{\theta_{1},qr}\rightarrow (C_{0},C_{1})_{\theta,r}.$$
Here $0<\theta_{i}<\theta<1$ and $1\leq p,q,r\leq \infty.$
\end{Lemma}

For the proof of Lemma 2.15, we refer the readers to Exercise 5, Page 76 of \cite{BL1976}.

\begin{Lemma}\label{lem2.16}
Let $0<q_{0}\leq \infty, 0<q_{1}\leq \infty, s_{0}\neq s_{1}$, then for $\forall q\leq \infty$,\\
if $s_{0}\neq s_{1}$, we have
$$(\ell_{q_{0}}^{s_{0}},\ell_{q_{1}}^{s_{1}})_{\theta,q}=\ell_{q}^{s},$$
provided that $s=(1-\theta)s_{0}+\theta s_{1}$,\\
if $s_{0}=s_{1}=s$, we have
$$(\ell_{q_{0}}^{s_{0}},\ell_{q_{1}}^{s_{1}})_{\theta,q}=\ell_{q}^{s},$$
provided that $\frac{1}{q}=\frac{1-\theta}{q_{0}}+\frac{\theta}{q_{1}}$.

\end{Lemma}

For the proof of Lemma 2.16, we refer the readers to Theorem 5.6.1,
Page 122 of \cite{BL1976}.

\begin{Lemma}\label{lem2.17}
Suppose that $p_{0}, p_{1}, q_{0}, q_{1}$ and $q$ are positive,
possibly infinite,
 numbers and  $0<\theta<1$. Then,\\
if $p_{0}\neq p_{1}$, we have
$$(L^{p_{0},q_{0}},L^{p_{1},q_{1}})_{\theta,q}=L^{p,q},$$
provided that $\frac{1}{p}=\frac{1-\theta}{p_{0}}+\frac{\theta}{p_{1}}$,\\
if $p_{0}= p_{1}=p$, we have
$$(L^{p_{0},q_{0}},L^{p_{1},q_{1}})_{\theta,q}=L^{p,q},$$
provided that $\frac{1}{q}=\frac{1-\theta}{q_{0}}+\frac{\theta}{q_{1}}$.
\end{Lemma}

For the proof of Lemma 2.17, we refer the readers to Theorem 5.3.1,
 Page 113 of \cite{BL1976}.

\begin{Lemma}\label{lem2.18}(Three line Theorem)
Let $f(z)$ be a function, holomorphic in the strip $a\leq Re z\leq b$. If
\begin{eqnarray}
&&|f(a+iy)|\leq C_{1},\,\,|f(b+iy)|\leq C_{2},\,(-\infty<y<+\infty)\label{2.044}
\end{eqnarray}
and
\begin{eqnarray}
&&\ln|f(z)|\leq N exp(k|Im z|),\, (a< Re z<b),\label{2.045}
\end{eqnarray}
where $0\leq k<\frac{\pi}{b-a}$, then for any $(a<\theta<b),$ we have
\begin{eqnarray}
&&|f(\theta+iy)|\leq C_{1}^{1-t_{\theta}}C_{2}^{t_{\theta}}\,\,(-\infty<y<+\infty),
\label{2.046}
\end{eqnarray}
where $t_{\theta}=\frac{\theta-a}{b-a}$.
\end{Lemma}

For the proof of Lemma 2.18, we refer the readers to Page 136 of \cite{GK1965}.

\begin{Lemma}\label{lem2.19}(Extension of Three line Theorem)
Let $f(z)$ be a continuous function on the strip $0\leq Re z\leq a$
 and is analytic on the strip $0< Re z< a$.  If
\begin{eqnarray}
&&\ln|f(z)|\leq N exp(k|Im z|),\, (0\leq Re z\leq a),\label{2.047}
\end{eqnarray}
where $0\leq k<\frac{\pi}{a}$, then for any $\theta\in (0,a)$, we have
\begin{eqnarray}
&&|f(\theta)|\leq exp\left\{\frac{\sin\pi \theta}{2}\int_{-\infty}^{\infty}
\left[\frac{\ln|f(iy)|}{\cosh\pi y-\cos\pi\theta }+
\frac{\ln|f(a+iy)|}{\cosh\pi y-\cos\pi\theta}\right]dy\right\}.\label{2.048}
\end{eqnarray}
\end{Lemma}

For the proof of Lemma 2.19, we refer the readers to Lemma 4.2, Page 206 of \cite{SW1971}.

\section{ Noncommutative-commutative interpolation theorems}
\setcounter{equation}{0}

\setcounter{Theorem}{0}

\setcounter{Lemma}{0}

\setcounter{section}{3}
In this section, we present noncommutative-commutative Riesz-Throin interpolation
 theorems and noncommutative-commutative Stein interpolation theorems. Lemma 3.2 is used to prove Theorem 1.1.
 Lemma 3.4 is used to prove Lemma 9.1.

\begin{Lemma}(Noncommutative-commutative Riesz-Throin interpolation theorem)\label{lem3.1}
Let  $\mathfrak{H}$ be a separable Hilbert space and $f\in L^{p_{1}}\cap L^{p_{2}}$,
$g\in L^{q_{1}}\cap L^{q_{2}}$, $T_{z}(z\in\mathbb{C})$
 be an operator-function holomorphic in the strip $a\leq Re z\leq b$
 and $fT_{z}g\in\mathfrak{R}$, where $\mathfrak{R}$ is the set of all bounded linear operators. If
\begin{eqnarray}
&&\|fT_{a+iy}g\|_{\mathfrak{S}^{r_{1}}}\leq M_{0}\|f\|_{L^{p_{1}}}\|g\|_{L^{q_{1}}},
\,(-\infty<y<+\infty),\label{3.01}\\
&&\|fT_{b+iy}g\|_{\mathfrak{S}^{r_{2}}}\leq M_{1}\|f\|_{L^{p_{2}}}\|g\|_{L^{q_{2}}},
\,(-\infty<y<+\infty),\label{3.02}
\end{eqnarray}
and for $h_{1},h_{2}\in \mathfrak{H}$
\begin{eqnarray}
&&\ln|(fT_{z}gh_{1},h_{2})|\leq N_{h_{1},h_{2}}exp(k_{h_{1},h_{2}}|Im z|),\, (a<Re z<b), \label{3.03}
\end{eqnarray}
where $0\leq k_{h_{1},h_{2}} <\frac{\pi}{b-a}$, $N_{h_{1},h_{2}}<\infty$.
Then, we have
\begin{eqnarray}
&&\|fT_{x+iy}g\|_{\mathfrak{S}^{r}}\leq M_{0}^{1-t_{x}}M_{1}^{t_{x}}\|f\|_{L^{p}}\|g\|_{L^{q}},
\,(a<x<b,\,-\infty<y<+\infty),\label{3.04}
\end{eqnarray}
where
\begin{eqnarray*}
&&\frac{1}{r}=\frac{1-t_{x}}{r_{1}}+\frac{t_{x}}{r_{2}},\,\frac{1}{p}=
\frac{1-t_{x}}{p_{1}}+\frac{t_{x}}{p_{2}},
\,\frac{1}{q}=\frac{1-t_{x}}{q_{1}}+\frac{t_{x}}{q_{2}},\,t_{x}=\frac{x-a}{b-a}.
\end{eqnarray*}

\end{Lemma}

\begin{proof}
Inspired by Theorem 13.1, Page 137 of \cite{GK1965}, we prove Lemma 3.1.

We denote $\alpha_{j}=\frac{1}{p_{j}}, \gamma_{j}=\frac{1}{q_{j}}$, $(j=1,2)$,
$\alpha_{3}=\frac{1}{p}$, $\gamma_{3}=\frac{1}{q}$ and for $\forall z\in\mathbb{C}$, we define
\begin{eqnarray}
&&\alpha(z)=\frac{b-z}{b-a}\alpha_{1}+\frac{z-a}{b-a}\alpha_{2},\,
\gamma(z)=\frac{b-z}{b-a}\gamma_{1}+\frac{z-a}{b-a}\gamma_{2},\label{3.05}\\
&&\alpha+\beta z=r^{\prime}\left(\frac{b-z}{b-a}\frac{1}{r^{\prime}_{1}}
+\frac{z-a}{b-a}\frac{1}{r^{\prime}_{2}}\right),\label{3.06}
\end{eqnarray}
where $\frac{1}{r_{j}}+\frac{1}{r^{\prime}_{j}}=1$, $(j=1,2)$. To prove \eqref{3.04},
it suffices to prove
\begin{eqnarray}
&&\|fT_{x+iy}g\|_{\mathfrak{S}^{r}}\leq M_{0}^{1-t_{x}}M_{1}^{t_{x}},\label{3.07}
\end{eqnarray}
where $f,g\in \mathscr{S}$ and $\|f\|_{L^{p}}=1=\|g\|_{L^{q}}$.

Without loss of generality, we can assume that $f=\sum\limits_{j=1}^{m}a_{j}\chi_{E_{j}}$,
 $g=\sum\limits_{l=1}^{n}b_{l}\chi_{F_{l}}$ and $\|f\|_{L^{p}}=1=\|g\|_{L^{p}}$.
  Let $a_{j}=|a_{j}|e^{i\theta_{j}}$, $b_{l}=|b_{l}|e^{i\phi_{l}}$. For $z\in\mathbb{C}$, we define
\begin{eqnarray}
&&f_{z}=\sum_{j=1}^{m}|a_{j}|^{\frac{\alpha(z)}{\alpha_{3}}}e^{i\theta_{j}}\chi_{E_{j}},\label{3.08}\\
&&g_{z}=\sum_{l=1}^{n}|b_{l}|^{\frac{\gamma(z)}{\gamma_{3}}}e^{i\phi_{l}}\chi_{F_{l}}.\label{3.09}
\end{eqnarray}
We assume that  $K$ is  an arbitrary finite-dimensional operator and $K=UH$ be its polar
decomposition, where $U$ is a partially isometric operator.
 Let $\|K\|_{\mathfrak{S}^{r^{\prime}}}=1$, $\frac{1}{r}+\frac{1}{r^{\prime}}=1$.  We define
\begin{eqnarray}
&&F(z)=Tr[f_{z}T_{z}g_{z}UH^{\alpha+\beta z}],\, (a\leq Re z\leq b).\label{3.010}
\end{eqnarray}
Obviously,
\begin{eqnarray*}
&&H=\sum_{j=1}^{r(H)}\lambda_{j}|\varphi_{j}\rangle\langle\varphi_{j}|
\end{eqnarray*}
which is  the spectral decomposition of the operator $H$.  Here
$\varphi_{j}$ and $\lambda_{j}$ are eigenvalues and eigenvectors of $H$
and   $r(H)$ denotes the dimension of $H$. Then, according to spectral mapping theorem
 of Page 227 in \cite{Y}, we have
\begin{eqnarray*}
H^{\alpha+\beta z}=\sum_{j=1}^{r(H)}\lambda_{j}^{\alpha+\beta z}|
\varphi_{j}\rangle\langle U\varphi_{j}|.
\end{eqnarray*}
From the above equality and (\ref{3.010}), we have
\begin{eqnarray}
&&f_{z}T_{z}g_{z}UH^{\alpha+\beta z}=
\sum_{j=1}^{r(H)}\lambda_{j}^{\alpha+\beta z}|\varphi_{j}\rangle
\langle f_{z}T_{z}g_{z}U\varphi_{j}|\label{3.011}
\end{eqnarray}
and
\begin{eqnarray}
&&F(z)=\sum_{j=1}^{r(H)}\lambda_{j}^{\alpha+\beta z}(f_{z}T_{z}g_{z}U\varphi_{j},\varphi_{j}).\label{3.012}
\end{eqnarray}
From \eqref{3.08} and \eqref{3.09}, we have
\begin{eqnarray}
&&f_{z}T_{z}g_{z}=\sum\limits_{l=1}^{n}\sum\limits_{j=1}^{m}|a_{j}|^{\frac{\alpha(z)}{\alpha_{3}}}
|b_{l}|^{\frac{\gamma(z)}{\gamma_{3}}}
e^{i\theta_{j}}e^{i\phi_{l}}\chi_{E_{j}}T_{z}\chi_{F_{l}}.\label{3.013}
\end{eqnarray}
Since $T_{z}$ is holomorphness, by using \eqref{3.012} and \eqref{3.013}, we have
 $F(z)$ is holomorphness in the strip $a\leq Re z\leq b$. From \eqref{3.03},
 \eqref{3.012} and \eqref{3.013}, we have
\begin{align}
\ln |F(z)|&=\ln\left|\sum_{j=1}^{r(H)}\lambda_{j}^{\alpha+\beta z}\sum\limits_{l=1}^{n}
\sum\limits_{j=1}^{m}|a_{j}|^{\frac{\alpha(z)}{\alpha_{3}}}
|b_{l}|^{\frac{\gamma(z)}{\gamma_{3}}}
e^{i\theta_{j}}e^{i\phi_{l}}(\chi_{E_{j}}T_{z}\chi_{F_{l}}U\varphi_{j},\varphi_{j})\right|\nonumber\\
&\leq\ln\left[\left|\sum_{j=1}^{r(H)}\lambda_{j}^{\alpha+\beta z}\sum\limits_{l=1}^{n}
\sum\limits_{j=1}^{m}|a_{j}|^{\frac{\alpha(z)}{\alpha_{3}}}
|b_{l}|^{\frac{\gamma(z)}{\gamma_{3}}}\nonumber\right|
\left|(\chi_{E_{j}}T_{z}\chi_{F_{l}}U\varphi_{j},\varphi_{j})\right|\right]\nonumber\\
&\leq  \ln\left[\left|\sum_{j=1}^{r(H)}\lambda_{j}^{\alpha+\beta z}\sum\limits_{l=1}^{n}
\sum\limits_{j=1}^{m}|a_{j}|^{\frac{\alpha(z)}{\alpha_{3}}}
|b_{l}|^{\frac{\gamma(z)}{\gamma_{3}}}\right| e^{N_{j}exp(k_{j}Im z)}\right]\nonumber\\
&\leq \ln (M)+N_{\max}exp(k_{\max} Im z)\leq 2N_{\max} exp(k_{\max} Im z),\label{3.014}
\end{align}
where $N=\max\left\{\ln M, N_{\max}\right\}$, $M=\sum\limits_{j=1}^{k}|\lambda_{j}|^{\alpha+\beta Rez}
\sum\limits_{l=1}^{n}\sum\limits_{j=1}^{m}|a_{j}|^{Re \frac{\alpha(z)}{\alpha_{3}}}|b_{l}|^{Re
\frac{\gamma(z)}{\gamma_{3}}}$, $N_{\max}=\max\{N_{j}\}$, $k_{\max}=\max\{k_{j}\}$. From
\eqref{3.014}, we have that $F(z)$  satisfies \eqref{2.045}.

\noindent We claim
\begin{eqnarray}
&&|F(a+iy)|\leq M_{0},\label{3.015}\\
&&|F(b+iy)|\leq M_{1}.\label{3.016}
\end{eqnarray}
For \eqref{3.015}, by using \eqref{3.05}, we have
\begin{align}
\alpha(a+iy)&=\frac{b-(a+iy)}{b-a}\alpha_{1}+\frac{a+iy-a}{b-a}\alpha_{2}\nonumber\\
&=\frac{1}{p_{1}}+\frac{iy}{b-a}\left(\frac{1}{p_{2}}-\frac{1}{p_{1}}\right)\nonumber\\
\gamma(a+iy)&=\frac{b-(a+iy)}{b-a}\gamma_{1}+\frac{a+iy-a}{b-a}\gamma_{2}\nonumber\\
&=\frac{1}{q_{1}}+\frac{iy}{b-a}\left(\frac{1}{q_{2}}-\frac{1}{q_{1}}\right).\label{3.017}
\end{align}
When $p_{1},q_{1}\neq+\infty$, from \eqref{3.08}, \eqref{3.09} and \eqref{3.017}, we have
\begin{eqnarray}
&&|f_{a+iy}|^{p_{1}}=\left|e^{i\arg f}|f|^{\frac{iyp}{b-a}\left(\frac{1}{p_{2}}-\frac{1}{p_{1}}\right)}
|f|^{\frac{p}{p_{1}}}\right|^{p_{1}}
=|f|^{p},\label{3.018}\\
&&|g_{a+iy}|^{q_{1}}=\left|e^{i\arg g}|g|^{\frac{iyp}{b-a}\left(\frac{1}{q_{2}}-\frac{1}{q_{1}}\right)}
|g|^{\frac{q}{q_{1}}}\right|^{q_{1}}
=|g|^{q}.\label{3.019}
\end{eqnarray}
When $p_{1},q_{1}=+\infty$, we have
\begin{eqnarray}
&&|f_{a+iy}|=\left|e^{i\arg f}|f|^{\frac{iyp}{b-a}\frac{1}{p_{2}}}\right|=1,\,|g_{a+iy}|=\left|e^{i\arg g}|g|^{\frac{iyp}{b-a}\frac{1}{q_{2}}}\right|=1.\label{3.020}
\end{eqnarray}
From \eqref{3.01}, \eqref{3.010}, \eqref{3.018}-\eqref{3.020} and H\"older
inequality in Schatten space, which can be found in  Theorem 3.4 of \cite{Y1975}, we have
\begin{align}
|F(a+iy)|&\leq \|f_{a+iy}T_{a+iy}g_{a+iy}\|_{\mathfrak{S}^{r_{1}}}
\|H^{\alpha+a\beta+i\beta y}\|_{\mathfrak{S}^{r^{\prime}_{1}}}\nonumber\\
&\leq M_{0}\|f_{a+iy}\|_{L^{p_{1}}}\|g_{a+iy}\|_{L^{q_{1}}}
\|H^{\alpha+a\beta}\|_{\mathfrak{S}^{r^{\prime}_{1}}}\nonumber\\
&\leq M_{0}\|f\|_{L^{p}}^{\frac{p}{p_{1}}}\|g\|_{L^{q}}^{\frac{q}{q_{1}}}
\|H\|_{\mathfrak{S}^{r^{\prime}_{1}(\alpha+a\beta)}}^{\alpha+a\beta}
\nonumber\\
&= M_{0}\|K\|_{\mathfrak{S}^{r^{\prime}}}^{\alpha+a\beta}=M_{0},\label{3.021}
\end{align}
where we used the following fact
\begin{eqnarray*}
&&r^{\prime}_{1}(\alpha+a\beta)=r^{\prime}
\end{eqnarray*}
and
\begin{eqnarray*}
&&\|H^{l}\|_{\mathfrak{S}^{r}}=\|H\|_{\mathfrak{S}^{lr}}^{l}(H\geq0).
\end{eqnarray*}
From \eqref{3.021}, we have that \eqref{3.015} is valid.
For \eqref{3.016}, by using \eqref{3.05}, we have
\begin{align}
\alpha(b+iy)&=\frac{b-(b+iy)}{b-a}\alpha_{1}+\frac{b+iy-a}{b-a}\alpha_{2}\nonumber\\
&=\frac{1}{p_{2}}+\frac{iy}{b-a}\left(\frac{1}{p_{2}}-\frac{1}{p_{1}}\right)\nonumber\\
\gamma(b+iy)&=\frac{b-(b+iy)}{b-a}\gamma_{1}+\frac{b+iy-a}{b-a}\gamma_{2}\nonumber\\
&=\frac{1}{q_{2}}+\frac{iy}{b-a}\left(\frac{1}{q_{2}}-\frac{1}{q_{1}}\right).\label{3.022}
\end{align}
When $p_{2},q_{2}\neq+\infty$, from \eqref{3.08}, \eqref{3.09} and \eqref{3.022}, we have
\begin{eqnarray}
&&|f_{b+iy}|^{p_{2}}=\left|e^{i\arg f}|f|^{\frac{iyp}{b-a}\left(\frac{1}{p_{2}}-\frac{1}{p_{1}}\right)}
|f|^{\frac{p}{p_{2}}}\right|^{p_{2}}
=|f|^{p},\label{3.023}\\
&&|g_{b+iy}|^{q_{2}}=\left|e^{i\arg g}|g|^{\frac{iyp}{b-a}\left(\frac{1}{q_{2}}-\frac{1}{q_{1}}\right)}
|g|^{\frac{q}{q_{2}}}\right|^{q_{2}}
=|g|^{q}.\label{3.024}
\end{eqnarray}
When $p_{2},q_{2}=+\infty$, we have
\begin{eqnarray}
&&|f_{a+iy}|=\left|e^{i\arg f}|f|^{-\frac{iyp}{b-a}\frac{1}{p_{1}}}\right|=1,\,|g_{a+iy}|=\left|e^{i\arg g}|g|^{-\frac{iyp}{b-a}\frac{1}{q_{1}}}\right|=1.\label{3.025}
\end{eqnarray}
From \eqref{3.023}-\eqref{3.025}, by using a proof similar to \eqref{3.021}, we have
\begin{align}
&|F(b+iy)|\leq M_{1}.\label{3.026}
\end{align}
From \eqref{3.026}, we have that \eqref{3.016} is valid.

\noindent Since $F(z)$  satisfies \eqref{2.045}, by using \eqref{3.015},
\eqref{3.016} and Lemma 2.18,  we have
\begin{eqnarray}
&&|F(x+iy)|\leq M_{0}^{1-t_{x}}M_{1}^{t_{x}},\,(-\infty<y<+\infty).\label{3.027}
\end{eqnarray}
Since $\alpha+\beta x=1$ and $\alpha(x)=\frac{1}{p}$, $\gamma(x)=\frac{1}{q}$, for $y=0$,
by using \eqref{3.027}, we have
\begin{eqnarray}
&&|F(x)|=|Tr[fT_{x}gK]|\leq M_{0}^{1-t_{x}}M_{1}^{t_{x}}.\label{3.028}
\end{eqnarray}
From \eqref{3.028}, Lemma 12.1, Page 132 of \cite{GK1965} and $fT_{x+iy}g\in\mathfrak{S}^{r}$, we have
\begin{eqnarray}
&&\|fT_{x+iy}g\|_{\mathfrak{S}^{r}}\leq M_{0}^{1-t_{x}}M_{1}^{t_{x}}.\label{3.029}
\end{eqnarray}

This completes the proof of Lemma 3.1.
\end{proof}

\begin{Lemma}(Noncommutative-commutative Riesz-Throin interpolation theorem related to
 the mixed Lebesgue spaces)\label{lem3.2}
Let $\mathfrak{H}$ be a separable Hilbert space and   $f\in L^{P_{1}}(\R^{d})\cap
 L^{P_{2}}(\R^{d})(d\geq2)$, $g\in L^{Q_{1}}(\R^{d})\cap L^{Q_{2}}(\R^{d})(d\geq2)$,
 $P=(p_{1},p_{2},\ldots,p_{d})$, $Q=(q_{1},q_{2},\ldots,q_{d})$,  $T_{z}(z\in\mathbb{C})$
  be an operator-function holomorphic in the strip $a\leq Re z\leq b$ and
   $fT_{z}g\in\mathfrak{R}$,
   where $\mathfrak{R}$ is the set of all all bounded linear operators. If
\begin{eqnarray}
&&\|fT_{a+iy}g\|_{\mathfrak{S}^{r_{1}}}\leq M_{0}\|f\|_{L^{P_{1}}}\|g\|_{L^{Q_{1}}},\,(-\infty<y<+\infty),\label{3.030}\\
&&\|fT_{b+iy}g\|_{\mathfrak{S}^{r_{2}}}\leq M_{1}\|f\|_{L^{P_{2}}}\|g\|_{L^{Q_{2}}},\,(-\infty<y<+\infty),\label{3.031}
\end{eqnarray}
and for $h_{1},h_{2}\in\mathfrak{H}$
\begin{eqnarray}
&&\ln|(fT_{z}gh_{1},h_{2})|\leq N_{h_{1},h_{2}}exp(k_{h_{1},h_{2}}|Im z|),\,
 (a<Re z<b), \label{3.032}
\end{eqnarray}
where $0\leq k_{h_{1},h_{2}} <\frac{\pi}{b-a}$, $N_{h_{1},h_{2}}<\infty$.
Then, we have
\begin{eqnarray}
&&\|fT_{\theta+iy}g\|_{\mathfrak{S}^{r}}\leq M_{0}^{1-t_{\theta}}
M_{1}^{t_{\theta}}\|f\|_{L^{P}}\|g\|_{L^{P}},\,
(a<\theta<b,\,-\infty<y<+\infty),\label{3.033}
\end{eqnarray}
where
\begin{eqnarray*}
&&\|f\|_{L^{P}}=\|f\|_{(p_{1},p_{2},\ldots,p_{d})}=\left(\int_{\mathbf{R}}
\ldots\left(\int_{\mathbf{R}}
|f(x_{1},x_{2},\ldots,x_{d})|^{p_{1}}dx_{1}\right)
^{\frac{p_{2}}{p_{1}}}\ldots dx_{n}\right)^{\frac{1}{p_{d}}},\\
&&\frac{1}{r}=\frac{1-t_{\theta}}{r_{1}}+\frac{t_{\theta}}{r_{2}},\,\frac{1}{P}
=\frac{1-t_{\theta}}{P_{1}}+\frac{t_{\theta}}{P_{2}}
,\,\frac{1}{Q}=\frac{1-t_{\theta}}{Q_{1}}+\frac{t_{\theta}}{Q_{2}},\,t_{\theta}
=\frac{\theta-a}{b-a}.
\end{eqnarray*}

\end{Lemma}

\begin{proof}
Inspired by Theorem 1, Page 313 of \cite{BP1961} and Theorem 13.1,
Page 137 of \cite{GK1965}, we prove Lemma 3.2.
Let $M=\left(\R^{d},\,\prod\limits_{j=1}^{d}m_{j}\right)$ be the Lebesgue measure,  where
$m_{j}(1\leq j\leq d)$. We denote $H(M)$ by
\begin{eqnarray*}
H(M)=\left\{f|f=\sum_{\nu_{1}=1}^{n_{1}}\ldots\sum_{\nu_{d}=1}^{n_{d}} c_{\nu}\chi_{A_{\nu}},\,A_{\nu_{1}}\times \ldots\times A_{\nu_{d}}=A_{\nu}\subset I_{m}=A_{1m}\times\ldots\times A_{dm}\right\},
\end{eqnarray*}
where $\nu=(\nu_{1},\nu_{2},\ldots,\nu_{d})$,  $1\leq i\leq d$,
$c_{\nu}=c(\nu_{1},\nu_{2},\ldots,\nu_{d})$, $m_{i}(A_{im})<\infty,\,
\cup_{m=1}^{\infty}I_{m}=\R^{d}$,  $A_{im}\subset\mathbf{R}$  and $\chi_{A_{\nu}}$
 is the characteristic function.

\noindent Let $\psi,\phi\in H(M)$ such that $\|\psi\|_{L^{P}}=\|\phi\|_{L^{Q}}=1$.
 We assume that  $\|f\|_{L^{P}}=\|g\|_{L^{Q}}=1$.  To prove \eqref{3.033}, it
 suffices to prove
\begin{eqnarray}
&&\|fT_{\theta+iy}g\|_{\mathfrak{S}^{r}}\leq M_{0}^{1-t_{\theta}}M_{1}^{t_{\theta}}.\label{3.034}
\end{eqnarray}
We denote
\begin{eqnarray*}
&&P_{i}=(p_{1i},p_{2i},\ldots,p_{di})=\left(\frac{1}{\alpha_{1i}},\frac{1}{\alpha_{2i}},
\ldots,\frac{1}{\alpha_{di}}\right),\,(i=1,2),\\
&&Q_{i}=(q_{1i},q_{2i},\ldots,q_{di})=\left(\frac{1}{\gamma_{1i}},\frac{1}{\gamma_{2i}},
\ldots,\frac{1}{\gamma_{di}}\right),\,(i=1,2),\\
&&P=(p_{1},p_{2},\ldots,p_{d})=\left(\frac{1}{\alpha_{1}},\frac{1}{\alpha_{2}},\ldots,
\frac{1}{\alpha_{d}}\right),\\
&&Q=(q_{1},q_{2},\ldots,q_{d})=\left(\frac{1}{\gamma_{1}},\frac{1}{\gamma_{2}},\ldots,
\frac{1}{\gamma_{d}}\right).
\end{eqnarray*}
For $z\in\mathbb{C}$, we define
\begin{eqnarray}
\alpha_{j}(z)=\frac{b-z}{b-a}\alpha_{j1}+\frac{z-a}{b-a}\alpha_{j2},
\,\gamma_{j}(z)=\frac{b-z}{b-a}\gamma_{j1}+\frac{z-a}{b-a}\gamma_{j2} ,\,j=1,2,\ldots,d.\label{3.035}
\end{eqnarray}
From \eqref{3.035}, we have $\alpha_{j}(x)=\alpha_{j}$,
$\gamma_{j}(x)=\gamma_{j}$, $j=1,2,\ldots,d$. For $1\leq i\leq d-1$,    we  define
\begin{align}
f_{i}&=f_{i}(x_{i+1},x_{i+2},\ldots,x_{d})\nonumber\\
&=\left[\|\psi(x_{1},x_{2},\ldots,x_{i},x_{i+1},\ldots,x_{d})\|_{(p_{1},p_{2},\ldots,p_{i})}\right]
^{\frac{\alpha_{i+1}(z)}{\alpha_{i+1}}-\frac{\alpha_{i}(z)}{\alpha_{i}}},\label{3.036}
\end{align}
and for $1\leq j\leq d-1$,
\begin{align}
g_{j}&=g_{j}(y_{j+1},y_{j+2},\ldots,y_{d})\nonumber\\
&=\left[\|\phi(y_{1},y_{2},\ldots,y_{j},y_{j+1},\ldots,y_{d})\|_{(q_{1},q_{2},\ldots,q_{j})}\right]
^{\frac{\gamma_{j+1}(z)}{\gamma_{j+1}}-\frac{\gamma_{j}(z)}{\gamma_{j}}}.\label{3.037}
\end{align}
From \eqref{3.036} and \eqref{3.037}, we define
\begin{eqnarray}
&&F_{z}(x_{1},x_{2},\ldots,x_{d})=|\psi|^{\frac{\alpha_{1}(z)}{\alpha_{1}}}(sign\psi)
\prod\limits_{i=1}^{d-1}f_{i},\label{3.038}\\
&&G_{z}(y_{1},y_{2},\ldots,y_{d})=|\phi|^{\frac{\gamma_{1}(z)}{\gamma_{1}}}(sign\phi)
\prod\limits_{j=1}^{d-1}g_{j}.\label{3.039}
\end{eqnarray}
From \cite{BP1961}, we have that $F(z), G(z)\in H(M)$.

\noindent We assume that $K$ is  an arbitrary finite-dimensional operator and $K=UH$
its polar decomposition, where $U$ is a partially isometric operator.
Let $\|K\|_{\mathfrak{S}^{r^{\prime}}}=1$ and  $\frac{1}{r}+\frac{1}{r^{\prime}}=1$.
We define
\begin{eqnarray}
&&\Phi(z)=Tr[F_{z}T_{z}G_{z}UH^{\alpha+\beta z}],\, (a\leq Re z\leq b),\label{3.040}
\end{eqnarray}
where
\begin{eqnarray}
&&\alpha+\beta z=r^{\prime}\left(\frac{b-z}{b-a}\frac{1}{r^{\prime}_{1}}+\frac{z-a}{b-a}
\frac{1}{r^{\prime}_{2}}\right)\label{3.041}
\end{eqnarray}
and  $\frac{1}{r_{j}}+\frac{1}{r^{\prime}_{j}}=1$, $(j=1,2)$.

\noindent Obviously,
\begin{eqnarray*}
&&H=\sum_{l=1}^{r(H)}\lambda_{l}|\varphi_{l}\rangle\langle\varphi_{l}|
\end{eqnarray*}
be the spectral decomposition of the operator $H$.  Here
$\varphi_{l}$ and $\lambda_{l}$ are eigenvalues and eigenvectors of $H$
and  and $r(H)$ denotes the dimension of $H$. Then, according to spectral mapping
 theorem of Page 227 in \cite{Y}, we have
\begin{eqnarray*}
H^{\alpha+\beta z}=\sum_{l=1}^{r(H)}\lambda_{l}^{\alpha+\beta z}|\varphi_{l}
\rangle\langle U\varphi_{l}|.
\end{eqnarray*}
From the above equality and (\ref{3.040}), we have
\begin{eqnarray}
&&F_{z}T_{z}G_{z}UH^{\alpha+\beta z}=\sum_{l=1}^{r(H)}\lambda_{j}^{\alpha+\beta z}
|\varphi_{l}\rangle\langle F_{z}T_{z}G_{z}U\varphi_{l}|
\label{3.042}
\end{eqnarray}
and
\begin{eqnarray}
&&\Phi(z)=\sum_{l=1}^{r(H)}\lambda_{l}^{\alpha+\beta z}(F_{z}T_{z}
G_{z}U\varphi_{l},\varphi_{l}).\label{3.043}
\end{eqnarray}
From \eqref{3.032}, \eqref{3.043} and \cite{BP1961},  since $F(z), G(z)\in H(M)$,
we have that $\Phi(z)$ is holomorphness in the strip $a\leq Re z\leq b$
and satisfies \eqref{2.045}.

We claim
\begin{eqnarray}
&&\|F_{a+iy}\|_{L^{P_{1}}}=\|\psi\|_{L^{P}}^{\frac{\alpha_{d}(a)}{\alpha_{d}}}
=\|\psi\|_{(p_{1},p_{2},\ldots,p_{d-1},p_{d})}^{\frac{\alpha_{d}(a)}{\alpha_{d}}}=1.\label{3.044}
\end{eqnarray}
We prove \eqref{3.044} by induction, for case $d=2$, by using \eqref{3.036} and
 \eqref{3.038}, we have
\begin{eqnarray}
&&F_{a+iy}(x_{1},x_{2})=|\psi|^{\frac{\alpha_{1}(a+iy)}{\alpha_{1}}}(sign\psi)
[\|\psi(x_{1},x_{2})\|_{p_{1}}]^{\frac{\alpha_{2}(a+iy)}{\alpha_{2}}-
\frac{\alpha_{1}(a+iy)}{\alpha_{1}}}.\label{3.045}
\end{eqnarray}
From \eqref{3.035}, we have
\begin{align}
\alpha_{1}(a+iy)&=\frac{b-(a+iy)}{b-a}\alpha_{11}+\frac{a+iy-a}{b-a}\alpha_{12}\nonumber\\
&=\alpha_{11}+\frac{iy}{b-a}\left(\alpha_{12}-\alpha_{11}\right).\label{3.046}
\end{align}
From \eqref{3.046}, we have
\begin{eqnarray}
&&\left||\psi|^{\frac{\alpha_{1}(a+iy)}{\alpha_{1}}}\right|^{p_{11}}=|\psi|^{p_{11}},\label{3.047}\\
&&|f_{1}|^{p_{21}}=\left|[\|\psi(x_{1},x_{2})\|_{p_{1}}]^{\frac{\alpha_{2}(a+iy)}{\alpha_{2}}
-\frac{\alpha_{1}(a+iy)}{\alpha_{1}}}\right|^{p_{21}}\nonumber\\
&&=[\|\psi(x_{1},x_{2})\|_{p_{1}}]^{p_{2}-\frac{p_{1}p_{21}}{p_{11}}}.\label{3.048}
\end{eqnarray}
By using \eqref{3.047} and \eqref{3.048}, we have
\begin{eqnarray}
&&\|F_{a+iy}\|_{L^{P_{1}}}=\|F_{a+iy}\|_{(p_{11},p_{21})}=\left(\int_{\SR}|f_{1}|^{p_{21}}
\left(\int_{\SR}\left||\psi|^{\frac{\alpha_{1}(a+iy)}{\alpha_{1}}}\right|
^{p_{11}}
dx_{1}\right)^{\frac{p_{21}}{p_{11}}}d x_{2}\right)^{\frac{1}{p_{21}}}\nonumber\\
&&=\left(\int_{\SR}[\|\psi(x_{1},x_{2})\|_{p_{1}}]^{p_{2}-\frac{p_{1}p_{21}}{p_{11}}}
\|\psi(x_{1},x_{2})\|_{p_{1}}^{\frac{p_{1}p_{21}}{p_{11}}}d x_{2}\right)^{\frac{1}{p_{21}}}\nonumber\\
&&=\left(\int_{\SR}\|\psi(x_{1},x_{2})\|_{p_{1}}^{p_{2}}d x_{2}\right)^{\frac{1}{p_{21}}}
=\|\psi\|_{(p_{1},p_{2})}^{\frac{p_{2}}{p_{21}}}
=\|\psi\|_{(p_{1},p_{2})}^{\frac{\alpha_{2}(a)}{\alpha_{2}}},\label{3.049}
\end{eqnarray}
where we used $\alpha_{2}(a)=\alpha_{21}$. From \eqref{3.049}, we have that
\eqref{3.044} is valid for case $d=2$.

Assume that \eqref{3.044} is valid for  case $d-1$, then we prove that
\eqref{3.044} is valid for case $d$.
\begin{align}
\|F_{a+iy}\|_{L^{P_{1}}}&=\left\|\|F_{a+iy}\|_{(p_{11},p_{21},\ldots,p_{(d-1)1})}\right\|_{p_{d1}}\nonumber\\
&=\left\|\|\psi\|_{(p_{1},p_{2},\ldots,p_{n-1})}^{\frac{\alpha_{d-1}(a)}{\alpha_{d-1}}}
\|\psi\|_{(p_{1},p_{2},\ldots,p_{d-1})}^{\frac{\alpha_{d}(a+iy)}{\alpha_{n}}-
\frac{\alpha_{d-1}(a+iy)}{\alpha_{d-1}}}\right\|_{p_{d1}}\nonumber\\
&=\left\|\|\psi\|_{(p_{1},p_{2},\ldots,p_{d-1})}^{\frac{\alpha_{d-1}(a)}{\alpha_{d-1}}}
\|\psi\|_{(p_{1},p_{2},\ldots,p_{d-1})}^{\frac{\alpha_{d}(a)}{\alpha_{d}}
-\frac{\alpha_{d-1}(a)}{\alpha_{d-1}}}\right\|_{p_{d1}}\nonumber\\
&=\left\|\|\psi\|_{(p_{1},p_{2},\ldots,p_{d-1})}^{\frac{\alpha_{d}(a)}{\alpha_{d}}}
\right\|_{p_{d1}}=\|\psi\|_{(p_{1},p_{2},\ldots,p_{d-1},p_{d})}^{\frac{\alpha_{d}(a)}{\alpha_{d}}}
,\label{3.050}
\end{align}
where we used the following fact
\begin{eqnarray*}
&&\left|\|\psi\|_{(p_{1},p_{2},\ldots,p_{d-1})}^{\frac{\alpha_{d}(a+iy)}{\alpha_{d}}
-\frac{\alpha_{d-1}(a+iy)}{\alpha_{d-1}}}\right|=
\|\psi\|_{(p_{1},p_{2},\ldots,p_{d-1})}^{\frac{p_{d}}{p_{d1}}-\frac{p_{d-1}}{p_{(d-1)1}}}
=\|\psi\|_{(p_{1},p_{2},\ldots,p_{d-1})}
^{\frac{\alpha_{d}(a)}{\alpha_{d}}-\frac{\alpha_{d-1}(a)}{\alpha_{d-1}}}.
\end{eqnarray*}
From \eqref{3.050}, we have that \eqref{3.044} is valid for case $d$.

By using a proof similar to \eqref{3.044}, we have
\begin{eqnarray}
&&\|F_{b+iy}\|_{L^{P_{2}}}=\|\psi\|_{L^{P}}^{\frac{\alpha_{d}(b)}{\alpha_{d}}}
=\|\psi\|_{(p_{1},p_{2},\ldots,p_{d-1},p_{d})}^{\frac{\alpha_{d}(b)}{\alpha_{d}}}=1,\label{3.051}\\
&&\|G_{a+iy}\|_{L^{Q_{1}}}=\|\phi\|_{L^{Q}}^{\frac{\gamma_{d}(a)}{\gamma_{d}}}
=\|\phi\|_{(q_{1},q_{2},\ldots,q_{d-1},q_{d})}^{\frac{\gamma_{d}(a)}{\gamma_{d}}}=1,\label{3.052}\\
&&\|G_{b+iy}\|_{L^{Q_{2}}}=\|\phi\|_{L^{Q}}^{\frac{\gamma_{d}(b)}{\gamma_{d}}}
=\|\phi\|_{(q_{1},q_{2},\ldots,q_{d-1},q_{d})}^{\frac{\gamma_{d}(b)}{\gamma_{d}}}=1.\label{3.053}
\end{eqnarray}
From \eqref{3.030}, \eqref{3.044}, \eqref{3.052} and H\"{o}lder inequality in Schatten
 space, which can be found in  Theorem 3.4 of \cite{Y1975}, we have
\begin{align}
|\Phi(a+iy)|&\leq \|F_{a+iy}T_{a+iy}G_{a+iy}\|_{\mathfrak{S}^{r_{1}}}
\|H^{\alpha+a\beta+i\beta y}\|_{\mathfrak{S}^{r^{\prime}_{1}}}\nonumber\\
&\leq M_{0}\|F_{a+iy}\|_{L^{P_{1}}}\|G_{a+iy}\|_{L^{Q_{1}}}
\|H^{\alpha+a\beta}\|_{\mathfrak{S}^{r^{\prime}_{1}}}\nonumber\\
&\leq M_{0}\|\psi\|_{L^{P}}^{\frac{\alpha_{d}(a)}{\alpha_{d}}}
\|\phi\|_{L^{Q}}^{\frac{\gamma_{d}(a)}{\gamma_{d}}}
\|H\|_{\mathfrak{S}^{r^{\prime}_{1}(\alpha+a\beta)}}^{\alpha+a\beta}
\nonumber\\
&= M_{0}\|K\|_{\mathfrak{S}^{r^{\prime}}}^{\alpha+a\beta}=M_{0},\label{3.054}
\end{align}
where we used the following fact
\begin{eqnarray*}
&&r^{\prime}_{1}(\alpha+a\beta)=r^{\prime}
\end{eqnarray*}
and
\begin{eqnarray*}
&&\|H^{l}\|_{\mathfrak{S}^{r}}=\|H\|_{\mathfrak{S}^{lr}}^{l}(H\geq0).
\end{eqnarray*}
From \eqref{3.031}, \eqref{3.051} and \eqref{3.053},  by using a proof
 similar to \eqref{3.054}, we have
\begin{align}
|\Phi(b+iy)|&\leq \|F_{b+iy}T_{b+iy}G_{b+iy}\|_{\mathfrak{S}^{r_{2}}}
\|H^{\alpha+b\beta+i\beta y}\|_{\mathfrak{S}^{r^{\prime}_{2}}}\nonumber\\
&\leq M_{1}\|F_{b+iy}\|_{L^{P_{2}}}\|G_{b+iy}\|_{L^{Q_{2}}}
\|H^{\alpha+b\beta}\|_{\mathfrak{S}^{r^{\prime}_{1}}}\nonumber\\
&\leq M_{1}\|\psi\|_{L^{P}}^{\frac{\alpha_{d}(b)}{\alpha_{d}}}
\|\phi\|_{L^{Q}}^{\frac{\gamma_{d}(b)}{\gamma_{d}}}
\|H\|_{\mathfrak{S}^{r^{\prime}_{2}(\alpha+b\beta)}}^{\alpha+b\beta}
\nonumber\\
&= M_{1}\|K\|_{\mathfrak{S}^{r^{\prime}}}^{\alpha+b\beta}=M_{1},\label{3.055}
\end{align}
where we used the following fact
\begin{eqnarray*}
&&r^{\prime}_{2}(\alpha+b\beta)=r^{\prime}.
\end{eqnarray*}
Since $\Phi(z)$  satisfies \eqref{2.045}, by using \eqref{3.054},
\eqref{3.055} and Lemma 2.18,  we have
\begin{eqnarray}
&&|\Phi(\theta+iy)|\leq M_{0}^{1-t_{\theta}}M_{1}^{t_{\theta}},
\,(-\infty<y<+\infty).\label{3.056}
\end{eqnarray}
Since $\alpha+\beta\theta=1$ and $\alpha_{j}(\theta)=\alpha_{j}$,
$\gamma_{j}(\theta)=\gamma_{j}$,
 for $y=0$, by using \eqref{3.038}, \eqref{3.039} and \eqref{3.056},
 we have
\begin{eqnarray}
&&|\Phi(\theta)|=|Tr[fT_{\theta}gK]|\leq M_{0}^{1-t_{\theta}}M_{1}^{t_{\theta}}.\label{3.057}
\end{eqnarray}
From \eqref{3.057}, Lemma 12.1, Page 132 of \cite{GK1965} and
$fT_{\theta+iy}g\in\mathfrak{S}^{r}$, we have
\begin{eqnarray}
&&\|fT_{\theta+iy}g\|_{\mathfrak{S}^{r}}\leq M_{0}^{1-t_{\theta}}
M_{1}^{t_{\theta}}.\label{3.058}
\end{eqnarray}

This completes the proof of Lemma 3.2.

\end{proof}

\begin{Lemma}(Noncommutative-commutative Stein interpolation theorem)\label{lem3.3}
Let $\mathfrak{H}$ be a separable Hilbert space and $f\in L^{p_{1}}
\cap L^{p_{2}}$, $g\in L^{q_{1}}\cap L^{q_{2}}$, $T_{z}(z\in\mathbb{C})$
be an operator-function holomorphic in the strip $0\leq Re z\leq a$ and $fT_{z}g\in\mathfrak{R}$,
 where $\mathfrak{R}$ is the set of all bounded linear operators. If
\begin{eqnarray}
&&\|fT_{iy}g\|_{\mathfrak{S}^{r_{1}}}\leq M_{0}(y)\|f\|_{L^{p_{1}}}\|g\|_{L^{q_{1}}},
\,(-\infty<y<+\infty),\label{3.059}\\
&&\|fT_{a+iy}g\|_{\mathfrak{S}^{r_{2}}}\leq M_{1}(y)\|f\|_{L^{p_{2}}}\|g\|_{L^{q_{2}}},
\,(-\infty<y<+\infty),\label{3.060}
\end{eqnarray}
and $M_{j}(y)$, $j=0,1$, such that
\begin{eqnarray}
&&\sup\limits_{y\in\SR}e^{-k|y|}\ln M_{j}(y)<\infty\label{3.061}
\end{eqnarray}
for some $0<k<\frac{\pi}{a}$,
and for $h_{1},h_{2}\in\mathfrak{H}$
\begin{eqnarray}
&&\ln|(fT_{z}gh_{1},h_{2})|\leq N_{h_{1},h_{2}}exp(k_{h_{1},h_{2}}|Im z|),\, (0< Re z< a), \label{3.062}
\end{eqnarray}
where $0\leq k_{h_{1},h_{2}} <\frac{\pi}{a}$, $N_{h_{1},h_{2}}<\infty$.
Then, we have
\begin{eqnarray}
&&\|fT_{x}g\|_{\mathfrak{S}^{r}}\leq M_{x}\|f\|_{L^{p}}\|g\|_{L^{q}},\,(0<x<a,\,-\infty<y<+\infty),\label{3.063}
\end{eqnarray}
where
\begin{eqnarray*}
&&M_{x}=exp\left\{\frac{\sin\pi x}{2}\int_{-\infty}^{\infty}
\left[\frac{\ln M_{0}(y)}{\cosh\pi y-\cos\pi x }+\frac{\ln M_{1}(y)}{\cosh\pi y-\cos\pi x}\right]dy\right\}\\
&&\frac{1}{r}=\frac{1-t_{x}}{r_{1}}+\frac{t_{x}}{r_{2}},\,\frac{1}{p}=\frac{1-t_{x}}{p_{1}}+\frac{t_{x}}{p_{2}},
\,\frac{1}{q}=\frac{1-t_{x}}{q_{1}}+\frac{t_{x}}{q_{2}},\,t_{x}=\frac{x}{a}.
\end{eqnarray*}

\end{Lemma}

\begin{proof}
Inspired by Theorem 4.1, Page 205 of \cite{SW1971} and Theorem 13.1, Page 137 of \cite{GK1965},
 by using Lemma 2.19 and a proof similar to Lemma 3.1, we have that Lemma 3.3 is valid.

This completes the proof of Lemma 3.3.
\end{proof}

\begin{Lemma}(Noncommutative-commutative Stein interpolation theorem related to  mixed norms)\label{lem3.4}
Let $\mathfrak{H}$ be a separable Hilbert space and $f\in L^{P_{1}}(\R^{d})\cap L^{P_{2}}(\R^{d})$, $g\in L^{Q_{1}}(\R^{d})\cap L^{Q_{2}}(\R^{d})$, $n\geq2$, $P=(p_{1},p_{2},\ldots,p_{d})$, $Q=(q_{1},q_{2},\ldots,q_{d})$,  $T_{z}(z\in\mathbb{C})$ be an operator-function holomorphic in the strip $0\leq Re z\leq a$ and $fT_{z}g\in\mathfrak{R}$, where $\mathfrak{R}$ is the set of all all bounded linear operators. If
\begin{eqnarray}
&&\|fT_{iy}g\|_{\mathfrak{S}^{r_{1}}}\leq M_{0}(y)\|f\|_{L^{P_{1}}}\|g\|_{L^{Q_{1}}},\,(-\infty<y<+\infty),\label{3.064}\\
&&\|fT_{a+iy}g\|_{\mathfrak{S}^{r_{2}}}\leq M_{1}(y)\|f\|_{L^{P_{2}}}\|g\|_{L^{Q_{2}}},\,(-\infty<y<+\infty),\label{3.065}
\end{eqnarray}
and $M_{j}(y)$, $j=0,1$, such that
\begin{eqnarray}
&&\sup\limits_{y\in\SR}e^{-k|y|}\ln M_{j}(y)<\infty\label{3.066}
\end{eqnarray}
for some $0<k<\frac{\pi}{a}$,
and for $h_{1},h_{2}\in\mathfrak{H}$
\begin{eqnarray}
&&\ln|(fT_{z}gh_{1},h_{2})|\leq N_{h_{1},h_{2}}exp(k_{h_{1},h_{2}}|Im z|),
\, (0<Re z<a), \label{3.067}
\end{eqnarray}
where $0\leq k_{h_{1},h_{2}} <\frac{\pi}{a}$, $N_{h_{1},h_{2}}<\infty$.
Then, we have
\begin{eqnarray}
&&\|fT_{\theta}g\|_{\mathfrak{S}^{r}}\leq M_{\theta}\|f\|_{L^{P}}\|g\|_{L^{P}},
\,(0<\theta<a,\,-\infty<y<+\infty),\label{3.068}
\end{eqnarray}
where
\begin{eqnarray*}
&&M_{\theta}=exp\left\{\frac{\sin\pi \theta}{2}\int_{-\infty}^{\infty}
\left[\frac{\ln M_{0}(y)}{\cosh\pi y-\cos\pi \theta }+\frac{\ln M_{1}(y)}{\cosh\pi y-\cos\pi \theta}\right]dy\right\}\\
&&\|f\|_{L^{P}}=\|f\|_{(p_{1},p_{2},\ldots,p_{d})}=\left(\int_{\mathbf{R}}
\ldots\left(\int_{\mathbf{R}}|f(x_{1},x_{2},\ldots,x_{d})|^{p_{1}}dx_{1}\right)
^{\frac{p_{2}}{p_{1}}}\ldots dx_{n}\right)^{\frac{1}{p_{d}}},\\
&&\frac{1}{r}=\frac{1-t_{\theta}}{r_{1}}+\frac{t_{\theta}}{r_{2}},\,
\frac{1}{P}=\frac{1-t_{\theta}}{P_{1}}+\frac{t_{\theta}}{P_{2}}
,\,\frac{1}{Q}=\frac{1-t_{\theta}}{Q_{1}}+\frac{t_{\theta}}{Q_{2}},
\,t_{\theta}=\frac{\theta}{a}.
\end{eqnarray*}

\end{Lemma}

\begin{proof}
Inspired by Theorem 4.1, Page 205 of \cite{SW1971}, Theorem 1, Page 313 of \cite{BP1961}
and Theorem 13.1, Page 137 of \cite{GK1965}, by using Lemma 2.19 and a proof similar
 to Lemma 3.2, we have that Lemma 3.4 is valid.

This completes the proof of Lemma 3.4.
\end{proof}

\section{ Probabilistic estimates of some random series}
\setcounter{equation}{0}

\setcounter{Theorem}{0}

\setcounter{Lemma}{0}

\setcounter{section}{4}

In this section, we present probabilistic estimates of some random series.

\begin{Lemma}\label{lem4.1}(Khinchin type inequalities)
Suppose that $\{g^{(1)}_ {k}(\omega)\}_{k\in \z}$ and $\{g^{(2)}_{j}(\widetilde{\omega})\}_{j\in\mathbb{N}^{+}}$ are independent, zero mean, real valued random variable sequences with probability distributions $\mu^{1}_{k}(k\in\Z)$ and $\mu^{2}_{j}(j\in\mathbb{N}^{+})$ on probability spaces $(\Omega,\mathcal{A}, \mathbb{P})$ and $(\widetilde{\Omega},\widetilde{\mathcal{A}},\widetilde{\mathbb{P}})$, respectively, where $\mu^{1}_{k}(k\in\Z)$ and $\mu^{2}_{j}(j\in\mathbb{N}^{+})$ satisfy \eqref {1.024} and \eqref {1.025}, respectively. Then,
for $r\in[2, \infty)$ and $(a_k)_k,(b_j)_{j}\in\ell^2$, we have
\begin{eqnarray}
&&\left\|\sum_{k\in\z}a_{k}g_{k}^{(1)}(\omega)\right\|_{L_{\omega}^{r}(\Omega)}\leq C_{k}r^{\frac{1}{2}}\|a_{k}\|_{\ell_{k}^{2}},\label{4.01}\\
&&\left\|\sum_{j=1}^{\infty}b_{j}g_{j}^{(2)}(\widetilde{\omega})\right\|_{L_{\widetilde{\omega}}^{r}(\widetilde{\Omega})}\leq C_{j}r^{\frac{1}{2}}\|b_{j}\|_{\ell_{j}^{2}},\label{4.02}
\end{eqnarray}
 $C_k , C_j$ are independent of $r$.

\end{Lemma}

 Lemma 4.1 can be found in  Lemma 3.1 of \cite{BT2008I}.

\begin{Lemma}\label{lem4.2} Let $F(t, \omega, \widetilde{\omega}):[-1,1] \times \Omega \times \widetilde{\Omega}\longrightarrow \R$
 be  measurable.
 For $p \geq 2$ and $\forall \epsilon>0, \exists \delta>0$, when $|t|<\delta$, we have
\begin{eqnarray}
&\|F(t, \omega, \widetilde{\omega})\|_{L_{\omega, \tilde{\omega}}^{p}(\Omega \times \tilde{\Omega})} \leq C p^{\frac{3}{2}} \epsilon^{2}\left(\left\|\gamma_0\right\|_{\mathfrak{S}^2}+1\right).\label{4.03}
\end{eqnarray}
Moreover, we derive
\begin{eqnarray}
\lim _{t \rightarrow 0}(\mathbb{P} \times \widetilde{\mathbb{P}})\left(\left\{(\omega , \widetilde{\omega}) \in\Omega \times \widetilde{\Omega}||F|>C\left(\left\|\gamma_{0}\right\|_{\mathfrak{S}^2}+1\right) e \epsilon^{\frac{1}{2}}\left(\epsilon \ln \frac{1}{\epsilon}\right)^{\frac{3}{2}}\right\}\right)=0.\label{4.04}
\end{eqnarray}
\end{Lemma}

For the proof of Lemma 4.2, we refer the readers to Lemma 6.6 of \cite{YDHXY2024}.

\begin{Lemma}\label{lem4.3}(The estimate related to $\ell_{j}^{2}L_{\omega}^{p}$ on $\R$)
Let $(\lambda_{j})_{j}\in \ell_{j}^{2}$, $(f_{j})_{j=1}^{\infty}$ be an orthonormal
system in $L^{2}\left(\R\right)$ and $f_{j}^{\omega}$ be the randomization of $f_{j}$
defined as in \eqref{1.026}. Then, for $\forall\epsilon>0, \exists\delta>0$, when
 $|t|<\delta<\frac{\epsilon^{2}}{M_{1}\sqrt{1+M_{1}^{2}}}$, we derive
\begin{align}
&\|\lambda_{j}\|U(t)f_{j}^{\omega}-f_{j}^{\omega}\|_{L_{\omega}^{p}}\|_{\ell_{j}^{2}}\leq
Cp^{\frac{1}{2}}\epsilon^{2}\left(\|\gamma_{0}\|_{\mathfrak{S}^{2}}+1\right),\label{4.05}
\end{align}
where $U(t)f=\int_{\SR}e^{ix\xi}e^{it\phi_{1}(\xi)}\mathscr{F}_{x}f(\xi)d\xi$,
$\phi_{1}(\xi)=\sqrt{\xi^{4}+\xi^{2}}$, $\|\gamma_{0}\|_{\mathfrak{S}^{2}}=
\left(\sum\limits_{j=1}^{\infty}|\lambda_{j}|^{2}\right)^{\frac{1}{2}}$
 and $M_{1}$ will appear in \eqref{4.010}.
\end{Lemma}

\begin{proof}
We define $Q_{1}=\sum\limits_{k \in \z}
\left|\int_{\SR} \psi(\xi-k)\left(e^{it\phi_{1}(\xi)}-1\right)e^{i x\xi}\mathscr{F}_x f_n(\xi) d \xi\right|^{2}$.
By using Lemma 4.1, we obtain
\begin{eqnarray}
&\|\lambda_{j}\|U(t)f_{j}^{\omega}-f_{j}^{\omega}\|_{L_{\omega}^{p}}\|_{\ell_{j}^{2}}\leq Cp^{\frac{1}{2}}\left(\sum\limits_{j=1}^{\infty}|\lambda_{j}|^{2}Q_{1}\right)^{\frac{1}{2}}.\label{4.06}
\end{eqnarray}
Since $(\lambda_{j})_{j}\in\ell^{2}$, for $\forall\epsilon>0$, $\exists M\in \mathbb{N}^{+}$, we derive
\begin{eqnarray}
&\left(\sum\limits_{j=M+1}^{\infty}|\lambda_{j}|^{2}\right)^{\frac{1}{2}}<\frac{\epsilon^{2}}{2}.\label{4.07}
\end{eqnarray}
Since $|e^{it\phi_{1}(\xi)}-1|\leq2$, $\|f_{j}\|_{L^{2}}=1$, by using \eqref{4.07} and the Cauchy-Schwartz inequality, we obtain
\begin{align}
\left(\sum_{j=M+1}^{\infty}|\lambda_{j}|^{2}Q_{1}\right)^{\frac{1}{2}}
&\leq C\left(\sum_{j=M+1}^{\infty}|\lambda_{j}|^{2}\sum_{k\in\z}\int_{\SR}|\psi(\xi-k)
\mathscr{F}_{x}f_{j}(\xi)|^{2}d\xi\right)^{\frac{1}{2}}\nonumber\\
&\leq C\left(\sum_{j=M+1}^{\infty}|\lambda_{j}|^{2}\|f_{j}\|_{L^{2}}^{2}\right)^{\frac{1}{2}}\leq C\left(\sum_{j=M+1}^{\infty}|\lambda_{j}|^{2}\right)^{\frac{1}{2}}\leq C\epsilon^{2}.\label{4.08}
\end{align}
We claim that for $1\leq j\leq M$, $\forall\epsilon>0$,
\begin{eqnarray}
&Q_{1}^{\frac{1}{2}}\leq C(\epsilon^{2}+|t|M_{1}\sqrt{1+M_{1}^{2}})\label{4.09}
\end{eqnarray}
is valid, where $M_1$ depends on $M$. The idea of Lemma 6.3 of \cite{YDHXY2024} will be  used
 to establish Lemma 4.3.From line 14 of 1908 in \cite{YZDY2022}, we obtain
$$
\sum_{k\in\z}\|\psi(\xi-k)\mathscr{F}_{x}f_{j}\|_{L^{2}}^{2}\leq \|f_{j}\|_{L^{2}}^{2}\leq 3\sum_{k\in\z}\|\psi(\xi-k)\mathscr{F}_{x}f_{j}\|_{L^{2}}^{2},
$$
consequently, for $\forall\epsilon>0$,  $\exists M_{1}\geq 100$, we have
\begin{eqnarray}
&\left(\sum\limits_{|k|\geq M_{1}}
\|\psi(\xi-k)\mathscr{F}_{x}f_{j}\|_{L^{2}}^{2}\right)^{\frac{1}{2}}\leq \frac{\epsilon^{2}}{2}.\label{4.010}
\end{eqnarray}
Since $\supp \psi\subset[-1,1]$ and $|e^{it\phi_{1}(\xi)}-1|\leq2$, by using \eqref{4.010}
 and the Cauchy-Schwartz inequality, we obtain
\begin{eqnarray}
&&\left(\sum\limits_{|k| \geq M_1}\left|\int_{\SR} \psi(\xi-k)\left(e^{it\phi_{1}(\xi)}-1\right)
 e^{i x\xi} \mathscr{F}_x f_j(\xi) d \xi\right|^2\right)^{\frac{1}{2}} \nonumber\\
&&\leq 2\left(\sum\limits_{|k| \geq M_1} \int_{\mathbf{R}}\left|\psi(\xi-k)
\mathscr{F}_x f_j(\xi)\right|^2 d \xi\right)^{\frac{1}{2}}\nonumber\\
&&=2\left(\sum\limits_{|k| \geq M_1}\left\|\psi(\xi-k)
\mathscr{F}_x f_j\right\|_{L^2}^2\right)^{\frac{1}{2}}\leq \epsilon^{2}.\label{4.011}
\end{eqnarray}
Since $\supp \psi\subset[-1,1]$ and $|e^{it\phi_{1}(\xi)}-1|\leq|t||\xi|\sqrt{1+\xi^{2}}$,
 $\|f_{j}\|_{L^{2}}=1$, by using a proof similar to \eqref{4.011},
we get
\begin{eqnarray}
&&\left(\sum_{|k| \leq M_1}\left|\int_{\SR} \psi(\xi-k)
\left(e^{it\phi_{1}(\xi)}-1\right) e^{i x\xi}\mathscr{F}_x f_n(\xi)
d \xi\right|^2\right)^{\frac{1}{2}} \nonumber\\
&&\leq\left(\sum_{|k| \leq M_1}|t|^2(|k|^4+|k|^{2})\left|\int_{\mathbf{R}}|
\psi(\xi-k) \mathscr{F}_x f_j(\xi)|d \xi\right|^2\right)^{\frac{1}{2}}\nonumber\\
&&\leq|t|M_{1}\sqrt{1+M_{1}^{2}}\left(\sum_{|k| \leq M_1}\left\|\psi(\xi-k)
\mathscr{F}_x f_j\right\|_{L^2}^2\right)^{\frac{1}{2}}\nonumber\\
&&\leq C|t|M_{1}\sqrt{1+M_{1}^{2}}\left(\sum_{k \in \z}\left\|\psi(\xi-k)
\mathscr{F}_x f_j\right\|_{L^2}^2\right)^{\frac{1}{2}}\nonumber \\
&&\leq C|t|M_{1}\sqrt{1+M_{1}^{2}}\left\|f_j\right\|_{L^2} \leq C|t|M_{1}
\sqrt{1+M_{1}^{2}}.\label{4.012}
\end{eqnarray}
From \eqref{4.011} and \eqref{4.012}, we have that \eqref{4.09} is valid.
By using \eqref{4.06}, \eqref{4.07} and \eqref{4.09} as well as
$(a+b)^{\frac{1}{2}}\leq a^{\frac{1}{2}}+b^{\frac{1}{2}}$,
$|e^{it\phi_{1}(\xi)}-1|\leq2$ and $\|f_{j}\|_{L^{2}}=1$,
when $|t|<\delta<\frac{\epsilon^{2}}{M_{1}\sqrt{1+M_{1}^{2}}}$,
we can get
\begin{eqnarray}
&&\left\|\lambda_{j}\|U(t)f_{j}^{\omega}-f_{j}^{\omega}\right\|_{L_{\omega}^{p}}\|_{\ell_{j}^{2}}
\leq C p^{\frac{1}{2}}\left(\sum_{j=1}^{+\infty}\left|\lambda_j\right|^2 Q_{1}\right)^{\frac{1}{2}}\nonumber\\
&&\leq C p^{\frac{1}{2}}\left(\sum_{j=1}^M\left|\lambda_j\right|^2Q_{1}\right)^{\frac{1}{2}}
+C p^{\frac{1}{2}}\left(\sum_{j=M+1}^{+\infty}\left|\lambda_j\right|^2 Q_{1}\right)^{\frac{1}{2}}\nonumber\\
&&\leq C p^{\frac{1}{2}}\left(|t|M_{1}\sqrt{1+M_{1}^{2}}+\epsilon^{2}\right)\left(\sum_{j=1}^M |\lambda_j|^2\right)^{\frac{1}{2}}+C p^{\frac{1}{2}}\left(\sum_{j=M+1}^{+\infty}\left|\lambda_j\right|^2 \sum_{k \in \z}\left|\mathscr{F}_x f_j(k)\right|^2\right)^{\frac{1}{2}}\nonumber\\
&&\leq C p^{\frac{1}{2}}\left(|t| M_{1}\sqrt{1+M_{1}^{2}}+\epsilon^{2}\right)\left(\sum_{j=1}^M |\lambda_j|^2\right)^{\frac{1}{2}}+C p^{\frac{1}{2}}\left(\sum_{j=M+1}^{+\infty}\left|\lambda_j\right|^2\right)^{\frac{1}{2}}\nonumber\\
&&\leq C p^{\frac{1}{2}}\left(|t|M_{1}\sqrt{1+M_{1}^{2}}+\epsilon^{2}\right)\left(\sum_{j=1}^M |\lambda_j|^2\right)^{\frac{1}{2}}+C p^{\frac{1}{2}}\epsilon^{2}\nonumber\\
&&\leq C p^{\frac{1}{2}}\left(\epsilon^{2}\left(\sum_{j=1}^{\infty} |\lambda_j|^2\right)^{\frac{1}{2}}+\epsilon^{2}\right)
\leq C p^{\frac{1}{2}} \epsilon^{2}\left(\left\|\gamma_0\right\|_{\mathfrak{S}^2}+1\right),\label{4.013}
\end{eqnarray}

The proof of Lemma 4.3 is completed.

\end{proof}

\begin{Lemma}\label{lem4.4}(The estimate related to $\ell_{j}^{2}L_{\omega}^{p}$
on $\mathbf{T}$)
Let $\|f_{j}\|_{L^{2}(\mathbf{T})}=\left(\sum\limits_{k\in\z}|\mathscr{F}_{x}f_{j}(k)|^{2}\right)^{\frac{1}{2}}
=1$ $(j\in \mathbb{N}^{+})$,  $(\lambda_{j})_{j}\in \ell_{j}^{2}$, and $f_{j}^{\omega}$
 be the randomization of $f_{j}$
 defined as in \eqref{1.027}. Then, for $\forall\epsilon>0, \exists\delta>0$, when
 $|t|<\delta<\frac{\epsilon^{2}}{M_{1}\sqrt{1+M_{1}^{2}}}$, we have
\begin{align}
&\|\lambda_{j}\|U(t)f_{j}^{\omega}-f_{j}^{\omega}\|_{L_{\omega}^{p}}\|_{\ell_{j}^{2}}\leq Cp^{\frac{1}{2}}\epsilon^{2}\left(\|\gamma_{0}\|_{\mathfrak{S}^{2}}+1\right),\label{4.014}
\end{align}
where $U(t)f=\sum\limits_{k\in\z}e^{ixk}e^{it\phi_{2}(k)}\mathscr{F}_{x}f(k)$,
$\phi_{2}(k)=\sqrt{k^{4}+k^{2}}$,  $M_{1}$ appears in \eqref{4.017} and
 $\|\gamma_{0}\|_{\mathfrak{S}^{2}}=\left(\sum\limits_{j=1}^{\infty}|
\lambda_{j}|^{2}\right)^{\frac{1}{2}}$.
\end{Lemma}

\begin{proof}
We define $Q_{2}=\sum\limits_{k\in\z}\left|(e^{it\phi_{2}(k)}-1)
\mathscr{F}_{x}f_{j}(k)\right|^{2}$.  By using Lemma 4.1, we obtain
\begin{eqnarray}
&\|\lambda_{j}\|U(t)f_{j}^{\omega}-f_{j}^{\omega}\|_{L_{\omega}^{p}}\|_{\ell_{j}^{2}}\leq Cp^{\frac{1}{2}}\left(\sum\limits_{n=1}^{\infty}|\lambda_{j}|^{2}Q_{2}\right)^{\frac{1}{2}}.\label{4.015}
\end{eqnarray}
Since $(\lambda_{j})_{j}\in\ell^{2}$, for $\forall\epsilon>0$,
   $\exists M\in \mathbb{N}^{+}$, we have
\begin{eqnarray}
&\left(\sum\limits_{j=M+1}^{\infty}|\lambda_{j}|^{2}\right)^{\frac{1}{2}}
<\frac{\epsilon^{2}}{2}.\label{4.016}
\end{eqnarray}
Since $|e^{it\phi_{2}(k)}-1|\leq2$, $\|f_{j}\|_{L^{2}(\mathbf{T})}=1$,
 then for $\forall\epsilon>0$,
 $\exists M_{1}\geq 100$, we have
\begin{eqnarray}
&\left(\sum\limits_{|k|\geq M_{1}}\left|(e^{it\phi_{2}(k)}-1)
\mathscr{F}_{x}f_{j}(k)\right|^{2}\right)^{\frac{1}{2}}\leq C
\left(\sum\limits_{|k|\geq M_{1}}\left|
\mathscr{F}_{x}f_{j}(k)\right|^{2}\right)^{\frac{1}{2}}\leq C
\epsilon^{2},\label{4.017}
\end{eqnarray}
We claim that for $1\leq j\leq M$, $\forall\epsilon>0$, we have that
\begin{eqnarray}
&Q_{2}^{\frac{1}{2}}\leq C\left(\epsilon^{2}+|t|M_{1}\sqrt{1+M_{1}^{2}}\right).\label{4.018}
\end{eqnarray}
Here, $M_1$ depends upon $M$. The idea of Lemma 7.1 of \cite{YDHXY2024}
 will be used to prove Lemma 4.4.
It follows from $|e^{it\phi_{2}(k)}-1|\leq|t|\sqrt{k^{2}+k^{4}}$ and $\|f_{j}\|_{L^{2}}=1$ that
\begin{align}
\left(\sum_{|k| \leq M_1}\left|\left(e^{it\phi_{2}(k)}-1\right)\mathscr{F}_{x} f_j(k)\right|^2\right)^{\frac{1}{2}}&\leq\left(\sum_{|k| \leq M_1}|t|^2(|k|^4+|k|^{2})|\mathscr{F}_{x} f_{j}(k)|^2\right)^{\frac{1}{2}}\nonumber\\
&\leq C|t|M_{1}\sqrt{1+M_{1}^{2}}\left\|f_j\right\|_{L^2}\nonumber\\
& \leq C|t|M_{1}\sqrt{1+M_{1}^{2}}.\label{4.019}
\end{align}
From \eqref{4.017} and \eqref{4.019}, we know that \eqref{4.018} is valid. From \eqref{4.015}, \eqref{4.016} and \eqref{4.018}, since
$(a+b)^{\frac{1}{2}}\leq a^{\frac{1}{2}}+b^{\frac{1}{2}}$, $|e^{it\phi_{2}(k)}-1|\leq2$ and $\|f_{j}\|_{L^{2}}^{2}=\sum\limits_{k \in \z}\left|\mathscr{F}_x f_j(k)\right|^2=1$, when $|t|<\delta<\frac{\epsilon^{2}}{M_{1}\sqrt{1+M_{1}^{2}}}$, we can get
\begin{eqnarray}
&&\|\lambda_{j}\|U(t)f_{j}^{\omega}-f_{j}^{\omega}\|_{L_{\omega}^{p}}\|_{\ell_{j}^{2}}
\leq C p^{\frac{1}{2}}\left(\sum_{j=1}^{+\infty}\left|\lambda_j\right|^2 Q_{2}\right)^{\frac{1}{2}}\nonumber\\
&&\leq C p^{\frac{1}{2}}\left(\sum_{j=1}^M\left|\lambda_j\right|^2 Q_{2}\right)^{\frac{1}{2}}
+C p^{\frac{1}{2}}\left(\sum_{j=M+1}^{+\infty}\left|\lambda_j\right|^2 Q_{2}\right)^{\frac{1}{2}}\nonumber\\
&&\leq C p^{\frac{1}{2}}\left(|t|M_{1}\sqrt{1+M_{1}^{2}}+\epsilon^{2}\right)\left(\sum_{j=1}^M |\lambda_j|^2\right)^{\frac{1}{2}}+C p^{\frac{1}{2}}\left(\sum_{j=M+1}^{+\infty}\left|\lambda_j\right|^2 \sum_{k \in \z}\left|\mathscr{F}_x f_n(k)\right|^2\right)^{\frac{1}{2}}\nonumber\\
&&\leq C p^{\frac{1}{2}}\left(|t|M_{1}\sqrt{1+M_{1}^{2}}+\epsilon^{2}\right)\left(\sum_{j=1}^M |\lambda_j|^2\right)^{\frac{1}{2}}+C p^{\frac{1}{2}}\left(\sum_{j=M+1}^{+\infty}\left|\lambda_j\right|^2\right)^{\frac{1}{2}}\nonumber\\
&&\leq C p^{\frac{1}{2}}\left(|t|M_{1}\sqrt{1+M_{1}^{2}}+\epsilon^{2}\right)\left(\sum_{j=1}^M |\lambda_{j}|^{2}\right)^{\frac{1}{2}}+C p^{\frac{1}{2}}\epsilon^{2}\nonumber\\
&&\leq C p^{\frac{1}{2}}\left(\epsilon^{2}\left(\sum_{j=1}^{\infty} |\lambda_j|^2\right)^{\frac{1}{2}}+\epsilon^{2}\right)
\leq C p^{\frac{1}{2}} \epsilon^{2}(\left\|\gamma_0\right\|_{\mathfrak{S}^2}+1).\label{4.020}
\end{eqnarray}

The proof of Lemma 4.4 is completed.

\end{proof}

\begin{Lemma}\label{lem4.5}(The estimate related to $\ell_{j}^{2}L_{\omega_{1}}^{p}$
 on $\Theta=\{x\in\R^{3}:|x|<1\}$)
Let $f_{j}=\sum\limits_{m=1}^{\infty}c_{j,m}e_{m}$, $f_{j}^{\omega_{1}}$ be the
 randomization of $f_{j}$  defined as in \eqref{1.028}. Then, for $\forall\epsilon>0,
  \exists\delta>0$, when $|t|<\delta<\frac{\epsilon^{2}}{M_{1}\sqrt{1+M_{1}^{2}}}$, we have
\begin{align}
&\|\lambda_{j}\|U(t)f_{j}^{\omega_{1}}-f_{j}^{\omega_{1}}\|_{L_{\omega_{1}}^{p}}\|_{\ell_{j}^{2}}\leq Cp^{\frac{1}{2}}\epsilon^{2}\left(\|\gamma_{0}\|_{\mathfrak{S}^{2}}+1\right),
\label{4.021}
\end{align}
where $U(t)f=\sum\limits_{m=1}^{\infty}e^{it\phi_{3}(m)}c_{j,m}e_{m}$,
$\phi_{3}(m)=\sqrt{(m\pi)^{2}+(m\pi)^{4}}$, $e_{m}=\frac{\sin m\pi|x|}{(2\pi)^{\frac{1}{2}}|x|}$,
 $c_{j,m}=\int_{\Theta}f_{j}e_{m}dx$, $\|\gamma_{0}\|_{\mathfrak{S}^{2}}=\left(\sum\limits_{j=1}^{\infty}|
\lambda_{j}|^{2}\right)^{\frac{1}{2}}$ and $M_{1}$ appears in \eqref{4.024}.
\end{Lemma}

\begin{proof}
We define $Q_{3}=\sum\limits_{m\in\z}\left|(e^{it\phi_{3}(m)}-1)
\frac{c_{j,m}e_{m}}{m\pi}\right|^{2}$. By using Lemma 4.1, we obtain
\begin{eqnarray}
&\|\lambda_{j}\|U(t)f_{j}^{\omega_{1}}-f_{j}^{\omega_{1}}\|_{L_{\omega_{1}}^{p}}\|_{\ell_{j}^{2}}
\leq Cp^{\frac{1}{2}}\left(\sum\limits_{j=1}^{\infty}
|\lambda_{j}|^{2}Q_{3}\right)^{\frac{1}{2}}.\label{4.022}
\end{eqnarray}
Since $(\lambda_{j})_{j}\in\ell^{2}$, for $\forall\epsilon>0$,
 $\exists M\in \mathbb{N}^{+}$, we have
\begin{eqnarray}
&\left(\sum\limits_{j=M+1}^{\infty}|\lambda_{j}|^{2}\right)^{\frac{1}{2}}
<\frac{\epsilon^{2}}{2}.\label{4.023}
\end{eqnarray}
Since $|e^{it\phi_{3}(m)}-1|\leq2$, $|\frac{e_{m}}{m\pi}|\leq C$, we obtain
\begin{eqnarray}
&\left(\sum\limits_{|m|\geq M_{1}}\left|(e^{it\phi_{3}(m)}-1)
\frac{c_{j,m}e_{m}}{m\pi}\right|^{2}\right)^{\frac{1}{2}}\leq C\left(\sum\limits_{|m|\geq M_{1}}\left|
c_{j,m}\right|^{2}\right)^{\frac{1}{2}}\leq C\epsilon^{2},\label{4.024}
\end{eqnarray}
where we use the fact that $1\leq j\leq M$. For $\forall\epsilon>0$,   $\exists M_{1}\geq 100$, we have
\begin{eqnarray*}
&\left(\sum\limits_{|m|\geq M_{1}}\left|
c_{j,m}\right|^{2}\right)^{\frac{1}{2}}\leq C\epsilon^{2}.
\end{eqnarray*}
We claim that for $1\leq j\leq M$, $\forall\epsilon>0$, we have that
\begin{eqnarray}
&&Q_{3}^{\frac{1}{2}}\leq C(\epsilon^{2}+|t|M_{1}\sqrt{1+M_{1}^{2}}).\label{4.025}
\end{eqnarray}
Here, $M_1$ depends upon $M$. The idea of Lemma 7.1 of \cite{YDHXY2024}
 will be used to prove Lemma 4.5. It follows from $|e^{it\phi_{3}(m)}-1|\leq
 C|t|\sqrt{m^{2}+m^{4}}$ and $|\frac{e_{m}}{m\pi}|\leq C$ that
\begin{align}
\left(\sum_{|m| \leq M_1}\left|\left(e^{it\phi_{3}(m)}-1\right)\frac{c_{j,m}
e_{m}}{m\pi}\right|^2\right)^{\frac{1}{2}}&\leq\left(\sum_{|m| \leq M_1}
|t|^2(|m|^4+|m|^{2})\left|c_{j,m}\right|^2\right)^{\frac{1}{2}}\nonumber\\
& \leq C|t|M_{1}\sqrt{1+M_{1}^{2}}.\label{4.026}
\end{align}
From \eqref{4.024} and \eqref{4.026}, we derive \eqref{4.025}. From
\eqref{4.022}, \eqref{4.023} and \eqref{4.025}, since
$(a+b)^{\frac{1}{2}}\leq a^{\frac{1}{2}}+b^{\frac{1}{2}}$,
$|e^{it\phi_{3}(m)}-1|\leq2$ and $|\frac{e_{m}}{m\pi}|\leq C$,
when $|t|<\delta<\frac{\epsilon^{2}}{M_{1}\sqrt{1+M_{1}^{2}}}$, we can get
\begin{eqnarray}
&&\|\lambda_{j}\|U(t)f_{j}^{\omega_{1}}-f_{j}^{\omega_{1}}\|_{L_{\omega_{1}}^{p}}\|_{\ell_{j}^{2}}
\leq C p^{\frac{1}{2}}\left(\sum_{j=1}^{+\infty}\left|\lambda_{j}\right|^2 Q_{3}\right)^{\frac{1}{2}}\nonumber\\
&&\leq C p^{\frac{1}{2}}\left(\sum_{j=1}^M\left|\lambda_j\right|^2 Q_{3}\right)^{\frac{1}{2}}
+C p^{\frac{1}{2}}\left(\sum_{j=M+1}^{+\infty}\left|\lambda_j\right|^2 Q_{3}\right)^{\frac{1}{2}}\nonumber\\
&&\leq C p^{\frac{1}{2}}\left(|t|M_{1}\sqrt{1+M_{1}^{2}}+\epsilon^{2}\right)\left(\sum_{j=1}^M |\lambda_j|^2\right)^{\frac{1}{2}}+C p^{\frac{1}{2}}\left(\sum_{j=M+1}^{+\infty}\left|\lambda_j\right|^2 \sum_{m \in \z}\left|c_{j,m}\right|^2\right)^{\frac{1}{2}}\nonumber\\
&&\leq C p^{\frac{1}{2}}\left(|t| M_{1}\sqrt{1+M_{1}^{2}}+\epsilon^{2}\right)\left(\sum_{j=1}^M |\lambda_j|^2\right)^{\frac{1}{2}}+C p^{\frac{1}{2}}\left(\sum_{j=M+1}^{+\infty}|\lambda_{j}|^2\right)^{\frac{1}{2}}\nonumber\\
&&\leq C p^{\frac{1}{2}}\left(\left(|t|M_{1}\sqrt{1+M_{1}^{2}}+\epsilon^{2}\right)\left(\sum_{j=1}^{\infty} |\lambda_j|^2\right)^{\frac{1}{2}}+\epsilon^{2}\right)\nonumber\\
&&\leq C p^{\frac{1}{2}} \epsilon^{2}\left(\left(\sum_{j=1}^{\infty} |\lambda_j|^2\right)^{\frac{1}{2}}+1\right)
\leq C p^{\frac{1}{2}} \epsilon^{2}(\left\|\gamma_0\right\|_{\mathfrak{S}^2}+1).\label{4.027}
\end{eqnarray}

The proof of Lemma 4.5 is finished.

\end{proof}

\begin{Lemma}\label{lem4.6}
Let $(f_{j})_{j=1}^{\infty}$ is an orthonormal system in $L^{2}\left(\R\right)$,
 $f_{j}^{\omega}$ be the randomization of $f_{j}$  defined as in \eqref{1.026}. Then, we have
\begin{eqnarray}
&&\left\|f_{j}^{\omega}\right\|_{L_{\omega}^{p}} \leq C p^{\frac{1}{2}}\left\|f_{j}\right\|_{L^{2}}=C p^{\frac{1}{2}}\left\|U(t) f_{j}^{\omega}\right\|_{L_{\omega}^{p}} \leq C p^{\frac{1}{2}}\left\|f_{j}\right\|_{L^{2}}=C p^{\frac{1}{2}},\label{4.028}
\end{eqnarray}
where $U(t)f=\int_{\SR}e^{ix\xi}e^{it\phi_{1}(\xi)}\mathscr{F}_{x}f(\xi)d\xi$, $\phi_{1}(\xi)=\sqrt{\xi^{4}+\xi^{2}}$ and $C$ is independent of $p.$
\end{Lemma}

By using a proof  similar to Lemma 2.3 of \cite{BOP2015}, we obtain  Lemma 4.6.

\begin{Lemma}\label{lem4.7}
Let $\|f_{j}\|_{L^{2}(\mathbf{T})}=1$, $j\in \mathbb{N}^{+}$ and $f_{n}^{\omega}$ be the randomization
of $f_{j}$  defined as in \eqref{1.027}. Then, we have
\begin{eqnarray}
&&\left\|f_{j}^{\omega}\right\|_{L_{\omega}^{p}} \leq C p^{\frac{1}{2}}\left\|f_{j}\right\|_{L^{2}}=C p^{\frac{1}{2}}\left\|U(t) f_{j}^{\omega}\right\|_{L_{\omega}^{p}} \leq C p^{\frac{1}{2}}\left\|f_{j}\right\|_{L^{2}}=C p^{\frac{1}{2}},\label{4.029}
\end{eqnarray}
where $U(t)f=\sum\limits_{k\in\z}e^{ixk}e^{it\phi_{2}(k)}\mathscr{F}_{x}f(k)$, $\phi_{2}(k)=\sqrt{k^{4}+k^{2}}$ and $C>0$ is independent of $p$.
\end{Lemma}

By using  a proof  similar   to   Lemma 7.2 of \cite{YDHXY2024}, we obtain  Lemma 4.7.

\begin{Lemma}\label{lem4.8}
Let $f_{j}=\sum\limits_{m=1}^{\infty}c_{j,m}e_{m}$, $f_{j}^{\omega_{1}}$
be  the randomization of $f_{j}$  defined as in \eqref{1.028}. Then, we have
\begin{eqnarray}
&&\left\|f_{j}^{\omega_{1}}\right\|_{L_{\omega_{1}}^{p}} \leq C p^{\frac{1}{2}}\left\|f_{j}\right\|_{L^{2}}=C p^{\frac{1}{2}}\left\|U(t) f_{j}^{\omega_{1}}\right\|_{L_{\omega_{1}}^{p}} \leq C p^{\frac{1}{2}}\left\|f_{j}\right\|_{L^{2}}=C p^{\frac{1}{2}},\label{4.030}
\end{eqnarray}
where $U(t)f=\sum\limits_{m=1}^{\infty}e^{it\phi_{3}(m)}c_{j,m}e_{m}$, $\phi_{3}(m)=\sqrt{(m\pi)^{2}+(m\pi)^{4}}$, $e_{m}=\frac{\sin m\pi|x|}{(2\pi)^{\frac{1}{2}}|x|}$, $c_{j,m}=\int_{\Theta}f_{j}e_{m}dx$ and  $C$ is independent of $p.$
\end{Lemma}

By using a  proof similar to Lemma 8.2 of \cite{YDHXY2024}, we obtain    Lemma 4.8.

\bigskip
\bigskip

\section{ Proof of Theorem 1.1: Maximal-in-time estimates for orthonormal functions on $\mathbf{R}$}
\setcounter{equation}{0}

\setcounter{Theorem}{0}

\setcounter{Lemma}{0}

\setcounter{section}{5}

In this section, we present the maximal-in-time estimates for orthonormal functions on $\mathbf{R}$.

Inspired by \cite{BLN2020,BHLNS2019}, to prove Theorem 1.1, by using the homogeneous  dyadic decomposition
\begin{eqnarray*}
&\sum\limits_{k=-\infty}^{+\infty}\psi(2^{k}x)=1,
\end{eqnarray*}
 where $\psi \in C_{c}^{\infty}(\{x:\frac{1}{2}<|x|<2\})$ is a suitable nonnegative even function.
We define $$U(t)f=e^{it\sqrt{\partial_{x}^{4}-\partial_{x}^{2}}}D_{x}^{-\frac{1}{4}}f.$$
Then, we have
$$U^{\ast}(t)g=\int_{I}e^{-it\sqrt{\partial_{x}^{4}-\partial_{x}^{2}}}D_{x}^{-\frac{1}{4}}gdt. $$
We define
$$K(t,x)=\frac{1}{2\pi}\int_{\SR}e^{ix\xi}e^{it\sqrt{\xi^{4}+\xi^{2}}}\frac{d\xi}{|\xi|^{\frac{1}{2}}}.$$
Then, we have
$$UU^{\ast}F(t,x)=\int_{\SR\times I}K(t-t^{\prime},x-x^{\prime})F(t^{\prime},x^{\prime})dt^{\prime}dx^{\prime}$$
and we decompose the operator $UU^{\ast}$ as follows
\begin{align*}
UU^{\ast}F(t,x)=\int_{\SR\times I}K(t-t^{\prime},x-x^{\prime})F(t^{\prime},x^{\prime})dt^{\prime}dx^{\prime}=\sum_{k\in\z}T_{k}F(t,x),
\end{align*}
where
\begin{align*}
T_{k}F(t,x)&=\int_{\SR\times I}K_{k}(t-t^{\prime},x-x^{\prime})F(t^{\prime},x^{\prime})dt^{\prime}dx^{\prime},
\end{align*}
and the integral kernel of $T_{k}$ is given by
$$K_{k}(t-t^{\prime},x-x^{\prime})=\psi(2^{k}(x-x^{\prime}))K(t-t^{\prime},x-x^{\prime}).$$
By using a direct calculation, we have
\begin{align*}
\langle T_{k}f, g\rangle_{L_{tx}^{2}}&=\int_{\SR\times I}\left(\int_{\SR\times I}K_{k}(t-t^{\prime},x-x^{\prime})f(t^{\prime},x^{\prime})dt^{\prime}dx^{\prime}\right)
\overline{g(x,t)}dxdt\nonumber\\
&=\int_{\SR\times I}f(t^{\prime},x^{\prime})\left(\overline{\int_{\SR\times I}\overline{K_{k}(t-t^{\prime},x-x^{\prime})}g(x,t)dxdt}\right)dt^{\prime}dx^{\prime}\nonumber\\
&=\left\langle f,\int_{\SR\times I}\overline{K_{k}(t-t^{\prime},x-x^{\prime})}g(x,t)dxdt\right\rangle_{L_{x}^{2}}
=\langle f, T_{k}^{\ast}g\rangle_{L_{x}^{2}}.
\end{align*}
From the above equality, we have
\begin{eqnarray*}
&&T_{k}^{\ast}g=\int_{\SR \times I}\overline{K_{k}(t-t^{\prime},x-x^{\prime})}g(x,t)dxdt.
\end{eqnarray*}
Before proving Theorem 1.1, we give the following estimates that  play a key role in proving Theorem 1.1.

\begin{Lemma}\label{lem5.1}
Let $k\in \mathbf{Z},k\neq0$. Then,  we have
\begin{align}
\|W_{1}T_{k}W_{2}\|_{\mathfrak{S}^{2}}&\leq C\|W_{1}\|_{L_{x}^{4}L_{t}^{2}(\mathbf{R}\times I)}
\|W_{2}\|_{L_{x}^{4}L_{t}^{2}(\mathbf{R}\times I)},\label{5.01}\\
\|W_{1}T_{k}W_{2}\|_{\mathfrak{S}^{4}}&\leq C2^{(\frac{1}{p_{1}}+\frac{1}{p_{2}}-\frac{1}{2})k}
\|W_{1}\|_{L_{x}^{p_{1}}L_{t}^{2}(\mathbf{R}\times I)}
\|W_{2}\|_{L_{x}^{p_{2}}L_{t}^{2}(\mathbf{R}\times I)},\label{5.02}
\end{align}
is valid, where $C$ independent of $k$, $p_1$ and $p_2$ satisfy
$(\frac{1}{p_{1}},\frac{1}{p_{2}})\in[0,\frac{1}{2}]^{2}$ and
$$\frac{1}{p_{1}}+\frac{1}{p_{2}}\geq\frac{1}{4}, \frac{1}{p_{1}}-\frac{1}{2p_{2}}\leq\frac{1}{4}, \frac{1}{p_{2}}-\frac{1}{2p_{1}}\leq\frac{1}{4}.$$

\end{Lemma}
\begin{proof}
Firstly, we prove \eqref{5.01}.  By using Lemma 2.3, we have
\begin{align}
&&|K_{k}(t-t^{\prime},x-x^{\prime})|\leq C\psi(2^{k}(x-x^{\prime}))|x-x^{\prime}|^{-\frac{1}{2}}.\label{5.03}
\end{align}
By using \eqref{5.03} and H\"{o}lder inequality and Young's convolution inequality, we obtain
\begin{align}
\|W_{1}T_{k}W_{2}\|_{\mathfrak{S}^{2}}^{2}&=\int_{\SR\times I}\int_{\SR\times I}|W_{1}(t,x)|^{2}|K_{k}(t-t^{\prime},x-x^{\prime})|^{2}|W_{2}(t^{\prime},x^{\prime})|^{2}dtdt^{\prime}dx
dx^{\prime}\nonumber\\
&\leq C\int_{\SR^{2}}\frac{\|W_{1}(x,\cdot)\|_{L_{t}^{2}}^{2}
\|W_{2}(x^{\prime},\cdot)\|_{L_{t}^{2}}^{2}}{|x-x^{\prime}|}dxdx^{\prime}\nonumber\\
&\leq C\|W_{1}\|_{L_{x}^{4}L_{t}^{2}}^{2}\left\|\int_{\SR}\frac{\psi^{2}(2^{k}(x-x^{\prime}))
\|W_{2}(x^{\prime},\cdot)\|_{L_{t}^{2}}^{2}}
{|x-x^{\prime}|}dx^{\prime}\right\|_{L_{x}^{2}}\nonumber\\
&\leq  C\|W_{1}\|_{L_{x}^{4}L_{t}^{2}}^{2}\|W_{2}\|_{L_{x}^{4}L_{t}^{2}}^{2}
\int_{\SR}\frac{\psi^{2}(2^{k}x)}{|x|}dx\nonumber\\
&=C\|W_{1}\|_{L_{x}^{4}L_{t}^{2}}^{2}\|W_{2}\|_{L_{x}^{4}L_{t}^{2}}^{2}2^{k}2^{-k}
\int_{\SR}\frac{\psi^{2}(y)}{|y|}dy\nonumber\\
&\leq C\|W_{1}\|_{L_{x}^{4}L_{t}^{2}}^{2}\|W_{2}\|_{L_{x}^{4}L_{t}^{2}}^{2}.\label{5.04}
\end{align}
Then, we prove \eqref{5.02}, since $\|W_{1}T_{k}W_{2}\|_{\mathfrak{S}^{4}}^{4}=
\|\overline{W_{2}}T_{k}^{\ast}|W_{1}|^{2}T_{k}W_{2}\|_{\mathfrak{S}^{2}}^{2}$ and
$$\overline{W_{2}}T_{k}^{\ast}|W_{1}|^{2}T_{k}W_{2}[F](t,x)=\int_{\SR\times I}
\mathcal{M}_{k}(t,t^{\prime\prime},x,x^{\prime\prime})
F(t^{\prime\prime},x^{\prime\prime})dt^{\prime\prime}dx^{\prime\prime},$$
where
\begin{eqnarray*}
&&\mathcal{M}_{k}(t,t^{\prime\prime},x,x^{\prime\prime})\nonumber\\&&=
\overline{W_{2}}(t,x)\left(\int_{\SR\times I}\overline{K_{k}(t-t^{\prime},x-x^{\prime})}
|W_{1}(t^{\prime},x^{\prime})|^{2}
K_{k}(t^{\prime}-t^{\prime\prime},x^{\prime}-x^{\prime\prime})
dt^{\prime}dx^{\prime}\right)W_{2}(t^{\prime\prime},x^{\prime\prime}).
\end{eqnarray*}
Let $x= x_{1}$ and $x^{\prime\prime}= x_{4}$, from Lemma 2.3,   we get
\begin{eqnarray*}
&&\|W_{1}T_{k}W_{2}\|_{\mathfrak{S}^{4}}^{4}=\int_{\SR\times I}\int_{\SR\times I}|\mathcal{M}_{k}(t,t^{\prime\prime},x_{1},x_{4})|^{2}dtdt^{\prime\prime}dx_{1}
dx_{4}\nonumber\\
&&=\int_{\SR\times I}\int_{\SR\times I}\left|\overline{W_{2}}(t,x_{1})\left(\int_{\SR\times I}\overline{K_{k}(t-t^{\prime},x_{1}-x^{\prime})}|W_{1}(t^{\prime},x^{\prime})
|^{2}K_{k}(t^{\prime}-t^{\prime\prime},x^{\prime}-x_{4})
dt^{\prime}dx^{\prime}\right)\right|^{2}\nonumber\\
&&\left|W_{2}(t^{\prime\prime},x_{4})\right|^{2}dtdt^{\prime\prime}dx_{1}
dx_{4}\nonumber\\
&&\leq C\int_{\SR}\int_{\SR}\|W_{2}(t,x_{1})\|_{L_{t}^{2}}^{2}
\left(\int_{\SR}2^{k}\psi(2^{k}(x_{1}-x^{\prime}))\|W_{1}(x^{\prime},\cdot)\|_{L_{t^{\prime}}^{2}}^{2}
\psi(2^{k}(x^{\prime}-x_{4}))dx^{\prime}\right)^{2}\nonumber\\
&&\|W_{2}(t^{\prime\prime},x_{4})\|_{L_{t^{\prime\prime}}^{2}}^{2}dx_{1}
dx_{4}\nonumber\\
&&= C\int_{\SR}\int_{\SR}\|W_{2}(x_{1},\cdot)\|_{L_{t}^{2}}^{2}
|N_{k}(x_{1},x_{4})|^{2}\|W_{2}(x_{4},\cdot)\|_{L_{t}^{2}}^{2}dx_{1}dx_{4},
\end{eqnarray*}
where
$$N_{k}(x_{1},x_{4})=\int_{\SR}2^{k}\psi(2^{k}(x_{1}-x^{\prime}))
\|W_{1}(x^{\prime},\cdot)\|_{L_{t}^{2}}^{2}\psi(2^{k}(x^{\prime}-x_{4}))dx^{\prime}.$$
According to the definition of $N_{k}(x_{1}, x_{4})$, we have
\begin{align}
&\|W_{1}T_{k}W_{2}\|_{\mathfrak{S}^{4}}^{4}\leq CZ_{k}(h_{2},h_{1},h_{1},h_{2}),\label{5.05}
\end{align}
where $h_{l}=\|W_{l}(\cdot,x)\|_{L_{t}^{2}}^{2}(l=1,2)$ and $Z_{k}$  is defined by
\begin{align*}
Z_{k}(g_{1},g_{2},g_{3},g_{4})&=2^{2k}\int_{\SR^{4}}\psi(2^{k}(x_{1}-x_{2}))\psi(2^{k}(x_{1}-x_{3}))\\
&\quad\times\psi(2^{k}(x_{2}-x_{4}))\psi(2^{k}(x_{3}-x_{4}))\prod_{i=1}^{4}g_{i}(x_{i})dx_{i}.
\end{align*}
Since the other points in the claimed region can be obtained with the aid of Lemma 3.2,
 then we only prove that \eqref{5.02}  is valid at the three points $( p_1, p_2) = (2, 2), (4, \infty)$,
  and $(\infty, 4).$ For $( p_1, p_2) = (2, 2)$, due to $\|\psi\|_{L^{\infty}}\leq C$,
  by using \eqref{5.05} and H\"{o}lder inequality, we have
\begin{align}
&&\|W_{1}T_{k}W_{2}\|_{\mathfrak{S}^{4}}^{4}\leq C2^{2k}\int_{\SR^{4}}\prod_{i=1}^{4}g_{i}(x_{i})dx_{i}\leq C2^{2k}\|W_{1}\|_{L_{x}^{2}L_{t}^{2}}^{4}\|W_{2}\|_{L_{x}^{2}L_{t}^{2}}^{4}.\label{5.06}
\end{align}
For $( p_1, p_2) = (4, \infty), (\infty,4)$,  it only need to show \eqref{5.02}
 at $( p_1, p_2)=(4, \infty)$ due to symmetry. Since  $\|\psi\|_{L^{\infty}}\leq C$ and
$\|\psi(2^{k}\cdot)\|_{L^{1}}\sim 2^{-k}$, by using H\"{o}lder inequality and Young's
convolution inequality, we get
\begin{align}
|Z_{k}(g_{1},g_{2},g_{3},g_{4})|&\leq C2^{2k}\|g_{1}\|_{L^{\infty}}\|g_{3}\|_{L^{\infty}}
\|g_{4}\|_{L^{\infty}}\nonumber\\
&\quad\times\int_{\SR^{4}}\psi(2^{k}(x_{1}-x_{2}))\psi(2^{k}(x_{1}-x_{3}))
\psi(2^{k}(x_{2}-x_{4}))g_{2}(x_{2})\prod_{i=1}^{4}dx_{i}\nonumber\\
&\leq C2^{2k}\|\psi(2^{k}\cdot)\|_{L^{1}}^{3}\|g_{1}\|_{L^{\infty}}
\|g_{2}\|_{L^{1}}\|g_{3}\|_{L^{\infty}}\|g_{4}\|_{L^{\infty}}\nonumber\\
&\leq C2^{-k}\|g_{1}\|_{L^{\infty}}\|g_{2}\|_{L^{1}}\|g_{3}\|_{L^{\infty}}
\|g_{4}\|_{L^{\infty}}.\label{5.07}
\end{align}
By symmetry, we can get
\begin{align}
&&|Z_{k}(g_{1},g_{2},g_{3},g_{4})|
\leq C2^{-k}\|g_{1}\|_{L^{\infty}}\|g_{2}\|_{L^{\infty}}
\|g_{3}\|_{L^{1}}\|g_{4}\|_{L^{\infty}}.\label{5.08}
\end{align}
From \eqref{5.07} and \eqref{5.08}, by using Exercise 13, Page18 of \cite{BL1976}, we have
\begin{align}
&&|Z_{k}(g_{1},g_{2},g_{3},g_{4})|
\leq C2^{-k}\|g_{1}\|_{L^{\infty}}\|g_{2}\|_{L^{2}}\|g_{3}\|_{L^{2}}
\|g_{4}\|_{L^{\infty}}.\label{5.09}
\end{align}
From \eqref{5.05} and \eqref{5.09}, we obtain
\begin{align}
\|W_{1}T_{k}W_{2}\|_{\mathfrak{S}^{4}}^{4}&\leq CZ_{k}(h_{2},h_{1},h_{1},h_{2})\nonumber\\
&\leq C2^{-k}\|h_{1}\|_{L^{2}}^{2}\|h_{2}\|_{L^{\infty}}^{2}\nonumber\\
&\leq C 2^{-k}\|W_{1}\|_{L_{x}^{4}L_{t}^{2}}^{4}\|W_{2}\|_{L_{x}^{\infty}L_{t}^{2}}^{4},\label{5.010}
\end{align}
which shows that \eqref{5.02} is valid at point $( p_1, p_2) = (4, \infty)$.

This completes the proof of Lemma 5.1.

\end{proof}

\begin{Lemma}\label{lem5.2}
Let $A_{0}=B_{0}=L_{x}^{r_{0}}L_{t}^{2}$, $ A_{1}=B_{1}=L_{x}^{r_{1}}L_{t}^{2}$,
$C_{0}=\ell_{\mu(r_{0},r_{0})}^{\infty},\, C_{1}=\ell_{\mu(r_{1},r_{0})}^{\infty}$
with $\frac{1}{r_{0}}=\frac{1}{4}-\delta$, $\frac{1}{r_{1}}=\frac{1}{4}+2\delta$,
$\mu(r_{0},r_{0})=\frac{1}{2}-\frac{2}{r_{0}}$, $\mu(r_0, r_1)=\frac{1}{2}-\frac{1}{r_{0}}-\frac{1}{r_{1}}$
 and  $T$ be a bilinear operator such that

\noindent $T: \,A_{0}\times B_{0}\rightarrow C_{0}$,
which is
\begin{eqnarray*}
&&\sup_{k}2^{-\mu(r_{0},r_{0})}\|W_{1}T_{k}W_{2}\|_{\mathfrak{S}^{\beta^{\prime}}}\leq C\|W_{1}\|_{L_{x}^{r_{0}}L_{t}^{2}}\|W_{2}\|_{L_{x}^{r_{0}}L_{t}^{2}}.
\end{eqnarray*}
$T:\, A_{0}\times B_{1}\rightarrow C_{1}$,
which is
\begin{eqnarray*}
&&\sup_{k}2^{-\mu(r_{0},r_{1})}\|W_{1}T_{k}W_{2}\|_{\mathfrak{S}^{\beta^{\prime}}}\leq C\|W_{1}\|_{L_{x}^{r_{0}}L_{t}^{2}}\|W_{2}\|_{L_{x}^{r_{1}}L_{t}^{2}}.
\end{eqnarray*}
$T:\,A_{1}\times B_{0}\rightarrow C_{1}$,
which is
\begin{eqnarray*}
&&\sup_{k}2^{-\mu(r_{1},r_{0})}\|W_{1}T_{k}W_{2}\|_{\mathfrak{S}^{\beta^{\prime}}}\leq C\|W_{1}\|_{L_{x}^{r_{1}}L_{t}^{2}}\|W_{2}\|_{L_{x}^{r_{0}}L_{t}^{2}}.
\end{eqnarray*}
Then, we have $T:\,L_{x}^{4,2}L_{t}^{2}\times L_{x}^{4,2}L_{t}^{2}\rightarrow \ell_{0}^{1}$,
which is
\begin{eqnarray*}
&&\sum_{k\geq1}\|W_{1}T_{k}W_{2}\|_{\mathfrak{S}^{\beta^{\prime}}}\leq C\|W_{1}\|_{L_{x}^{4,2}L_{t}^{2}}\|W_{2}\|_{L_{x}^{4,2}L_{t}^{2}}.
\end{eqnarray*}
\end{Lemma}

\begin{proof}
By using Lemmas 2.15, 2.16, 2.17, we have $(A_{0},A_{1})_{\frac{\theta}{2},2}\times(B_{0},B_{1})_{\frac{\theta}{2},2}
\rightarrow (C_{0},C_{1})_{\theta,1}$, which is
\begin{eqnarray}
&&\sum_{k\geq1}2^{-(1-\frac{\theta}{2})\mu(r_{0},r_{0})k}2^{-\frac{\theta}{2}\mu(r_{1},r_{0})k}
\|W_{1}T_{k}W_{2}\|_{\mathfrak{S}^{\beta^{\prime}}}\leq C\|W_{1}\|_{L_{x}^{q_{1},2}L_{t}^{2}}\|W_{2}\|_{L_{x}^{q_{2},2}L_{t}^{2}}.\label{5.011}
\end{eqnarray}
Here $\frac{1}{q_{1}}=\frac{1-\frac{\theta}{2}}{r_{0}}+\frac{\theta}{r_{1}}$, $\frac{1}{q_{2}}=\frac{1-\frac{\theta}{2}}{r_{0}}+\frac{\theta}{r_{1}}$.

\noindent From \eqref{5.011}, we take $\theta=\frac{2}{3}$,  since $2\delta=\mu(r_0, r_0)\neq \mu(r_0, r_1)=-\delta$, $\frac{1}{3}\mu(r_0, r_0)+\frac{2}{3}\mu(r_1, r_0)=0$,  we have
\begin{eqnarray}
(L_{x}^{r_{0}}L_{t}^{2},L_{x}^{r_{1}}L_{t}^{2})_{\frac{1}{3},2}
\times(L_{x}^{r_{0}}L_{t}^{2},L_{x}^{r_{1}}L_{t}^{2})_{\frac{1}{3},2}=L_{x}^{4,2}L_{t}^{2}\times L_{x}^{4,2}L_{t}^{2}
\rightarrow (C_{0},C_{1})_{\frac{2}{3},1}=\ell_{0}^{1},\label{5.012}
\end{eqnarray}
which is
\begin{eqnarray}
&&\sum_{k\geq1}\|W_{1}T_{k}W_{2}\|_{\mathfrak{S}^{\beta^{\prime}}}\leq C\|W_{1}\|_{L_{x}^{4,2}L_{t}^{2}}\|W_{2}\|_{L_{x}^{4,2}L_{t}^{2}}.\label{5.013}
\end{eqnarray}
From \eqref{5.012} and \eqref{5.013}, we have that Lemma 5.2 is valid.

The proof of Lemma 5.2 is finished.

\end{proof}

\noindent{\bf Proof of Theorem  1.1}

Now, we prove Theorem 1.1.  By using Lemma 2.13, it suffices to prove that
\begin{align}
&\|WUU^{\ast}\overline{W}\|_{\mathfrak{S}^{\beta^{\prime}}}\leq C
\|W\|_{L_{x}^{4,2}L_{t}^{2}(\mathbf{R}\times I)}^{2}.\label{5.014}
\end{align}
First, we consider the operator $WT_{k}\overline{W}$ $(k\in \Z)$. From \eqref{5.01}
 and \eqref{5.02},
by using Lemma 3.2,     for
$\beta^{\prime}\in(2,4]$,   $\exists\delta(\delta^{\prime})>0$, we have that
$\lim\limits_{\beta^{\prime}\rightarrow2}\delta(\delta^{\prime})=0$
and
\begin{align}
&&\|W_{1}T_{k}W_{2}\|_{\mathfrak{S}^{\beta^{\prime}}}\leq C
2^{(\frac{1}{p_{3}}+\frac{1}{p_{4}}-\frac{1}{2})k}
\|W_{1}\|_{L_{x}^{p_{3}}L_{t}^{2}(\mathbf{R}\times I)}
\|W_{2}\|_{L_{x}^{p_{4}}L_{t}^{2}(\mathbf{R}\times I)},\label{5.015}
\end{align}
where all $p_3, p_4$ satisfy $|(\frac{1}{p_{3}}, \frac{1}{p_{4}})-(\frac{1}{4},
\frac{1}{4})|\leq \delta(\beta^{\prime})$.

Now fix $\beta_{\ast}< 2$ satisfying $\beta_{\ast}^{\prime}\in(2, 4)$.
We define
$$\mathcal{T}(W_{1},W_{2}):=(W_{1}T_{k}W_{2})_{k\in \z}.$$
From  \eqref{5.011}, we know that  $\mathcal{T}: L_{x}^{p_{3}}L_{t}^{2}\times L_{x}^{p_4}L_{t}^{2}\rightarrow \ell_{\mu(p_{3},p_{4})}^{\infty}(\mathfrak{S}^{\beta_{\ast}^{\prime}})$
is bounded for all $(\frac{1}{p_{3}},\frac{1}{p_{4}})$ in
$\delta(\beta_{\ast}^{\prime})$-neighborhood of $(\frac{1}{4}, \frac{1}{4})$.
 Here, if $k\geq1$, $\mu(p_3, p_4)=\frac{1}{2}-\frac{1}{p_{3}}-\frac{1}{p_{4}}$;
 if $k\leq-1$, $\mu(p_3, p_4)=\frac{1}{p_{3}}+\frac{1}{p_{4}}-\frac{1}{2}$.
Then, we have that
$\mathcal{T}=(T_k)_{k}$ is a bilinear operator such that

\noindent $\mathcal{T}: \,A_{0}\times B_{0}\rightarrow C_{0}$,
which is
\begin{eqnarray*}
&&\sup_{k}2^{-\mu(r_{0},r_{0})}\|W_{1}T_{k}W_{2}\|_{\mathfrak{S}^{\beta^{\prime}}}\leq C\|W_{1}\|_{L_{x}^{r_{0}}L_{t}^{2}}\|W_{2}\|_{L_{x}^{r_{0}}L_{t}^{2}}.
\end{eqnarray*}
$\mathcal{T}:\, A_{0}\times B_{1}\rightarrow C_{1}$,
which is
\begin{eqnarray*}
&&\sup_{k}2^{-\mu(r_{0},r_{1})}\|W_{1}T_{k}W_{2}\|_{\mathfrak{S}^{\beta^{\prime}}}\leq C\|W_{1}\|_{L_{x}^{r_{0}}L_{t}^{2}}\|W_{2}\|_{L_{x}^{r_{1}}L_{t}^{2}}.
\end{eqnarray*}
$\mathcal{T}:\,A_{1}\times B_{0}\rightarrow C_{1}$,
which is
\begin{eqnarray*}
&&\sup_{k}2^{-\mu(r_{1},r_{0})}\|W_{1}T_{k}W_{2}\|_{\mathfrak{S}^{\beta^{\prime}}}\leq C\|W_{1}\|_{L_{x}^{r_{1}}L_{t}^{2}}\|W_{2}\|_{L_{x}^{r_{0}}L_{t}^{2}}.
\end{eqnarray*}
where
$$A_{0}=B_{0}=L_{x}^{r_{0}}L_{t}^{2},\, A_{1}=B_{1}=L_{x}^{r_{1}}L_{t}^{2}\,
\,C_{0}=\ell_{\mu(r_{0},r_{0})}^{\infty},\, C_{1}=\ell_{\mu(r_{1},r_{0})}^{\infty}.$$
In particular, we choose $r_0, r_1$ such that $\frac{1}{r_{0}}=\frac{1}{4}-\delta$
 and $\frac{1}{r_{1}}=\frac{1}{4}+2\delta$, where $\delta> 0$ is sufficiently small.
When $k\geq1$, by using Lemma 5.2, we have
\begin{align*}
&\mathcal{T}:\,L_{x}^{4,2}L_{t}^{2}\times L_{x}^{4,2}L_{t}^{2}\rightarrow \ell_{0}^{1},
\end{align*}
which means that
\begin{align}
&\sum_{k\geq1}\|WT_{k}\overline{W}\|_{\mathfrak{S}^{\beta_{\ast}^{\prime}}}\leq C
\|W\|_{L_{x}^{4,2}L_{t}^{2}(\mathbf{R}\times I)}^{2}.\label{5.016}
\end{align}
When $k\leq-1$, we denote $j=-k\geq1$ and $O_{j}=T_{-j}$, from \eqref{5.015}, we have
\begin{eqnarray}
\|W_{1}O_{j}W_{2}\|_{\mathfrak{S}^{\beta^{\prime}}}=\|W_{1}T_{-j}W_{2}\|_{\mathfrak{S}^{\beta^{\prime}}}\leq C2^{-(\frac{1}{p_{3}}+\frac{1}{p_{4}}-\frac{1}{2})j}
\|W_{1}\|_{L_{x}^{p_{3}}L_{t}^{2}(\mathbf{R}\times I)}
\|W_{2}\|_{L_{x}^{p_{4}}L_{t}^{2}(\mathbf{R}\times I)}.\label{5.017}
\end{eqnarray}
From \eqref{5.017}, by using a  proof similar to \eqref{5.016}, we have
\begin{align}
&\sum_{k\leq-1}\|WT_{k}\overline{W}\|_{\mathfrak{S}^{\beta_{\ast}^{\prime}}}
=\sum_{j\geq1}\|WO_{j}\overline{W}\|_{\mathfrak{S}^{\beta_{\ast}^{\prime}}}\leq C
\|W\|_{L_{x}^{4,2}L_{t}^{2}(\mathbf{R}\times I)}^{2}.\label{5.018}
\end{align}
When $k=0$, by using a  proof similar to \eqref{5.01}, we have
\begin{align}
\|W_{1}T_{0}W_{2}\|_{\mathfrak{S}^{2}}^{2}&=\int_{\SR\times I}\int_{\SR\times I}|W_{1}(t,x)|^{2}|K_{0}(t-t^{\prime},x-x^{\prime})|^{2}|W_{2}(t^{\prime},x^{\prime})|^{2}dtdt^{\prime}dx
dx^{\prime}\nonumber\\
&\leq C\|W_{1}\|_{L_{x}^{4}L_{t}^{2}}^{2}\left\|\int_{\SR}\frac{\psi^{2}((x-x^{\prime}))
\|W_{2}(x^{\prime},\cdot)\|_{L_{t}^{2}}^{2}}
{|x-x^{\prime}|}dx^{\prime}\right\|_{L_{x}^{2}}\nonumber\\
&\leq C\|W_{1}\|_{L_{x}^{4}L_{t}^{2}}^{2}\|W_{2}\|_{L_{x}^{4}L_{t}^{2}}^{2}.\label{5.019}
\end{align}
From \eqref{5.019}, since $\beta_{\ast}^{\prime}\in (2,4)$, we have
\begin{align*}
\|W_{1}T_{0}W_{2}\|_{\mathfrak{S}^{\beta_{\ast}^{\prime}}}\leq
\|W_{1}T_{0}W_{2}\|_{\mathfrak{S}^{2}}\leq C\|W_{1}\|_{L_{x}^{4}L_{t}^{2}}\|W_{2}\|_{L_{x}^{4}L_{t}^{2}}\leq C\|W\|_{L_{x}^{4,2}L_{t}^{2}(\mathbf{R}\times I)}^{2}.
\end{align*}
Hence, by using the triangle inequality, we have
\begin{align}
\|WUU^{\ast}\overline{W}\|_{\mathfrak{S}^{\beta_{\ast}^{\prime}}}\leq
\sum_{k}\|WT_{k}\overline{W}\|_{\mathfrak{S}^{\beta_{\ast}^{\prime}}}\leq C
\|W\|_{L_{x}^{4,2}L_{t}^{2}(\mathbf{R}\times I)}^{2}.\label{5.020}
\end{align}
From \eqref{5.020}, we have that \eqref{5.014} is valid, for all
 $\beta_{\ast}^{\prime}\in(2, 4)$.

The proof of Theorem 1.1 is finished.

\bigskip

\section{ Proof of Theorem 1.2: Pointwise convergence of density function on $\mathbf{R}$}

\setcounter{equation}{0}

\setcounter{Theorem}{0}

\setcounter{Lemma}{0}

\setcounter{section}{6}

In this section, we present the pointwise convergence of density function on $\mathbf{R}$.

\noindent{\bf Proof.} Inspired by \cite{BLN2020},
let $\gamma(t)=e^{it\sqrt{\partial_{x}^{4}-\partial_{x}^{2}}}
\gamma_{0}e^{it\sqrt{\partial_{x}^{4}-\partial_{x}^{2}}}$, for $\beta<2$
 and $\gamma_{0}\in \mathfrak{S}^{\beta}(\dot{H}^{\frac{1}{4}})$,
  to prove \eqref{1.038}, it is enough to prove
\begin{eqnarray}
&\left\|\limsup\limits_{t\rightarrow0}|\rho_{\gamma(t)}-\rho_{\gamma_{0}}|
\right\|_{L_{x}^{2,\infty}}=0.\label{6.01}
\end{eqnarray}
From the discussion in \cite{BLN2020}, we define $\Pi_{g}:\dot{H}^{\frac{1}{4}}
\rightarrow\dot{H}^{\frac{1}{4}}$
 to be the orthogonal projection onto the span of $g$ given by
  $\Pi_{g}\phi=\langle\phi,g\rangle g$, $g\in \dot{H}^{\frac{1}{4}}$ and $\|g\|_{\dot{H}^{\frac{1}{4}}}=1$.
   Then, we can approximate
   $\gamma_{0}=\sum\limits_{j=1}^{\infty}\lambda_{j}\Pi_{g_{j}}$ by the finite-rank operator $\gamma_{0}^{N}=\sum\limits_{j=1}^{N}\lambda_{j}\Pi_{g_{j}}$,
    and define $\gamma^{N}(t)=e^{it\sqrt{\partial_{x}^{4}-\partial_{x}^{2}}}
    \gamma_{0}^{N}e^{it\sqrt{\partial_{x}^{4}-\partial_{x}^{2}}}$ for orthonormal
     systems $g_{j}\in \dot{H}^{\frac{1}{4}}$.  By using Theorem 1.1, we have
\begin{align}
\|\rho_{\gamma(t)}-\rho_{\gamma^{N}(t)}\|_{L_{x}^{2,\infty}L_{t}^{\infty}}&=
\Big\|\sum_{j=N+1}^{\infty}\lambda_{j}|e^{it\sqrt{-\partial_{x}^{2}+\partial_{x}^{4}}}g_{j}|^{2}
\Big\|_{L_{x}^{2,\infty}L_{t}^{\infty}}\nonumber\\
& \leq \left(\sum_{j=N+1}^{\infty}|\lambda_{j}|^{\beta}\right)^{\frac{1}{\beta}}.\label{6.02}
\end{align}
By using \eqref{6.02} and  $\lambda\in \ell^{\beta}$, then we have
\begin{eqnarray}
&\lim\limits_{N\rightarrow\infty}\left\|\rho_{\gamma(t)}-\rho_{\gamma^{N}(t)}
\right\|_{L_{x}^{2,\infty}L_{t}^{\infty}}=0.\label{6.03}
\end{eqnarray}
From the definition of $\rho_{\gamma_0}$ and \eqref{6.03}, for any $\epsilon>0$,
 we can find $N_{\epsilon}$ such that
\begin{eqnarray}
&\left\|\rho_{\gamma_{0}}-\rho_{\gamma_{0}^{N_{\epsilon}}}\right\|_{L_{x}^{2}L_{t}^{\infty}},
\,\,\left\|\rho_{\gamma(t)}-\rho_{\gamma^{N_{\epsilon}}(t)}\right\|
_{L_{x}^{2,\infty}L_{t}^{\infty}}<\epsilon.\label{6.04}
\end{eqnarray}
From \eqref{6.04}, for such $N_{\epsilon}$, we have
\begin{align}
\left\|\limsup_{t\rightarrow0}|\rho_{\gamma(t)}-\rho_{\gamma_{0}}|\right\|_{L_{x}^{2,\infty}}&
\leq \left\|\limsup_{t\rightarrow0}|\rho_{\gamma(t)}-\rho_{\gamma^{N_{\epsilon}}(t)}
|\right\|_{L_{x}^{2,\infty}}
+\left\|\limsup_{t\rightarrow0}|\rho_{\gamma^{N_{\epsilon}}(t)}-
\rho_{\gamma_{0}^{N_{\epsilon}}}|\right\|_{L_{x}^{2}}\nonumber\\
&+\left\|\rho_{\gamma_{0}^{N_{\epsilon}}}-\rho_{\gamma_{0}}|\right\|_{L_{x}^{2}}\nonumber\\
&\leq 2\epsilon+\left\|\limsup_{t\rightarrow0}|\rho_{\gamma^{N_{\epsilon}}(t)}-
\rho_{\gamma_{0}^{N_{\epsilon}}}|\right\|_{L_{x}^{2}}.\label{6.05}
\end{align}
Since $g_{j}\in \dot{H}^{\frac{1}{4}}(\R)$, it follows from Lemma 2.1 that
\begin{eqnarray}
&\limsup\limits_{t\rightarrow0}\rho_{\gamma^{N_{\epsilon}}(t)}=
\sum\limits_{j=1}^{N_{\epsilon}}
\lambda_{j}\limsup\limits_{t\rightarrow0}\left|
e^{it\sqrt{\partial_{x}^{4}-\partial_{x}^{2}}}g_{j}(x)\right|^{2}=
\sum\limits_{j=1}^{N_{\epsilon}}
\lambda_{j}|g_{j}(x)|^{2}
=\rho_{\gamma_{0}^{N_{\epsilon}}}(x)\label{6.06}
\end{eqnarray}
holds almost everywhere.
From \eqref{6.05} and \eqref{6.06}, we have that \eqref{6.01} is valid.

This completes the proof of Theorem 1.2.

\section{ Proof of Theorem  1.3: Maximal-in-time estimates for orthonormal
functions on $\mathbf{B}^{d}$}

\setcounter{equation}{0}

\setcounter{Theorem}{0}

\setcounter{Lemma}{0}

\setcounter{section}{7}

In this section, we present the maximal-in-time estimates for orthonormal
functions on $\mathbf{B}^{d}(d\geq1)$.

We define
\begin{eqnarray*}
&&U(t)f=e^{it\sqrt{\partial_{x}^{4}-\partial_{x}^{2}}}D_{x}^{-s}f.
\end{eqnarray*}
Then, we have
\begin{eqnarray*}
&&U^{\ast}(t)g=\int_{I}e^{-it\sqrt{\partial_{x}^{4}-\partial_{x}^{2}}}D_{x}^{-s}gdt.
\end{eqnarray*}
By using the homogeneous dyadic decomposition, we have
\begin{eqnarray*}
&\sum\limits_{k\in\mathbf{Z}}\psi^{2}(2^{k}|\xi|)=1,
\end{eqnarray*}
where $\psi \in C_{c}^{\infty}(\{x:\frac{1}{2}<|\xi|<2\})$. We
denote $T_{k}=P_{k}^{2}U$, where $\widehat{P_{k}f}(\xi)=\psi(2^{-k}|\xi|)\widehat{f}(\xi)$.
we decompose the operator $UU^{\ast}$ as follows
\begin{align*}
T_{k}F&=\int F(t^{\prime},x^{\prime})O_{k}(t-t^{\prime},x-x^{\prime})dt^{\prime}d\mu(x^{\prime})\\
&=\sum_{l\geq0}\int  F(t^{\prime},x^{\prime})O_{k,l}(t-t^{\prime},x-x^{\prime})dt^{\prime}d\mu(x^{\prime})\\
&=\sum_{l\geq0}T_{k,l}F(t,x),
\end{align*}
where $O_{k,l}=\chi_{l}(|x|)O_{k}(t,x)$, $\chi_{l}(|x|)=\chi_{(2^{-l-1},2^{-l})}(|x|)$ and
\begin{eqnarray*}
&&O_{k}(t,x)=\int e^{ix\xi}e^{it\sqrt{|\xi|^{4}+|\xi|^{2}}}\frac{\phi(2^{-k}|\xi|)^{2}}{|\xi|^{2s}}d\xi.
\end{eqnarray*}
We give the following Lemma that  play a key role in proving Lemmas 7.2, 7.3.
\begin{Lemma}\label{lem7.1}
Suppose that  $0<\alpha\leq d$ and $\mu\in\mathcal{M}^{\alpha}(\mathbf{B}^{d})$.
Then, we have
\begin{eqnarray*}
&&\int_{\SR^{d}}\int_{\SR^{d}}|f(x)||g(x^{\prime})|\chi_{(0,2^{-l})}(|x-x^{\prime}|)
d\mu(x)d\mu(x^{\prime})\leq C2^{-\alpha l}\|f\|_{L^{2}(\mathbf{B}^{d},d\mu)}\|g\|_{L^{2}(\mathbf{B}^{d},d\mu)},
\end{eqnarray*}
where $l\in\Z$.
\end{Lemma}

 Lemma 7.1 can be found in  Lemma 6.2 of \cite{BKS2023}.

We consider the following case $\frac{d}{4}<s<\frac{d}{2}$ and $s=\frac{d}{4}$,
respectively.

\noindent When $\frac{d}{4}<s<\frac{d}{2}$, we give the following estimates that
play a key role in proving Theorem 1.3.

\begin{Lemma}\label{lem7.2}
Let $d\geq1$, $\frac{d}{4}<s<\frac{d}{2}$, $0<\alpha\leq d$. Then, we have
\begin{eqnarray}
&&\|WT_{k}\overline{W}\|_{\mathfrak{S}^{1}}\leq C2^{(d-2s)k}\|W\|_{L^{2}(d\mu)L_{t}^{2}}^{2},\label{7.01}\\
&&\|WT_{k,l}\overline{W}\|_{\mathfrak{S}^{2}}\leq C2^{(d-2s)k}\frac{2^{-\frac{\alpha l}{2}}}{(1+2^{k-l})^{\frac{d}{2}}}\|W\|_{L^{4}(d\mu)L_{t}^{2}}^{2},\label{7.02}\\
&&\|WT_{k,l}\overline{W}\|_{\mathfrak{S}^{\infty}}\leq C2^{(d-2s)k}\frac{2^{-\alpha l}}{(1+2^{k-l})^{\frac{d}{2}}}\|W\|_{L^{\infty}(d\mu)L_{t}^{2}}^{2}.\label{7.03}
\end{eqnarray}
\end{Lemma}
\begin{proof}
For \eqref{7.01}, since $T_{k}=P_{k}U(P_{k}U)^{\ast}$, this is equivalent to
\begin{eqnarray*}
&&\left\|\sum\limits_{j}\lambda_{j}|P_{k}Uf_{j}|^{2}\right\|_{L^{\infty}(d\mu)L^{\infty}}
\leq C2^{(d-2s)k}\|\lambda\|_{\ell^{\infty}}
\end{eqnarray*}
for orthonormal systems $(f_{j})_{j}$ in $L^{2}(\R^{d})$.
 For \eqref{7.02}, by using \eqref{2.05}, we have
\begin{eqnarray}
&&|O_{k,l}(t,x)|\leq C\chi_{l}(|x|)2^{(d-2s)k}\frac{1}{(1+2^{k-l})^{\frac{d}{2}}}.\label{7.04}
\end{eqnarray}
By using \eqref{7.04} and Lemma 7.1, we have
\begin{eqnarray}
&&\|WT_{k,l}\overline{W}\|_{\mathfrak{S}^{2}}^{2}\nonumber\\&&
=\int_{\SR^{d}}\int_{\SR}\int_{\SR^{d}}\int_{\SR}|W_{1}(x,t)|^{2}|O_{k,l}(x-x^{\prime},
t-t^{\prime})|^{2}
|W_{2}(x^{\prime},t^{\prime})|^{2}d\mu(x^{\prime})d\mu(x)dt^{\prime}dt
\nonumber\\
&&\leq\frac{ 2^{2(d-2s)k}}{(1+2^{k-l})^{d}}\int_{\SR^{d}}\int_{\SR}\int_{\SR^{d}}\int_{\SR} |W_{1}(x,t)|^{2}|\chi_{(2^{-l-1},2^{-l})}(|x-x^{\prime}|)|^{2}|W_{2}(x^{\prime},t^{\prime})|^{2}
d\mu(x^{\prime})dt^{\prime}d\mu(x)dt\nonumber\\
&&\leq \frac{ 2^{2(d-2s)k}}{(1+2^{k-l})^{d}}\int_{\SR^{d}}\int_{\SR^{d}} \|W_{1}(x,t)\|_{L_{t}^{2}}^{2}|\chi_{(2^{-l-1},2^{-l})}(|x-x^{\prime}|)|^{2}\|W_{2}(x^{\prime},
t^{\prime})\|_{L_{t^{\prime}}^{2}}^{2}
d\mu(x^{\prime})d\mu(x)\nonumber\\
&&\leq 2^{2(d-2s)k}\frac{2^{-\alpha l}}{(1+2^{k-l})^{d}}\|W_{1}\|_{L^{4}(d\mu)L_{t}^{2}}^{2}
\|W_{2}\|_{L^{4}(d\mu)L_{t}^{2}}^{2}.\label{7.05}
\end{eqnarray}
For \eqref{7.03}, by using \eqref{7.04}, we have
\begin{eqnarray}
&&\|WT_{k,l}\overline{W}\|_{\mathfrak{S}^{\infty}}=\sup\limits_{\|f_{1}\|_{L^{2}}=
\|f_{2}\|_{L^{2}}=1}\left|\int_{\SR^{d}}\int_{\SR} f_{1}W(t,x)T_{k,l}\overline{W}f_{2}
(x,t)dtd\mu(x)\right|\nonumber\\
&&=\sup\limits_{\|f_{1}\|_{L^{2}}=\|f_{2}\|_{L^{2}}=1}\left|\int_{\SR^{d}}\int_{\SR} f_{1}W(t,x)\int_{\SR^{d}}\int_{\SR}O_{k,l}(x-x^{\prime},t-t^{\prime})
\overline{W}(x^{\prime},t^{\prime})f_{2}(x^{\prime},t^{\prime})\right.\nonumber\\
&&\left.d\mu(x^{\prime})
dt^{\prime}d\mu(x)dt\right|\nonumber\\
&&\leq C\frac{2^{(d-2s)k}}{(1+2^{k-l})^{\frac{d}{2}}}\int \|f_{1}W(\cdot,x)\|_{L_{t}^{1}}\|f_{2}\overline{W}(\cdot,x^{\prime})\|_{L_{t^{\prime}}^{1}}
\chi_{l}(|x-x^{\prime}|)d\mu(x)d\mu(x^{\prime})\nonumber\\
&&\leq C2^{(d-2s)k}\frac{2^{-\alpha l}}{(1+2^{k-l})^{\frac{d}{2}}}\|W\|_{L^{\infty}(d\mu)
L_{t}^{2}}^{2}.\label{7.06}
\end{eqnarray}

The proof of Lemma 7.2 is completed.
\end{proof}

\noindent{\bf Proof of Theorem  1.3: $\frac{d}{4}<s<\frac{d}{2}.$}

Inspired by \cite{BKS2023}, we prove Theorem 1.3.  By using Lemma 2.14, it suffices to prove
\begin{eqnarray}
&&\|WA\overline{W}\|_{\mathfrak{S}^{\beta^{\prime}}}\leq C\|W\|_{L^{\infty}(d\mu)L_{t}^{2}}^{2},\label{7.07}
\end{eqnarray}
where $A=UU^{\ast}$.
To prove \eqref{7.07}, we just need to prove
\begin{eqnarray}
&&\sum\limits_{k\in\z}\|WT_{k}\overline{W}\|_{\mathfrak{S}^{\beta^{\prime}}}\leq C
\|W\|_{L^{\infty}(d\mu)L_{t}^{2}}^{2}.\label{7.08}
\end{eqnarray}
When $k<0$, by using \eqref{7.01}, H\"{o}lder inequality with respect to $d\mu$,
since $\frac{d}{4}<s<\frac{d}{2}$, we have
\begin{align}
\sum\limits_{k<0}\|WT_{k}\overline{W}\|_{\mathfrak{S}^{\beta^{\prime}}}&\leq C\sum\limits_{k<0}\|WT_{k}\overline{W}\|_{\mathfrak{S}^{1}}\nonumber\\
&\leq C\sum\limits_{k<0}2^{(d-2s)k}\|W\|_{L^{2}(d\mu)L_{t}^{2}}^{2}\nonumber\\
&\leq C\|W\|_{L^{2}(d\mu)L_{t}^{2}}^{2}\leq C\|W\|_{L^{\infty}(d\mu)L_{t}^{2}}^{2}.\label{7.09}
\end{align}
When $k>0$, we consider the following case $\frac{\alpha}{d-2s}>2$ and $\frac{\alpha}{d-2s}\leq2$, respectively.
\noindent By using \eqref{7.02}, we have
\begin{align}
\|WT_{k}\overline{W}\|_{\mathfrak{S}^{2}}&\leq C\left(\sum\limits_{0\leq l\leq k}\|WT_{k,l}
\overline{W}\|_{\mathfrak{S}^{2}}+\sum\limits_{l\geq k\geq0}\|WT_{k,l}
\overline{W}\|_{\mathfrak{S}^{2}}\right)\nonumber\\
&\leq  C2^{(d-2s)k}\|W\|_{L^{4}(d\mu)L_{t}^{2}}^{2}\left(2^{-\frac{d}{2}k}\sum\limits_{0\leq
 l\leq k}2^{-\frac{(\alpha-d)l}{2}}+\sum\limits_{l\geq k\geq0}2^{-\frac{\alpha l}{2}}\right)\nonumber\\
&\leq Ck2^{(d-2s-\frac{\alpha}{2})k}\|W\|_{L^{4}(d\mu)L_{t}^{2}}^{2},\label{7.010}
\end{align}
where we use the following inequality
\begin{eqnarray*}
&&\sum\limits_{0\leq l\leq k}2^{-\frac{(\alpha-d)l}{2}}\leq Ck2^{-\frac{(\alpha-d)k}{2}},\,
\sum\limits_{l\geq k\geq0}2^{-\frac{\alpha l}{2}}\leq C2^{-\frac{\alpha k}{2}}.
\end{eqnarray*}
{\bf Case 1}: $\frac{\alpha}{d-2s}>2$, we consider $2<\beta<\frac{\alpha}{d-2s}$. It
 follows from the interpolation theorem in $\mathfrak{S}^{\beta^{\prime}}$
and  H\"{o}lder inequality with respect to $d\mu$ and $2<\beta<\frac{\alpha}{d-2s}$
that
\begin{align}
\sum\limits_{k>0}\|WT_{k}\overline{W}\|_{\mathfrak{S}^{\beta^{\prime}}}&\leq C\sum\limits_{k>0}\|WT_{k}\overline{W}\|_{\mathfrak{S}^{1}}^{\frac{2}{\beta^{\prime}}-1}
\|WT_{k}\overline{W}\|
_{\mathfrak{S}^{2}}^{2-\frac{2}{\beta^{\prime}}}\nonumber\\
&\leq C2^{(\frac{2}{\beta^{\prime}}-1)(d-2s)k}
k^{2-\frac{2}{\beta^{\prime}}}2^{(2-\frac{2}{\beta^{\prime}})(d-2s-\frac{\alpha}{2})k}
\|W\|_{L^{2}(d\mu)L_{t}^{2}}^{2(\frac{2}{\beta^{\prime}}-1)}\|W\|_{L^{4}(d\mu)L_{t}^{2}}^{2(2-\frac{2}{\beta^{\prime}})}\nonumber\\
&\leq C\sum\limits_{k>0}k^{\frac{2}{\beta}}2^{(d-2s)k}2^{-\frac{2}{\beta}\frac{\alpha}{2}k}\|W\|_{L^{4}(d\mu)L_{t}^{2}}^{2}\nonumber\\
&=C\sum\limits_{k>0}k^{\frac{2}{\beta}}2^{(d-2s-\frac{\alpha}{\beta})k}\|W\|_{L^{4}(d\mu)L_{t}^{2}}^{2}\nonumber\\
&\leq C\|W\|_{L^{4}(d\mu)L_{t}^{2}}^{2}.\label{7.011}
\end{align}
{\bf Case 2}: $\frac{\alpha}{d-2s}\leq2$, we have $1\leq\beta<\frac{\alpha}{d-2s}\leq 2$. It
 follows from the interpolation theorem in
$\mathfrak{S}^{\beta^{\prime}}$  and the H\"{o}lder inequality, we have
\begin{align}
\|WT_{k,l}\overline{W}\|_{\mathfrak{S}^{\beta^{\prime}}}&\leq
\|WT_{k}\overline{W}\|_{\mathfrak{S}^{2}}^{\frac{2}{\beta^{\prime}}}\|WT_{k}\overline{W}\|
_{\mathfrak{S}^{\infty}}^{1-\frac{2}{\beta^{\prime}}}\nonumber\\
&\leq C2^{\frac{2}{\beta^{\prime}}(d-2s)k}
\left(\frac{2^{-\frac{\alpha l}{2}}}{(1+2^{k-l})^{\frac{d}{2}}}\right)^{\frac{2}{\beta^{\prime}}}
2^{(1-\frac{2}{\beta^{\prime}})(d-2s)k}\left(\frac{2^{-\alpha l}}{(1+2^{k-l})^{\frac{d}{2}}}\right)^{1-\frac{2}{\beta^{\prime}}}\nonumber\\
&\qquad\|W\|_{L^{4}(d\mu)L_{t}^{2}}^{2\frac{2}{\beta^{\prime}}}\|W\|_{L^{\infty}(d\mu)L_{t}^{2}}^{2(1-\frac{2}{\beta^{\prime}})}\nonumber\\
&= C2^{(d-2s)k}\frac{2^{-\frac{1}{\beta^{\prime}}\alpha l}2^{-(1-\frac{2}{\beta^{\prime}})\alpha l}}{(1+2^{k-l})^{\frac{d}{2}}}\|W\|_{L^{4}(d\mu)L_{t}^{2}}^{2\frac{2}{\beta^{\prime}}}
\|W\|_{L^{\infty}(d\mu)L_{t}^{2}}^{2(1-\frac{2}{\beta^{\prime}})}\nonumber\\
&\leq C2^{(d-2s)k}\frac{2^{-\frac{\alpha}{\beta} l}}{(1+2^{k-l})^{\frac{d}{2}}}
\|W\|_{L^{\infty}(d\mu)L_{t}^{2}}^{2}.\label{7.012}
\end{align}
By using \eqref{7.012}, since $\beta<\frac{\alpha}{d-2s}$, by using a proof similar
 to \eqref{7.010}, we have
\begin{align}
\sum\limits_{k>0}\|WT_{k}\overline{W}\|_{\mathfrak{S}^{\beta^{\prime}}}&\leq C\sum\limits_{k>0,l\geq0}\|WT_{k,l}\overline{W}\|_{\mathfrak{S}^{\beta^{\prime}}}\nonumber\\
&\leq C\sum\limits_{k>0,l\geq0}2^{(d-2s)k}\frac{2^{-\frac{\alpha}{\beta} l}}{(1+2^{k-l})^{\frac{d}{2}}}\|W\|_{L_{x}^{\infty}(d\mu)L_{t}^{2}}^{2}\nonumber\\
&\leq C\sum\limits_{k>0}2^{(d-2s-\frac{\alpha}{\beta})k}\|W\|_{L_{x}^{\infty}(d\mu)L_{t}^{2}}^{2}\nonumber\\
&\leq C\|W\|_{L_{x}^{\infty}(d\mu)L_{t}^{2}}^{2}.\label{7.013}
\end{align}
When $k=0$, by using \eqref{7.01} and H\"{o}lder inequality, since $\frac{d}{4}<s<\frac{d}{2}$
 and $\beta<\frac{\alpha}{d-2s}$, we have
\begin{align}
\|WT_{0}\overline{W}\|_{\mathfrak{S}^{\beta^{\prime}}}\leq C\|WT_{0}\overline{W}\|_{\mathfrak{S}^{1}}
\leq C\|W\|_{L^{2}(d\mu)L_{t}^{2}}^{2}\leq C\|W\|_{L^{\infty}(d\mu)L_{t}^{2}}^{2}.\label{7.014}
\end{align}

The proof of Theorem 1.3 is  completed.

\begin{Lemma}\label{lem7.3}
Let $d\geq1$, $s=\frac{d}{4}$, $0<\alpha\leq d$. Then, we have
\begin{eqnarray}
&&\|WT_{l}\overline{W}\|_{\mathfrak{S}^{2}}\leq C2^{(\frac{d}{2}-\frac{\alpha}{2})l}
\|W\|_{L^{\infty}(d\mu)L_{t}^{2}}^{2},\label{7.015}\\
&&\|WT_{l}\overline{W}\|_{\mathfrak{S}^{\infty}}\leq C2^{(\frac{d}{2}-\alpha)l}
\|W\|_{L^{\infty}(d\mu)L_{t}^{2}}^{2}.\label{7.016}
\end{eqnarray}
\end{Lemma}

\begin{proof}
For \eqref{7.015}, by using \eqref{2.06}, we have
\begin{eqnarray}
&&|O_{l}(t,x)|\leq C\chi_{l}(|x|)|x|^{-\frac{d}{2}}\leq C\chi_{l}(|x|)2^{\frac{d}{2}l}.\label{7.017}
\end{eqnarray}
By using \eqref{7.016} and Lemma 7.1, we have
\begin{align}
\|WT_{l}\overline{W}\|_{\mathfrak{S}^{2}}^{2}&=\int|W_{1}(x,t)|^{2}|O_{l}(x-x^{\prime},t-t^{\prime})|^{2}
|W_{2}(x^{\prime},t^{\prime})|^{2}d\mu(x^{\prime})dt^{\prime}d\mu(x)dt
\nonumber\\
&\leq 2^{dl}\int |W_{1}(x,t)|^{2}|\chi_{(2^{-l-1},2^{-l})}(|x-x^{\prime}|)|^{2}
|W_{2}(x^{\prime},t^{\prime})|^{2}d\mu(x^{\prime})dt^{\prime}d\mu(x)dt\nonumber\\
&\leq 2^{dl}2^{-\alpha l}\|W_{1}\|_{L^{4}(d\mu)L_{t}^{2}}^{2}
\|W_{2}\|_{L^{4}(d\mu)L_{t}^{2}}^{2}\nonumber\\
&\leq C 2^{dl}2^{-\alpha l}\|W_{1}\|_{L^{\infty}(d\mu)L_{t}^{2}}^{2}
\|W_{2}\|_{L^{\infty}(d\mu)L_{t}^{2}}^{2} .\label{7.018}
\end{align}
For \eqref{7.016}, by using \eqref{7.017}, we have
\begin{eqnarray}
&&\|WT_{l}\overline{W}\|_{\mathfrak{S}^{\infty}}=\sup\limits_{\|f_{1}\|_{L^{2}}
=\|f_{2}\|_{L^{2}}=1}\left|\int_{\SR^{d}}\int_{\SR} f_{1}W(t,x)T_{l}\overline{W}f_{2}(x,t)dtd\mu(x)\right|\nonumber\\
&&=\sup\limits_{\|f_{1}\|_{L^{2}}=\|f_{2}\|_{L^{2}}=1}\left|\int_{\SR^{d}}\int_{\SR} f_{1}W(t,x)\int_{\SR^{d}}\int_{\SR}O_{l}(x-x^{\prime},t-t^{\prime})
\overline{W}(x^{\prime},t^{\prime})f_{2}(x^{\prime},t^{\prime})\right.\nonumber\\
&&\qquad\left.d\mu(x^{\prime})
dt^{\prime}d\mu(x)dt\right|\nonumber\\
&&\leq C2^{\frac{d}{2}l}\int \|f_{1}W(\cdot,x)\|_{L_{t}^{1}}\|f_{2}\overline{W}(\cdot,x^{\prime})\|_{L_{t^{\prime}}^{1}}
\chi_{l}(|x-x^{\prime}|)d\mu(x)d\mu(x^{\prime})\nonumber\\
&&\leq C2^{\frac{d}{2}l}2^{-\alpha l}\|W\|_{L^{\infty}{d\mu}L_{t}^{2}}^{2}.\label{7.019}
\end{eqnarray}

This completes the proof of Lemma 7.3.

\noindent {\bf Remark3:} When $s=\frac{d}{4}$,  we give the following estimates that  play
 a key role in proving Theorem 1.3.

\end{proof}
\noindent{\bf Proof of Theorem  1.3: $s=\frac{d}{4}.$}

\noindent Since $s=\frac{d}{4}$, $0<\alpha\leq d$,  we have
$\beta<\frac{\alpha}{d-2s}=\frac{2\alpha}{d}\leq 2$.

\noindent It  follows from the interpolation theorem in $\mathfrak{S}^{\beta^{\prime}}$,
 we have
\begin{align}
\sum\limits_{l\geq0}\|WT_{l}\overline{W}\|_{\mathfrak{S}^{\beta^{\prime}}}&\leq C\sum\limits_{l\geq0}\|WT_{l}\overline{W}\|_{\mathfrak{S}^{2}}^{\frac{2}{\beta^{\prime}}}
\|WT_{k}\overline{W}\|
_{\mathfrak{S}^{\infty}}^{1-\frac{2}{\beta^{\prime}}}\nonumber\\
&\leq C\sum\limits_{l\geq0}k^{\frac{2}{\beta}}2^{(\frac{d}{2}-\frac{\alpha}{\beta})l}
\|W\|_{L^{\infty}(d\mu)L_{t}^{2}}^{2}\nonumber\\
&\leq C\|W\|_{L^{4}(d\mu)L_{t}^{2}}^{2}.\label{7.020}
\end{align}

The proof of Theorem 1.3 is completed.

\bigskip

\section { Proof of Theorem  1.4: The Hausdorff dimension of the divergence set}
\setcounter{equation}{0}

\setcounter{Theorem}{0}

\setcounter{Lemma}{0}

\setcounter{section}{8}

In this section, we present the Hausdorff dimension of the divergence set.

By using a proof similar to Corollary 2.3 of \cite{BKS2023}, we have
\begin{eqnarray}
&&\mu(\{x\in\mathbf{B}^{d}:\lim\sup\limits_{t\rightarrow0}
|\rho_{\gamma(t)}(x)-\rho_{\gamma_{0}}(x)|>k^{-1}\})=0,\label{8.01}
\end{eqnarray}
where $k\in\mathbb{N}^{+}$, $\mu\in \mathcal{M}^{\alpha}(\mathbf{B}^{d})$
with $\alpha>(d-2s)\beta$.
From \eqref{8.01}, we have
\begin{eqnarray}
&&\mu(D(\gamma_{0}))=\mu(\{x\in \mathbf{B}^{d}:\lim\limits_{t\longrightarrow 0}
\rho_{\gamma(t)}(x)\neq\rho_{\gamma_{0}}(x)\})=0.\label{8.02}
\end{eqnarray}
From \eqref{8.02}, we have that \eqref{1.040} is valid.

By using \eqref{8.02} and Frostmans's lemmas (see for Theorem 8.8, page 112 of
\cite{M1995}) as well as Lemma 4.6, page 58 of \cite{M1995}, we have
\begin{eqnarray}
&&\mathcal{H}^{\alpha}(\{x\in \mathbf{B}^{d}:\lim\limits_{t\longrightarrow 0}
\rho_{\gamma(t)}(x)\neq\rho_{\gamma_{0}}(x)\})=0, \,\alpha>(d-2s)\beta,\label{8.03}
\end{eqnarray}
where $\mathcal{H}^{\alpha}$ is the $\alpha$-Hausdorff measure. From
\eqref{8.03}, we have that \eqref{1.041} is valid.

The proof of Theorem 1.4 is finished.

\section{Proof of Theorem  1.5: Strichartz estimates for orthonormal functions on $\mathbf{T}$}

\setcounter{equation}{0}

\setcounter{Theorem}{0}

\setcounter{Lemma}{0}

\setcounter{section}{9}

In this section, we present Schatten bound with space-time norms on $\mathbf{T}$ and
the Strichartz estimates for orthonormal functions on $\mathbf{T}$.

\subsection{Schatten bound with space-time norms on $\mathbf{T}$}
\begin{Lemma}\label{lem9.1} Let $N\geq1$, $I_{N}=[-\frac{1}{2N},
\frac{1}{2N}],\,S_{N}=\Z\cap[-N,N]$, $\frac{1}{p^{\prime}}+\frac{1}{2q^{\prime}}=1$,
$\frac{2}{3}<\frac{1}{q^{\prime}}<1$. Then, for all
$W_{1},W_{2}\in L_{t}^{2p^{\prime}}L_{x}^{2q^{\prime}}(\mathbf{T}^{2})$,
we have
\begin{eqnarray}
\|W_{1}\chi_{I_{N}}(t)\mathscr{D}_{N}\mathscr{D}_{N}^{\ast}
\chi_{I_{N}}(s)W_{2}\|_{\mathfrak{S}^{2q^{\prime}}(L^{2}(\mathbf{T}^{2}))}\leq
C\|W_{1}\|_{L_{t}^{2p^{\prime}}L_{x}^{2q^{\prime}}(\mathbf{T}^{2})}
\|W_{2}\|_{L_{t}^{2p^{\prime}}L_{x}^{2q^{\prime}}(\mathbf{T}^{2})},\label{9.01}
\end{eqnarray}
where  $\frac{1}{q}+\frac{1}{q^{\prime}}=\frac{1}{p}+\frac{1}{p^{\prime}}=1$,
$\mathbf{T}=[0,2\pi)$  and
$$\mathscr{D}_{N}f=e^{it\sqrt{\partial_{x}^{4}-\partial_{x}^{2}}}P_{\leq N}f
=\frac{1}{2\pi}\sum\limits_{n\in S_{N}}\mathscr{F}_{x}f(n)e^{i(xn+t\sqrt{n^{2}+n^{4}})}.$$

\end{Lemma}
\noindent{Proof.}
By using a direct calculation, we have
\begin{align}
\langle\mathscr{D}_{N}f, g\rangle_{L_{tx}^{2}(\mathbf{T}^{2})}&=
\int_{\mathbf{T}^{2}}e^{it\sqrt{\partial_{x}^{4}-\partial_{x}^{2}}}P_{\leq N}f\overline{g}dxdt
=\int_{\mathbf{T}}f(x)\overline{\int_{\mathbf{T}}e^{-it\sqrt{\partial_{x}^{4}
-\partial_{x}^{2}}}P_{\leq N}gdt}dx\nonumber\\
&=\left\langle f,\int_{\mathbf{T}}e^{-it\sqrt{\partial_{x}^{4}-\partial_{x}^{2}}}
P_{\leq N}gdt\right\rangle_{L_{x}^{2}(\mathbf{T})}=\langle f, \mathscr{D}_{N}^{\ast}
g\rangle_{L_{x}^{2}(\mathbf{T})}.\label{9.02}
\end{align}
From \eqref{9.02}, we have
\begin{eqnarray*}
&&\mathscr{D}_{N}^{\ast}g=\int_{\mathbf{T}}e^{it\sqrt{\partial_{x}^{4}-\partial_{x}^{2}}}P_{\leq N}gdt.
\end{eqnarray*}
We define
\begin{eqnarray*}
&&D_{N}(t,x)=\frac{1}{2\pi}\sum\limits_{k\in S_{N}}e^{ixk}e^{it\sqrt{k^{4}+k^{2}}}.
\end{eqnarray*}
Then, we have
\begin{eqnarray}
&&\mathscr{D}_{N}\mathscr{D}_{N}^{\ast}F(t,x)=\int_{\mathbf{T}^{2}}
D_{N}(t-t^{\prime},x-x^{\prime})F(t^{\prime},x^{\prime})dt^{\prime}dx^{\prime}.\label{9.03}
\end{eqnarray}
We also define
\begin{eqnarray}
&&D_{N,z}(t,x)=zt^{-1+z}\chi_{I_{N}}(t)D_{N}(t,x),\,D_{N,z,\epsilon}=
zt^{-1+z}\chi_{\{t||t|>\epsilon\}\cap I_{N}}(t)D_{N}(t,x),\nonumber\\
&&\mathscr{K}_{N,z}f=D_{N,z}\ast f,\, \mathscr{K}_{N, z,\epsilon}=
D_{N,z,\epsilon}\ast f, G_{N}(s,k)=e^{is\sqrt{k^{4}+k^{2}}}.\label{9.04}
\end{eqnarray}
From \eqref{9.03}, we have
\begin{eqnarray}
&&D_{N,z}(t,x)=zt^{-1+z}\chi_{I_{N}}(t)D_{N}(t,x)\nonumber\\&&=\lim_{\epsilon\rightarrow0}
D_{N,z,\epsilon}(t,x)=\lim_{\epsilon\rightarrow0}zt^{-1+z}\chi_{\{t||t|>\epsilon\}
\cap I_{N}}(t)D_{N}(t,x),\nonumber\\
&&\mathscr{K}_{N,z}f=\lim\limits_{\epsilon\rightarrow0}\mathscr{K}_{N,z,\epsilon}f
=\lim\limits_{\epsilon\rightarrow0}D_{N,z,\epsilon}\ast f.\label{9.05}
\end{eqnarray}
By using \eqref{9.05} and the Plancherel identity, we have
\begin{align}
\left\|\mathscr{K}_{N,z}f\right\|_{L_{x}^{2}(\mathbf{T})}&=
\left\|\lim\limits_{\epsilon\rightarrow0}D_{N,z,\epsilon}\ast f\right\|_{L_{x}^{2}(\mathbf{T})}
=\left(\sum\limits_{k\in S_{N}}\left|\lim\limits_{\epsilon\rightarrow0}
\mathscr{F}_{x}D_{N,z,\epsilon}(t,k)\ast \mathscr{F}_{x}f(k,s)\right|^{2}\right)^{\frac{1}{2}}\nonumber\\
&=\left(\sum\limits_{k\in S_{N}}\left|\lim\limits_{\epsilon\rightarrow0}
\int_{\mathbf{T}}e^{-it\sqrt{k^{4}+k^{2}}}zt^{-1+z}\chi_{\{t||t|>\epsilon\}
\cap I_{N}}(t) \mathscr{F}_{x}f(k,s-t)dt\right|^{2}\right)^{\frac{1}{2}}\nonumber\\
&=\left(\sum\limits_{k\in S_{N}}\left|\lim\limits_{\epsilon\rightarrow0}
\int_{\mathbf{T}}zt^{-1+z}\chi_{\{t||t|>\epsilon\}\cap I_{N}}(t)e^{i(s-t)\sqrt{k^{4}+k^{2}}} \mathscr{F}_{x}f(k,s-t)dt\right|^{2}\right)^{\frac{1}{2}}.\label{9.06}
\end{align}
From \eqref{9.06}, we get
\begin{eqnarray}
&&\left\|\mathscr{K}_{N,z}f\right\|_{L_{s}^{2}L_{x}^{2}(\mathbf{T}^{2})}\nonumber\\
&&=\left(\int_{\mathbf{T}}\sum\limits_{k\in S_{N}}\left|\lim\limits_{\epsilon\rightarrow0}
\int_{\mathbf{T}}zt^{-1+z}\chi_{\{t||t|>\epsilon\}\cap I_{N}}(t)e^{i(s-t)\sqrt{k^{4}+k^{2}}} \mathscr{F}_{x}f(k,s-t)|dt\right|^{2}ds\right)^{\frac{1}{2}}\nonumber\\
&&=\left(\sum\limits_{\tau\in \z}\sum\limits_{k\in S_{N}}\left|\mathscr{F}_{t}
(\lim\limits_{\epsilon\rightarrow0}\int_{\mathbf{T}}
zt^{-1+z}\chi_{\{t||t|>\epsilon\}\cap I_{N}}(t))\mathscr{F}_{t}G(\tau,k)\right|^{2}\right)^{\frac{1}{2}}.\label{9.07}
\end{eqnarray}
When $z=ib$, by using Lemma 2.6 , we have
\begin{eqnarray}
&\left|\mathscr{F}_{t}\left(i\lim\limits_{\epsilon\rightarrow0}bt^{-1+ib}\chi_{\{t||t|>\epsilon\}
\cap I_{N}}(t)\right)\right|\leq CC(b),\label{9.08}
\end{eqnarray}
where $C(b)=(e^{\pi b}+e^{-\pi b})\left(|b|+1\right)^{2}$.
It follows from \eqref{9.07}-\eqref{9.08} and the Plancherel identity that
\begin{align}
\left\|\mathscr{K}_{N,ib}f\right\|_{L_{s}^{2}L_{x}^{2}}&\leq C(b)
\left(\sum\limits_{\tau\in \z}\sum\limits_{k\in S_{N}}\left|\mathscr{F}_{t}G(\tau,k)\right|^{2}\right)^{\frac{1}{2}}
=C(b)\left(\int_{\mathbf{T}}\sum\limits_{k\in S_{N}}\left|G(s,k)\right|^{2}ds\right)^{\frac{1}{2}}\nonumber\\
&\leq C(b)\left(\int_{\mathbf{T}}\sum\limits_{k\in S_{N}}\left|\mathscr{F}_{x}f(s,k)\right|^{2}ds\right)^{\frac{1}{2}}
\leq C(b)\|f\|_{L_{s}^{2}L_{x}^{2}}.\label{9.09}
\end{align}
By using \eqref{9.09}, we have
\begin{eqnarray}
&\|W_{1}\chi_{I_{N}}(t)\mathscr{K}_{N,ib}\chi_{I_{N}}(s)W_{2}\|_{\mathfrak{S}^{\infty}}\leq C(b)\prod\limits_{j=1}^{2}\|W_{j}\|_{L_{tx}^{\infty}(\mathbf{T}^{2})}.\label{9.010}
\end{eqnarray}
By using Lemma 2.12 and Hardy-Littlewood-Sobolev inequality, for $Re z=\frac{1}{q-1}+1\in [0,\frac{3}{2}]$, we get
\begin{align}
&\left\|W_{1}\chi_{I_{N}}(t)\mathscr{K}_{N,z}\chi_{I_{N}}(s)W_{2}\right\|_{\mathfrak{S}^{2}}\nonumber\\
&=\left(\int_{I_{N}}\int_{I_{N}}\int_{\mathbf{T}}\int_{\mathbf{T}}|W_{1}(t,x)|^{2}
|D_{N,z}(t-s,x-y)|^{2}|W_{2}(s,y)|^{2}dxdydtds\right)^{\frac{1}{2}}\nonumber\\
&\leq C(1+Re z+|Imz|)\left(\int_{I_{N}}\int_{I_{N}}\int_{\mathbf{T}}\int_{\mathbf{T}}
\frac{|W_{1}(t,x)|^{2}|W_{2}(s,y)|^{2}}{|t-s|^{1-2Re z}}dxdydtds\right)^{\frac{1}{2}}\nonumber\\
&\leq C(1+Re z+|Imz|)\left(\int_{I_{N}}\int_{I_{N}}\frac{\|W_{1}(t,\cdot)\|_{L_{x}^{2}(\mathbf{T})}^{2}
\|W_{2}(s,\cdot)\|_{L_{x}^{2}(\mathbf{T})}^{2}}{|t-s|^{1-2Re z}}dtds\right)^{\frac{1}{2}}\nonumber\\
&\leq C(1+Re z+|Imz|)\prod_{j=1}^{2}\|W_{j}(t,\cdot)\|_{L_{t}^{2\tilde{u}}L_{x}^{2}(\mathbf{T}^{2})},\label{9.011}
\end{align}
where $0\leq 1-2Re z<1$,\,$\frac{2}{\tilde{u}}+1-2Re z=2$.

By  using Lemma 3.4, since
$\mathscr{K}_{N,1}=\mathscr{D}_{N}\mathscr{D}_{N}^{\ast}$, we obtain
\begin{align}
\left\|W_{1}\chi_{I_{N}}(t)\mathscr{K}_{N,1}\chi_{I_{N}}(s)W_{2}\right\|_{\mathfrak{S}^{2q^{\prime}}}
&=\left\|W_{1}\chi_{I_{N}}(t)\mathscr{D}_{N}\mathscr{D}_{N}^{\ast}\chi_{I_{N}}(s)W_{2}\right\|_{\mathfrak{S}^{2q^{\prime}}}\nonumber\\
&\leq C\prod\limits_{j=1}^{2}\|W_{j}\|_{L_{t}^{2p^{\prime}}L_{x}^{2q^{\prime}}(\mathbf{T}^{2})},\label{9.012}
\end{align}
where
$\frac{1}{p^{\prime}}+\frac{1}{2q^{\prime}}=1,\,\frac{2}{3}<\frac{1}{q^{\prime}}<1.
$

This completes the proof of Lemma 9.1.

\begin{Lemma}\label{lem9.2} Let $N\geq1,\,I_{N}=[-\frac{1}{2N},\frac{1}{2N}],
\,S_{N}=\Z\cap[-N,N]$,$(\frac{1}{q},\frac{1}{p})\in (A,B],\,A=(0,\frac{1}{2}),\,B=(1,0)$, $\lambda=(\lambda_{j})_{j=1}^{+\infty}\in\ell^{\beta}(\beta\leq\frac{2q}{q+1})$
and orthonormal system $(f_{j})_{j=1}^{+\infty}$ in $L^{2}(\mathbf{T})$.
Then,  we have
\begin{eqnarray}
&\left\|\sum\limits_{j=1}^{\infty}\lambda_{j}
\left|\mathscr{D}_{N}f_{j}\right|^{2}\right\|_{L_{t}^{p}L_{x}^{q}(I_{N}
\times\mathbf{T})}\leq C(q)\|\lambda\|_{\ell^{\alpha}},\label{9.013}
\end{eqnarray}
where $\mathbf{T}=[0,2\pi)$  and
\begin{eqnarray*}
&&\mathscr{D}_{N}f_{j}=e^{it\sqrt{\partial_{x}^{4}-\partial_{x}^{2}}}P_{\leq N}f_{j}
=\frac{1}{2\pi}\sum\limits_{n\in S_{N}}\mathscr{F}_{x}f(n)e^{i(xn+t\sqrt{n^{2}+n^{4}})}.
\end{eqnarray*}
\end{Lemma}

\begin{proof}

It follows from Lemmas 2.14,  9.1 that Lemma 9.2 is valid.

The proof of Lemma 9.2 is finished.

\end{proof}

\subsection{Proof of Theorem  1.5}

It follows from Lemma 9.2 and Theorem 1.5 of \cite{N2020} that Theorem 1.5 is valid.

The proof of Theorem 1.5 is finished.

\section{Proof of Theorem  1.6: Strichartz estimates for orthonormal functions on $\mathbf{T}$}

\setcounter{equation}{0}

\setcounter{Theorem}{0}

\setcounter{Lemma}{0}

\setcounter{section}{10}

In this section, we present the Strichartz estimates for orthonormal
functions on $\mathbf{T}$.

\noindent{\bf Proof.}
Before proving Theorem 1.6, we give the following estimates that
  play a key role in proving Theorem 1.6.
\begin{Lemma}\label{lem10.1}
Let $S_{N}=\Z\cap[-N,N]$,\, $\mathbf{T}=[0,2\pi)$. Then, for all
$W_{1}, W_{2}\in L_{t}^{4}L_{x}^{2}(\mathbf{T}\times\mathbf{T}),$
we have
\begin{eqnarray}
&\|W_{1}\mathscr{D}_{N}\mathscr{D}_{N}^{\ast}W_{2}\|_{\mathfrak{S}^{2}(L^{2}(\mathbf{T}^{2}))}\leq CN^{\frac{1}{2}}\|W_{1}\|_{L_{t}^{4}L_{x}^{2}(\mathbf{T}^{2})}\|W_{2}\|_{L_{t}^{4}L_{x}^{2}(\mathbf{T}^{2})},\label{10.01}
\end{eqnarray}
where
$$\mathscr{D}_{N}f=\frac{1}{2\pi}\sum\limits_{k\in S_{N}}\mathscr{F}_{x}f(k)e^{i(xk+t\sqrt{k^{2}+k^{4}})}.$$
\end{Lemma}
\begin{proof}
From \eqref{9.03}, we have
\begin{eqnarray}
\|W_{1}\mathscr{D}_{N}\mathscr{D}_{N}^{\ast}W_{2}\|_{\mathfrak{S}^{2}(L^{2}(\mathbf{T}^{2}))}^{2}
=\int_{\mathbf{T}^{2}}\int_{\mathbf{T}^{2}}|W_{1}(x,t)D_{N}(x-x^{\prime},t-t^{\prime})
W_{2}(x^{\prime},t^{\prime})|^{2}dtdxdt^{\prime}dx^{\prime},\label{10.02}
\end{eqnarray}
where $$D_{N}(t,x)=\frac{1}{2\pi}\sum\limits_{k\in S_{N}}e^{ixk}e^{it\sqrt{k^{4}+k^{2}}}.$$
Obviously, we have
\begin{eqnarray}
&&\left|D_{N}(x-x^{\prime},t-t^{\prime})\right|^{2}\nonumber\\
&&=\left|\sum\limits_{ k_{1},k_{2}=- N}^{N}e^{i(x-x^{\prime})(k_{1}-k_{2})}\left[\left(e^{i(t-t^{\prime})(\sqrt{k_{1}^{4}+k_{1}^{2}}-\sqrt{k_{2}^{4}+k_{2}^{2}})}
-e^{i(t-t^{\prime})(k_{1}^{2}+\frac{1}{2}-(k_{2}^{2}+\frac{1}{2}))}\right)+e^{i(t-t^{\prime})(k_{1}^{2}-k_{2}^{2})}\right]\right|\nonumber\\
&&\leq \left|\sum\limits_{ k_{1},k_{2}=- N}^{N}e^{i(x-x^{\prime})(k_{1}-k_{2})}\left(e^{i(t-t^{\prime})(\sqrt{k_{1}^{4}+k_{1}^{2}}-\sqrt{k_{2}^{4}+k_{2}^{2}})}
-e^{i(t-t^{\prime})(k_{1}^{2}+\frac{1}{2}-(k_{2}^{2}+\frac{1}{2}))}\right)\right|\nonumber\\
&&+\left|\sum\limits_{ k_{1},k_{2}=- N}^{N}e^{i(x-x^{\prime})(k_{1}-k_{2})}e^{i(t-t^{\prime})
(k_{1}^{2}-k_{2}^{2})}\right|
=E_{1}+E_{2}.\label{10.03}
\end{eqnarray}
For $E_{2}$, by using  a proof similar to Theorem 1.6 of \cite{N2020}, we have
\begin{eqnarray}
&&\left(\int_{\mathbf{T}^{2}}\int_{\mathbf{T}^{2}}|W_{1}(x,t)|^{2}|E_{2}(x-x^{\prime},t-t^{\prime})|
|W_{2}(x^{\prime},t^{\prime})|^{2}dtdxdt^{\prime}dx^{\prime}\right)^{\frac{1}{2}}\nonumber\\
&&\leq CN^{\frac{1}{2}}\|W_{1}\|_{L_{t}^{4}L_{x}^{2}(\mathbf{T}^{2})}\|W_{2}\|_{L_{t}^{4}L_{x}^{2}(\mathbf{T}^{2})}.\label{10.04}
\end{eqnarray}
For $E_{1}$, when $k_{1}=k_{2}$, by using H\"{o}lder inequality, we have
\begin{eqnarray}
&&\|W_{1}\mathscr{D}_{N}\mathscr{D}_{N}^{\ast}W_{2}\|_{\mathfrak{S}^{2}(L^{2}(\mathbf{T}^{2}))}^{2}
=\sum\limits_{ k_{1},k_{2}=- N}^{N}\int_{\mathbf{T}^{2}}\int_{\mathbf{T}^{2}}|W_{1}(x,t)
W_{2}(x^{\prime},t^{\prime})|^{2}dtdxdt^{\prime}dx^{\prime}\nonumber\\
&&\leq CN\|W_{1}\|_{L_{t}^{4}L_{x}^{2}(\mathbf{T}^{2})}^{2}\|W_{2}\|_{L_{t}^{4}L_{x}^{2}(\mathbf{T}^{2})}^{2}.\label{10.05}
\end{eqnarray}
When $k_{1}\neq k_{2}$, since
\begin{eqnarray*}
&&\left|e^{i(t-t^{\prime})[\sqrt{k_{1}^{4}+k_{1}^{2}}-(k_{1}^{2}+\frac{1}{2})
-(\sqrt{k_{2}^{4}+k_{2}^{2}}-(k_{2}^{2}+\frac{1}{2}))]}
-1\right|\nonumber\\&&\leq C|t|\left|\frac{1}{k_{1}^{2}+\frac{1}{2}+\sqrt{k_{1}^{4}+k_{1}^{2}}}-
\frac{1}{k_{2}^{2}+\frac{1}{2}+\sqrt{k_{2}^{4}+k_{2}^{2}}}\right|\nonumber\\
&&\leq C,
\end{eqnarray*}
we have
\begin{eqnarray}
&&|E_{1}|\nonumber\\
&&=\left|\sum\limits_{ k_{1},k_{2}=- N}^{N}e^{i(x-x^{\prime})(k_{1}-k_{2})}e^{i(t-t^{\prime})(k_{1}^{2}+\frac{1}{2}-(k_{2}^{2}+\frac{1}{2}))}
\left(e^{i(t-t^{\prime})[\sqrt{k_{1}^{4}+k_{1}^{2}}-(k_{1}^{2}+\frac{1}{2})-(\sqrt{k_{2}^{4}+k_{2}^{2}}-(k_{2}^{2}+\frac{1}{2}))]}
-1\right)\right|\nonumber\\
&&\leq CN.\label{10.06}
\end{eqnarray}
By using \eqref{10.06} and H\"{o}lder inequality, we have
\begin{eqnarray}
&&\int_{\mathbf{T}^{2}}\int_{\mathbf{T}^{2}}|W_{1}(x,t)|^{2}|E_{1}(x-x^{\prime},t-t^{\prime})|
|W_{2}(x^{\prime},t^{\prime})|^{2}dtdxdt^{\prime}dx^{\prime}\nonumber\\
&&\leq CN\int_{\mathbf{T}^{2}}\int_{\mathbf{T}^{2}}|W_{1}(x,t)|^{2}
|W_{2}(x^{\prime},t^{\prime})|^{2}dtdxdt^{\prime}dx^{\prime}\nonumber\\
&&\leq CN\|W_{1}\|_{L_{t}^{4}L_{x}^{2}(\mathbf{T}^{2})}^{2}\|W_{2}\|_{L_{t}^{4}L_{x}^{2}(\mathbf{T}^{2})}^{2}.\label{10.07}
\end{eqnarray}
From \eqref{10.03}-\eqref{10.07}, we have
\begin{eqnarray}
&&\left(\int_{\mathbf{T}^{2}}\int_{\mathbf{T}^{2}}|W_{1}(x,t)D_{N}(x-x^{\prime},t-t^{\prime})
W_{2}(x^{\prime},t^{\prime})|^{2}dtdxdt^{\prime}dx^{\prime}\right)^{\frac{1}{2}}\nonumber\\
&&\leq CN^{\frac{1}{2}}\|W_{1}\|_{L_{t}^{4}L_{x}^{2}(\mathbf{T}^{2})}\|W_{2}\|_{L_{t}^{4}L_{x}^{2}(\mathbf{T}^{2})}.\label{10.08}
\end{eqnarray}

This completes the proof of Lemma 10.1.

\end{proof}

\noindent{\bf Proof of Theorem  1.6}.
By using Lemma 2.14 and Lemma 10.1, we have that Theorem 1.6 is valid.

The proof of Theorem 1.6 is finished.

\section{Proof of Theorem  1.7:  Optimality of Theorems 1.5, 1.6}

\setcounter{equation}{0}

\setcounter{Theorem}{0}

\setcounter{Lemma}{0}

\setcounter{section}{11}

In this section, we present the optimality of Theorems 1.5, 1.6.

\noindent{\bf Proof.}
Inspired by \cite{N2020}. Firstly, we suppose that \eqref{1.042} is valid. Then, we define $f_{j}=e^{ixj}$, $j\in \mathbb{N}^{+},\,x\in\mathbf{T}$, by simple calculation, we have
\begin{eqnarray}
&&\mathscr{F}_{x}f_{j}=\chi_{j}(k),\,k\in\mathbf{Z}.\label{11.01}
\end{eqnarray}
From \eqref{11.01}, we have
\begin{align}
\mathscr{D}_{N}f_{j}&=\frac{1}{2\pi}\chi_{S_{N}}(j)e^{it\sqrt{k^{2}+k^{4}}}e^{ixj}=\frac{1}{2\pi}\chi_{S_{N}}(j)e^{it\sqrt{k^{2}+k^{4}}}f_{j}.\label{11.02}
\end{align}
By using \eqref{11.02}, we have
\begin{eqnarray}
&&|\mathscr{D}_{N}f_{j}|=\frac{1}{2\pi},\,\lambda_{j}=\frac{1}{2\pi}\chi_{S_{N}}(j)e^{it\sqrt{k^{2}+k^{4}}},\label{11.03}
\end{eqnarray}
where $\lambda_{j}$ is the eigenvalue of operator $\mathscr{D}_{N}$.
On the one hand, by using \eqref{11.03}, we have
\begin{align}
\left\|\sum_{j\in\mathbb{N}^{+}}\lambda_{j}|U_{N}f_{j}|^{2}\right\|_{L_{t}^{q}L_{x}^{p}}
&=\frac{1}{2\pi}\left\|\sum_{j\in\mathbb{N}^{+}}\chi_{S_{N}}(j)\right\|_{L_{t}^{q}L_{x}^{p}}\sim N.\label{11.04}
\end{align}
On the other hand, we have
\begin{eqnarray}
&&N^{\frac{1}{p}}\|\lambda\|_{\ell^{\beta}}\sim N^{\frac{1}{p}}N^{\frac{1}{\beta}}.\label{11.05}
\end{eqnarray}
By using \eqref{1.042}, \eqref{11.04} and \eqref{11.05}, we have
\begin{eqnarray}
&&N\leq CN^{\frac{1}{p}}N^{\frac{1}{\beta}}.\label{11.06}
\end{eqnarray}
When $N$ tends to $\infty$, from \eqref{11.06},  since $\frac{2}{p}+\frac{1}{q}=1$, we have $\beta\leq \frac{2q}{q+1}$, this contradicts $\beta> \frac{2q}{q+1}$.

Finally, we suppose that \eqref{1.043} is valid. Then, we define $f_{j}=e^{ixj}$, $j\in \mathbb{N}^{+},\,x\in\mathbf{T}$. By using a proof similar to \eqref{11.06}, we can get
\begin{eqnarray}
&&N\leq CN^{\frac{1}{2}}N^{\frac{1}{\beta}}.\label{11.07}
\end{eqnarray}
When $N$ tends to $\infty$, from \eqref{11.07},  we have $\beta\leq 2$, this contradicts $\beta> 2$.

This completes the proof of Theorem 1.7.

\section{Proof of Theorem  1.8: Stochastic continuity at zero related to Schatten norm on $\R$}

\setcounter{equation}{0}

\setcounter{Theorem}{0}

\setcounter{Lemma}{0}

\setcounter{section}{12}

In this section, we present the stochastic continuity at zero related to Schatten norm on $\R$.
Before proving Theorem 1.8, we give the following estimates that  play a key role in proving Theorem 1.8.

\begin{Lemma}\label{lem12.1}
Let $(\lambda_{j})_{j}\in \ell_{j}^{2}$, $(f_{j})_{j=1}^{\infty}$ is an orthonormal system in $L^{2}\left(\R\right)$ and $f_{j}^{\omega}$ the randomization of $f_{j}$ be define as in \eqref{1.026}. Then, for $\forall\epsilon>0, \exists\delta>0$, when $|t|<\delta<\frac{\epsilon^{2}}{M_{1}\sqrt{1+M_{1}^{2}}}$, we have
\begin{align}
&&\left\|\sum\limits_{j=1}^{\infty} \lambda_{j} g_{j}^{(2)}(\widetilde{\omega}) I_{1}\right\|_{L_{\omega, \tilde{\omega}}^{r}(\Omega \times \tilde{\Omega})}
\leq Cp^{\frac{1}{2}}\epsilon^{2}\left(\|\gamma_{0}\|_{\mathfrak{S}^{2}}+1\right),\label{12.01}
\end{align}
where $I_{1}=\left(e^{it\sqrt{\partial_{x}^{4}-\partial_{x}^{2}}} f_{j}^{\omega}-f_{j}^{\omega}\right) e^{-it\sqrt{\partial_{x}^{4}-\partial_{x}^{2}}}\overline{f}_{j}^{\omega}$,  $\|\gamma_{0}\|_{\mathfrak{S}^{2}}=\left(\sum\limits_{j=1}^{\infty}\lambda_{j}^{2}\right)^{\frac{1}{2}}$  and $M_{1}$ appears in \eqref{4.010}.
\end{Lemma}

\begin{proof}
By using Minkowski inequality, H\"{o}lder inequality as well as Lemmas 4.1, 4.3, 4.6,
 $\forall \epsilon>0$, which appears in Lemma 4.2, $\exists \delta$  such that $0<\delta<\frac{\epsilon^{2}}{M_{1}\sqrt{1+M_{1}^{2}}}$, when $|t|<\delta$, we have
\begin{align}
\left\|\sum_{j=1}^{\infty} \lambda_{j} g_{j}^{(2)}(\widetilde{\omega}) I_{1}\right\|_{L_{\omega, \tilde{\omega}}^{r}(\Omega \times \tilde{\Omega})}
&\leq C r^{\frac{1}{2}}\left\|\lambda_{j} I_{1}\right\|_{L_\omega^{r} \ell_{j}^{2}}
\leq C r^{\frac{1}{2}}\left\|\lambda_{j} I_{1}\right\|_{\ell_{j}^{2} L_{\omega}^{r}}\nonumber\\
&\leq C r^{\frac{1}{2}}\left\|\lambda_{j}\left\|\left(e^{it\sqrt{\partial_{x}^{4}-\partial_{x}^{2}}} f_{j}^{\omega}-f_{j}^{\omega}\right)\right\|_{L_{\omega}^{2r}}\left\| e^{-it\sqrt{\partial_{x}^{4}-\partial_{x}^{2}}}\overline{f}_{j}^{\omega}\right\|_{L_{\omega}^{2r}}\right\|_{\ell_{j}^{2}}\nonumber\\
&\leq C r^{\frac{3}{2}} \epsilon^{2}\left(\left\|\gamma_{0}\right\|_{\mathfrak{S}^2}+1\right).\label{12.02}
\end{align}

This completes the proof of Lemma 12.1.
\end{proof}

\begin{Lemma}\label{lem12.2}
Let $(\lambda_{j})_{j}\in \ell_{j}^{2}$, $(f_{j})_{j=1}^{\infty}$
be an orthonormal system in $L^{2}\left(\R\right)$ and $f_{j}^{\omega}$
the randomization of $f_{j}$ be define as in \eqref{1.026}. Then, for
 $\forall\epsilon>0, \exists\delta>0$, when $|t|<\delta<\frac{\epsilon^{2}}{M_{1}\sqrt{1+M_{1}^{2}}}$, we have
\begin{align}
&&\left\|\sum\limits_{n=1}^{\infty} \lambda_{j} g_{j}^{(2)}(\widetilde{\omega})
I_{2}\right\|_{L_{\omega, \tilde{\omega}}^{r}(\Omega \times \tilde{\Omega})}
\leq Cp^{\frac{1}{2}}\epsilon^{2}\left(\|\gamma_{0}\|_{\mathfrak{S}^{2}}+1\right),\label{12.03}
\end{align}
where $I_{2}=f_{j}^{\omega}\left(e^{it\sqrt{\partial_{x}^{4}-\partial_{x}^{2}}}
\overline{f}_{j}^{\omega}-\overline{f}_{j}^{\omega}\right)$,
$\|\gamma_{0}\|_{\mathfrak{S}^{2}}=\left(\sum\limits_{j=1}^{\infty}
\lambda_{j}^{2}\right)^{\frac{1}{2}}$  and $ M_{1}$  appears in \eqref{4.010}.
\end{Lemma}

\begin{proof}
By using Minkowski inequality, H\"{o}lder inequality as well as Lemmas 4.1,
4.3, 4.6, $\forall \epsilon>0$, which appears in Lemma 4.2,
$\exists\delta(0<\delta<\frac{\epsilon^{2}}{M_{1}\sqrt{1+M_{1}^{2}}})$,
 when $|t|<\delta$, we have
\begin{align}
\left\|\sum_{j=1}^{\infty} \lambda_{j} g_{j}^{(2)}(\widetilde{\omega})
I_{2}\right\|_{L_{\omega, \tilde{\omega}}^{r}(\Omega \times \tilde{\Omega})}
&\leq C r^{\frac{1}{2}}\left\|\lambda_{j} I_{2}\right\|_{L_{\omega}^{r} \ell_{j}^{2}}
\leq C r^{\frac{1}{2}}\left\|\lambda_{j} I_{2}\right\|_{\ell_{j}^{2} L_{\omega}^{r}}\nonumber\\
&\leq C r^{\frac{1}{2}}\left\|\lambda_{j}\left\|\left(e^{-it\sqrt{\partial_{x}^{4}-\partial_{x}^{2}}}
\overline{f}_{j}^{\omega}-\overline{f}_{j}^{\omega}\right)
\right\|_{L_{\omega}^{2r}}\left\| f_{j}^{\omega}\right\|_{L_{\omega}^{2r}}\right\|_{\ell_{j}^{2}}\nonumber\\
&\leq C r^{\frac{3}{2}} \epsilon^{2}
\left(\left\|\gamma_{0}\right\|_{\mathfrak{S}^2}+1\right).\label{12.04}
\end{align}

This completes the proof of Lemma 12.2.
\end{proof}

\noindent{\bf Proof of Theorem  1.8}. Now, we use Lemmas 12.1,
 12.2 to prove Theorem 1.8. Obviously, we have
\begin{eqnarray}
&&\left|e^{it\sqrt{\partial_{x}^{4}-\partial_{x}^{2}}} f_{j}^{\omega}\right|^2
-\left|f_{j}^{\omega}\right|^2=I_{1}+I_{2}.\label{12.05}
\end{eqnarray}
By using \eqref{12.05} and  Lemmas 12.1, 12.2, we have
\begin{eqnarray}
&&\left\|\sum_{j=1}^{\infty} \lambda_{j} g_{j}^{(2)}(\widetilde{\omega})
\left|f_{j}^{\omega}\right|^{2}-\sum_{j=1}^{\infty} \lambda_{j}
g_{j}^{(2)}(\widetilde{\omega})\left|e^{it\sqrt{\partial_{x}^{4}
-\partial_{x}^{2}}}f_{j}^{\omega}\right|^2\right\|_{L_{\omega,
\tilde{\omega}}^{r}(\Omega \times \widetilde{\Omega})}\nonumber\\
&&=\left\|\sum_{j=1}^{\infty} \lambda_{j} g_{j}^{(2)}(\widetilde{\omega})
\left(I_{1}+I_{2}\right)\right\|_{L_{\omega, \tilde{\omega}}^{r}
(\Omega \times \widetilde{\Omega})}\nonumber\\
&&\leq\left\|\sum_{j=1}^{\infty} \lambda_{j} g_{j}^{(2)}(\widetilde{\omega})
 I_{1}\right\|_{L_{\omega, \tilde{\omega}}^{r}(\Omega \times \tilde{\Omega})}
 +\left\|\sum_{j=1}^{\infty} \lambda_{j} g_{j}^{(2)}(\widetilde{\omega})
 I_{2}\right\|_{L_{\omega, \tilde{\omega}}^{r}(\Omega \times \tilde{\Omega})}\nonumber\\
&&\leq C r^{\frac{3}{2}} \epsilon^{2}
\left(\left\|\gamma_{0}\right\|_{\mathfrak{S}^2}+1\right).\label{12.06}
\end{eqnarray}
Combining \eqref{12.06} with Lemma 4.2, we have that \eqref{1.022} is valid.

This completes the proof of Theorem 1.8.

\section{Proof of Theorem  1.9: Stochastic continuity at zero
 related to Schatten norm on $\mathbf{T}$}

\setcounter{equation}{0}

\setcounter{Theorem}{0}

\setcounter{Lemma}{0}

\setcounter{section}{13}

In this section, we present the stochastic continuity at
 zero related to Schatten norm on $\mathbf{T}$.
Before proving Theorem 1.9, we give the following estimates
 that  play a key role in proving Theorem 1.9.

\begin{Lemma}\label{lem13.1}
Let $\|f_{j}\|_{L^{2}(\mathbf{T})}=\left(\sum\limits_{k\in\z}
|\mathscr{F}_{x}f_{j}(k)|^{2}\right)^{\frac{1}{2}}=1$,
$j\in \mathbb{N}^{+}$,  $(\lambda_{j})_{j}\in \ell_{j}^{2}$, and $f_{j}^{\omega}$ be
the randomization of $f_{j}$  defined as in \eqref{1.027}. Then,
 for $\forall\epsilon>0, \exists\delta>0$,
 when $|t|<\delta<\frac{\epsilon^{2}}{M_{1}\sqrt{1+M_{1}^{2}}}$, we have
\begin{align}
&&\left\|\sum\limits_{j=1}^{\infty}\lambda_{j}g_{j}^{(2)}
(\widetilde{\omega})H_{1}\right\|_{L_{\omega, \tilde{\omega}}^{r}(\Omega \times \tilde{\Omega})}
\leq Cr^{\frac{3}{2}}\epsilon^{2}(\|\gamma_{0}\|_{\mathfrak{S}^{2}}+1),\label{13.01}
\end{align}
where $H_{1}=\left(e^{it\sqrt{\partial_{x}^{4}-\partial_{x}^{2}}} f_{j}^{\omega}-f_{j}^{\omega}\right) e^{-it\sqrt{\partial_{x}^{4}-\partial_{x}^{2}}}\overline{f}_{j}^{\omega}$, $\|\gamma_{0}\|_{\mathfrak{S}^{2}}=
\left(\sum\limits_{j=1}^{\infty}\lambda_{j}^{2}\right)^{\frac{1}{2}}$ and $M_{1}$  appears in \eqref{4.017}.
\end{Lemma}

\begin{proof}
By using Minkowski inequality, H\"{o}lder inequality and Lemmas 4.1,
4.4, 4.7, $\forall \epsilon>0$, which appears in Lemma 4.2,
$\exists\delta$ such that $0<\delta<\frac{\epsilon^{2}}{M_{1}\sqrt{1+M_{1}^{2}}}$, when $|t|<\delta$, we have
\begin{align}
\left\|\sum_{j=1}^{\infty} \lambda_{j} g_{j}^{(2)}(\widetilde{\omega})
 H_{1}\right\|_{L_{\omega, \tilde{\omega}}^{r}(\Omega \times \tilde{\Omega})}
&\leq C r^{\frac{1}{2}}\left\|\lambda_{j} H_{1}\right\|_{L_\omega^{r}
\ell_{j}^{2}}\leq C r^{\frac{1}{2}}\left\|\lambda_{j} H_{1}\right\|_{\ell_{j}^{2} L_{\omega}^{r}}\nonumber\\
&\leq C r^{\frac{1}{2}}\left\|\lambda_{j}\left\|\left(e^{it\sqrt{\partial_{x}^{4}-\partial_{x}^{2}}} f_{j}^{\omega}-f_{j}^{\omega}\right)\right\|_{L_{\omega}^{2r}}\left\| e^{-it\sqrt{\partial_{x}^{4}-\partial_{x}^{2}}}\overline{f}_{j}^{\omega}\right\|_{L_{\omega}^{2r}}\right\|_{\ell_{j}^{2}}\nonumber\\
&\leq C r^{\frac{3}{2}} \epsilon^{2}\left(\left\|\gamma_{0}\right\|_{\mathfrak{S}^2}+1\right).\label{13.02}
\end{align}

This completes the proof of Lemma 13.1.

\end{proof}

\begin{Lemma}\label{lem13.2}
Let $\|f_{j}\|_{L^{2}(\mathbf{T})}=\left(\sum\limits_{k\in\z}
|\mathscr{F}_{x}f_{j}(k)|^{2}\right)^{\frac{1}{2}}=1$, $j\in \mathbb{N}^{+}$,
  $(\lambda_{j})_{j}\in \ell_{j}^{2}$,and $f_{j}^{\omega}$ the
   randomization of $f_{j}$ be defined as in \eqref{1.027}. Then,
   for $\forall\epsilon>0, \exists\delta>0$, when $|t|<\delta<\frac{\epsilon^{2}}{M_{1}\sqrt{1+M_{1}^{2}}}$, we have
\begin{align}
&&\left\|\sum\limits_{j=1}^{\infty}\lambda_{j}g_{j}^{(2)}(\widetilde{\omega})
H_{2}\right\|_{L_{\omega, \tilde{\omega}}^{r}(\Omega \times \tilde{\Omega})}
\leq Cr^{\frac{3}{2}}\epsilon^{2}(\|\gamma_{0}\|_{\mathfrak{S}^{2}}+1),\label{13.03}
\end{align}
where $H_{2}=f_{j}^{\omega}\left(e^{-it\sqrt{\partial_{x}^{4}-\partial_{x}^{2}}}
\overline{f}_{j}^{\omega}-\overline{f}_{n}^{\omega}\right)$, $\|\gamma_{0}\|_{\mathfrak{S}^{2}}=
\left(\sum\limits_{n=1}^{\infty}\lambda_{j}^{2}\right)^{\frac{1}{2}}$ and $M_{1}$ appears in \eqref{4.017}.
\end{Lemma}

\begin{proof}
By using Minkowski inequality,  H\"{o}lder inequality and Lemmas 4.1,
4.4, 4.7, $\forall \epsilon>0$, which appears in Lemma 4.2, $\exists\delta$
 such that $0<\delta<\frac{\epsilon^{2}}{M_{1}\sqrt{1+M_{1}^{2}}}$, when $|t|<\delta$, we have
\begin{align}
\left\|\sum_{j=1}^{\infty} \lambda_{n} g_{j}^{(2)}(\widetilde{\omega})
 H_{2}\right\|_{L_{\omega, \tilde{\omega}}^{r}(\Omega \times \tilde{\Omega})}
&\leq C r^{\frac{1}{2}}\left\|\lambda_{j} H_{2}\right\|_{L_{\omega}^{r}
\ell_{j}^{2}}\leq C r^{\frac{1}{2}}\left\|\lambda_{j} H_{2}\right\|_{\ell_{j}^{2} L_{\omega}^{r}}\nonumber\\
&+C r^{\frac{1}{2}}\left\|\lambda_{j}\left\|\left(e^{-it\sqrt{\partial_{x}^{4}
-\partial_{x}^{2}}}\overline{f}_{j}^{\omega}-\overline{f}_{j}^{\omega}\right)
\right\|_{L_{\omega}^{2r}}\left\| f_{j}^{\omega}\right\|_{L_{\omega}^{2r}}\right\|_{\ell_{j}^{2}}\nonumber\\
&\leq C r^{\frac{3}{2}} \epsilon^{2}\left(\left\|\gamma_{0}\right\|_{\mathfrak{S}^2}+1\right).\label{13.04}
\end{align}

This completes the proof of Lemma 13.2.

\end{proof}

\noindent{\bf Proof of Theorem  1.9}. Now, we use Lemmas 13.1,  13.2
to prove Theorem 1.9. Obviously, we have
\begin{eqnarray}
&&\left|e^{it\sqrt{\partial_{x}^{4}-\partial_{x}^{2}}} f_{j}^{\omega}
\right|^2-\left|f_{j}^{\omega}\right|^2=H_{1}+H_{2}.\label{13.05}
\end{eqnarray}
By using \eqref{13.05} and Lemmas 13.1, 13.2, we obtain
\begin{eqnarray}
&&\left\|\sum_{j=1}^{\infty} \lambda_{j} g_{j}^{(2)}(\widetilde{\omega})
\left|f_{j}^{\omega}\right|^{2}-\sum_{j=1}^{\infty} \lambda_{j} g_{j}^{(2)}
(\widetilde{\omega})\left|e^{it\sqrt{\partial_{x}^{4}-
\partial_{x}^{2}}}f_{j}^{\omega}\right|^2\right\|_{L_{\omega, \tilde{\omega}}^{r}
(\Omega \times \widetilde{\Omega})}\nonumber\\
&&=\left\|\sum_{j=1}^{\infty} \lambda_{j} g_{j}^{(2)}(\widetilde{\omega})
\left(H_{1}+H_{2}\right)\right\|_{L_{\omega, \tilde{\omega}}^{r}(\Omega \times
 \widetilde{\Omega})}\nonumber\\
&&\leq\left\|\sum_{j=1}^{\infty} \lambda_{j} g_{j}^{(2)}(\widetilde{\omega})
 H_{1}\right\|_{L_{\omega, \tilde{\omega}}^{r}(\Omega \times \tilde{\Omega})}
 +\left\|\sum_{j=1}^{\infty} \lambda_{j} g_{j}^{(2)}(\widetilde{\omega})
  H_{2}\right\|_{L_{\omega, \tilde{\omega}}^{r}(\Omega \times \tilde{\Omega})}\nonumber\\
&&\leq C r^{\frac{3}{2}} \epsilon^{2}
\left(\left\|\gamma_{0}\right\|_{\mathfrak{S}^2}+1\right).\label{13.06}
\end{eqnarray}
Combining \eqref{13.06} with Lemma 4.2, we have that \eqref{1.024} is valid.

This completes the proof of Theorem 1.9.

\section {Proof of Theorem  1.10: Stochastic continuity at zero
 related to Schatten norm on $\Theta=\{x\in\R^{3}:|x|<1\}$}

\setcounter{equation}{0}

\setcounter{Theorem}{0}

\setcounter{Lemma}{0}

\setcounter{section}{14}

In this section, we present the stochastic continuity at zero
 related to Schatten norm on $\Theta=\{x\in\R^{3}:|x|<1\}$.
Before proving Theorem 1.10, we give the following estimates that
  play a key role in proving Theorem 1.10.

\begin{Lemma}\label{lem14.1}
Let $f_{j}=\sum\limits_{m=1}^{\infty}c_{j,m}e_{m}$, $f_{j}^{\omega_{1}}$
 be the randomization of $f_{j}$  defined as in \eqref{1.028}. Then, for
 $\forall\epsilon>0, \exists\delta>0$, when $|t|<\delta<\frac{\epsilon^{2}}{M_{1}\sqrt{1+M_{1}^{2}}}$, we have
\begin{align}
&&\left\|\sum\limits_{j=1}^{\infty}\lambda_{j}g_{j}^{(2)}(\widetilde{\omega})
J_{1}\right\|_{L_{\omega_{1}, \tilde{\omega}}^{r}(\Omega_{1} \times \tilde{\Omega})}
\leq Cr^{\frac{3}{2}}\epsilon^{2}(\|\gamma_{0}\|_{\mathfrak{S}^{2}}+1),\label{14.01}
\end{align}
where $J_{1}=\left(e^{it\sqrt{\partial_{x}^{4}-\partial_{x}^{2}}}
 f_{j}^{\omega_{1}}-f_{j}^{\omega_{1}}\right) e^{-it\sqrt{\partial_{x}^{4}
 -\partial_{x}^{2}}}\overline{f}_{j}^{\omega_{1}}$, $\|\gamma_{0}\|_{\mathfrak{S}^{2}}=
 \left(\sum\limits_{j=1}^{\infty}\lambda_{j}^{2}\right)^{\frac{1}{2}}$ and $M_{1}$ appears in \eqref{4.024}.
\end{Lemma}

\begin{proof}
By using Minkowski inequality, H\"{o}lder inequality and Lemmas 4.1, 4.5,
 4.8, $\forall \epsilon>0$, which appears in Lemma 4.2, $\exists\delta$
  such that $0<\delta<\frac{\epsilon^{2}}{M_{1}\sqrt{1+M_{1}^{2}}}$, when $|t|<\delta$, we have
\begin{align}
\left\|\sum_{j=1}^{\infty} \lambda_{j} g_{j}^{(2)}(\widetilde{\omega})
 J_{1}\right\|_{L_{\omega_{1}, \tilde{\omega}}^{r}(\Omega_{1} \times \tilde{\Omega})}
&\leq C r^{\frac{1}{2}}\left\|\lambda_{j} J_{1}\right\|_{L_{\omega_{1}}^{r} \ell_{j}^{2}}
\leq C r^{\frac{1}{2}}\left\|\lambda_{j} J_{1}\right\|_{\ell_{j}^{2} L_{\omega_{1}}^{r}}\nonumber\\
&\leq C r^{\frac{1}{2}}\left\|\lambda_{j}\left\|\left(e^{it\sqrt{\partial_{x}^{4}-\partial_{x}^{2}}} f_{j}^{\omega_{1}}-f_{j}^{\omega_{1}}\right)\right\|_{L_{\omega_{1}}^{2r}}
\left\| e^{-it\sqrt{\partial_{x}^{4}-\partial_{x}^{2}}}
\overline{f}_{j}^{\omega_{1}}\right\|_{L_{\omega_{1}}^{2r}}\right\|_{\ell_{j}^{2}}\nonumber\\
&\leq C r^{\frac{3}{2}} \epsilon^{2}
\left(\left\|\gamma_{0}\right\|_{\mathfrak{S}^2}+1\right).\label{14.02}
\end{align}

This completes the proof of Lemma 14.1.

\end{proof}

\begin{Lemma}\label{lem14.2}
Let $f_{j}=\sum\limits_{m=1}^{\infty}c_{j,m}e_{m}$, $f_{j}^{\omega_{1}}$
be the randomization of $f_{j}$  defined as in \eqref{1.028}. Then, for
$\forall\epsilon>0, \exists\delta>0$, when $|t|<\delta<\frac{\epsilon^{2}}{M_{1}\sqrt{1+M_{1}^{2}}}$, we have
\begin{align}
&&\left\|\sum\limits_{j=1}^{\infty}\lambda_{n}g_{j}^{(2)}
(\widetilde{\omega})J_{2}\right\|_{L_{\omega_{1}, \tilde{\omega}}^{r}(\Omega_{1} \times \tilde{\Omega})}
\leq Cr^{\frac{3}{2}}\epsilon^{2}(\|\gamma_{0}\|_{\mathfrak{S}^{2}}+1),\label{14.03}
\end{align}
where $J_{2}=f_{j}^{\omega_{1}}\left(e^{it\sqrt{\partial_{x}^{4}-\partial_{x}^{2}}}
\overline{f}_{j}^{\omega_{1}}-\overline{f}_{j}^{\omega_{1}}\right)$,
$\|\gamma_{0}\|_{\mathfrak{S}^{2}}=\left(\sum\limits_{j=1}^{\infty}
\lambda_{j}^{2}\right)^{\frac{1}{2}}$ and $M_{1}$ appears in \eqref{4.024}.
\end{Lemma}

\begin{proof}
By using Minkowski inequality, H\"{o}lder inequality and Lemmas 4.1, 4.5,
4.8, $\forall \epsilon>0$, which appears in Lemma 4.2,
$\exists \delta(0<\delta<\frac{\epsilon^{2}}{M_{1}\sqrt{1+M_{1}^{2}}})$,
 when $|t|<\delta$, we have
\begin{align}
\left\|\sum_{j=1}^{\infty} \lambda_{j} g_{n}^{(2)}(\widetilde{\omega})
 J_{2}\right\|_{L_{\omega_{1}, \tilde{\omega}}^{r}(\Omega_{1} \times \tilde{\Omega})}
&\leq C r^{\frac{1}{2}}\left\|\lambda_{j} J_{2}\right\|_{L_{\omega_{1}}^{r} \ell_{j}^{2}}
\leq C r^{\frac{1}{2}}\left\|\lambda_{j} J_{2}\right\|_{\ell_{j}^{2} L_{\omega_{1}}^{r}}\nonumber\\
&+C r^{\frac{1}{2}}\left\|\lambda_{j}\left\|\left(e^{-it\sqrt{\partial_{x}^{4}-\partial_{x}^{2}}}
\overline{f}_{j}^{\omega_{1}}-\overline{f}_{j}^{\omega_{1}}\right)
\right\|_{L_{\omega_{1}}^{2r}}\left\| f_{j}^{\omega_{1}}\right\|_{L_{\omega_{1}}^{2r}}\right\|_{\ell_{j}^{2}}\nonumber\\
&\leq C r^{\frac{3}{2}} \epsilon^{2}\left(\left\|\gamma_{0}\right\|_{\mathfrak{S}^2}+1\right).\label{14.04}
\end{align}

The proof of Lemma 14.2 is finished.

\end{proof}

\noindent{\bf Proof of Theorem  1.10}. Now, we use Lemmas 14.1,  14.2 to
 prove Theorem 1.10. Obviously, we have
\begin{eqnarray}
&&\left|e^{it\sqrt{\partial_{x}^{4}-\partial_{x}^{2}}}
f_{j}^{\omega_{1}}\right|^2-\left|f_{j}^{\omega_{1}}\right|^2=J_{1}+J_{2}.\label{14.05}
\end{eqnarray}
By using \eqref{14.05} and Lemmas 14.1, 14.2, we get
\begin{eqnarray}
&&\left\|\sum_{j=1}^{\infty} \lambda_{j} g_{j}^{(2)}(\widetilde{\omega})
\left|f_{j}^{\omega_{1}}\right|^{2}-\sum_{j=1}^{\infty} \lambda_{j}
g_{j}^{(2)}(\widetilde{\omega})\left|e^{it\sqrt{\partial_{x}^{4}-
\partial_{x}^{2}}}f_{j}^{\omega_{1}}\right|^2\right\|_{L_{\omega_{1},
\tilde{\omega}}^{r}(\Omega_{1} \times \widetilde{\Omega})}\nonumber\\
&&=\left\|\sum_{j=1}^{\infty} \lambda_{n} g_{j}^{(2)}(\widetilde{\omega})
\left(J_{1}+J_{2}\right)\right\|_{L_{\omega_{1}, \tilde{\omega}}^{r}(\Omega_{1}
\times \widetilde{\Omega})}\nonumber\\
&&\leq\left\|\sum_{j=1}^{\infty} \lambda_{j} g_{j}^{(2)}(\widetilde{\omega})
 J_{1}\right\|_{L_{\omega_{1}, \tilde{\omega}}^{r}(\Omega_{1} \times \tilde{\Omega})}+\left\|\sum_{j=1}^{\infty} \lambda_{j} g_{j}^{(2)}(\widetilde{\omega}) J_{2}\right\|_{L_{\omega_{1}, \tilde{\omega}}^{r}(\Omega_{1} \times \tilde{\Omega})}\nonumber\\
&&\leq C r^{\frac{3}{2}} \epsilon^{2}\left(\left\|\gamma_{0}\right\|_{\mathfrak{S}^2}+1\right).\label{14.06}
\end{eqnarray}
Combining \eqref{14.06} with Lemma 4.2, we have that \eqref{1.026} is valid.

The proof of Theorem 1.10 is finished.

\bigskip
\bigskip

\leftline{\large \bf Acknowledgments}

  \bigskip

  \bigskip

\baselineskip=18pt


\end{document}